\documentclass[11pt,letterpaper]{amsart}

\usepackage[left=3.5cm,right=3.5cm,top=2.2cm,bottom=2.2cm]{geometry}

\usepackage{longtable}

\usepackage{amsmath,amssymb,mathrsfs,stmaryrd}
\usepackage{amsthm}
\usepackage{mathtools,bm}
\usepackage[normalem]{ulem}

\usepackage[LGR,T1]{fontenc} 
\usepackage[utf8]{inputenc} %
\usepackage{lmodern}
\usepackage[english]{babel}
\usepackage{cancel}
\usepackage{dsfont}
\usepackage[colorlinks,linkcolor=NavyBlue,citecolor=Gray,urlcolor=Periwinkle]
{hyperref}

\usepackage[nameinlink,capitalize]{cleveref}
\usepackage[dvipsnames]{xcolor}
\usepackage[shortlabels] {enumitem}
\usepackage[babel=true,kerning=true]{microtype}
\usepackage[abbrev,msc-links,alphabetic]{amsrefs}
\usepackage{lscape}
\usepackage{afterpage}

\usepackage{lineno}
\usepackage[textsize=tiny]{todonotes}
\usepackage{graphics} 

\newcommand{\reply}[1]{#1}

\theoremstyle{plain}
\newtheorem{theorem}{Theorem}[section]

\newtheorem{corollary}[theorem]{Corollary}

\newtheorem{lemma}[theorem]{Lemma}
\newtheorem{proposition}[theorem]{Proposition}

\theoremstyle{definition}
\newtheorem{definition}[theorem]{Definition}

\newtheorem{example}[theorem]{Example}

\newtheorem{remark}[theorem]{Remark}

\crefname{assumption}{Assumption}{Assumptions}

\numberwithin{equation}{section}

\makeatletter
\def\subsubsection{\@startsection{subsubsection}{3}%
  \z@{.5\linespacing\@plus.7\linespacing}{-.5em}%
  {\normalfont\bfseries}}
\def\paragraph{\@startsection{paragraph}{4}%
	\z@{.25\baselineskip\@plus.35\baselineskip}{-\fontdimen2\font}%
	{\normalfont\itshape}}
\def\subparagraph{\@startsection{subparagraph}{5}%
	\z@{.15\baselineskip\@plus.25\baselineskip}{-\fontdimen2\font}
	{\itshape}}
\makeatother

\newcommand{\assign}{:=}

\newcommand{\tmop}[1]{\ensuremath{\operatorname{#1}}}

\newcommand{\wei}[1]{\langle#1\rangle}
\newcommand{\bk}[1]{\llbracket#1\rrbracket}
\newcommand{\norm}[1]{\|#1\|}
\newcommand{\normlr}[1]{\left\|#1\right\|}
\newcommand{\bklr}[1]{\left\llbracket#1\right\rrbracket}
\newcommand{\tand}{\quad\textrm{and} \quad}

\newcommand{\ctrl}{\eta}
\newcommand{\law}{\mathrm{Law}}
\newcommand{\1}{\mathds{1}}

\colorlet{darkred}{red!90!black}

\newcommand{\caa}{\mathcal A}

\newcommand{\cll}{\mathcal L}
\newcommand{\cff}{\mathcal F}

\newcommand{\cpp}{{\mathcal P}}
\def\cgg{{\mathcal G}}
\def\cmm{{\mathcal M}}
\def\cjj{{\mathcal J}}

\newcommand{\Bor}{\ensuremath{\mathrm{Bor}}} 
\newcommand{\C}{{\mathcal C}}
\newcommand{\Cb}{\ensuremath{\mathcal C_b}}
\newcommand{\CC}{{\mathscr{C}}}

\newcommand{\sym}{{\mathrm{Sym} }}

\newcommand{\X}{{\ensuremath{\bm{X}}}}

\newcommand{\Z}{\ensuremath{\bm{Z}}}
\newcommand{\XX}{\ensuremath{\mathbb{X}}}
\newcommand{\R}{\ensuremath{\mathbf R}}

\newcommand{\E}{\ensuremath{\mathbb E}}
\renewcommand{\P}{\ensuremath{\mathbb P}}
\newcommand{\A}{\ensuremath{\mathcal A}}
\newcommand{\T}{\ensuremath{\mathcal T}}
\newcommand{\Taylor}{\ensuremath{L}}
\newcommand{\TT}{\ensuremath{\mathcal{\bar T}}}
\newcommand{\LL}{\ensuremath{\mathcal{\bar L}}}
\newcommand{\D}{\ensuremath{\mathbf D}}
\newcommand{\DD}{\ensuremath{\mathbf{\bar D}}}
\newcommand{\Vone}{{\bar{V}}}

\newcommand{\abx}{{\mathcal X}}

\newcommand{\lip}{\ensuremath{{\mathrm{Lip}}}}

\DeclareMathOperator*{\esssup}{ess\,sup}

\newcommand{\EE}{{{\mathbb E}_{\bigcdot}}}
\newcommand{\W}{{\bm{\Omega}}}
\usepackage{graphicx}

\makeatletter
\newcommand*\bigcdot{{\mathpalette\bigcdot@{.5}}}
\newcommand*\bigcdot@[2]{\mathbin{\vcenter{\hbox{\scalebox{#2}{$\m@th#1\bullet$}}}}}
\makeatother

\newcommand{\Deltatwo}{\ensuremath{\Delta\!\!\!\!\Delta}}

\renewcommand{\tilde}{\widetilde}

\DeclareFontEncoding{FMS}{}{}
\DeclareFontSubstitution{FMS}{futm}{m}{n}
\DeclareFontEncoding{FMX}{}{}
\DeclareFontSubstitution{FMX}{futm}{m}{n}
\DeclareSymbolFont{fouriersymbols}{FMS}{futm}{m}{n}
\DeclareSymbolFont{fourierlargesymbols}{FMX}{futm}{m}{n}
\DeclareMathDelimiter{\vvert}{\mathord}{fouriersymbols}{152}{fourierlargesymbols}{147}

\DeclarePairedDelimiter{\fl}{\lfloor}{\rfloor}

\begin{document}

\title{Rough stochastic differential equations}
\author{Peter K.~Friz}
\address{TU Berlin and WIAS Berlin}
\email{friz@math.tu-berlin.de}

\author{Antoine Hocquet}
\address{TU Berlin}
\email{hocquet@math.tu-berlin.de}

\author{Khoa Lê}
\address{TU Berlin. Current address: School of Mathematics, University of Leeds, U.K.}
\email{k.le@leeds.ac.uk}




\begin{abstract}
We establish a simultaneous generalization of It\^o's theory of stochastic and Lyons' theory of rough
differential equations. The interest in such a unification comes from a variety of applications, including
pathwise stochastic filtering, control and the conditional analysis of stochastic systems with common noise.
\end{abstract}
\maketitle

\tableofcontents
\section{Introduction} 
\label{sec.introduction}

It\^o's important theory of stochastic integration 
gives meaning and well-posedness
of multidimensional stochastic differential equations (SDEs) of the form
\begin{equation}\label{eq:SDE}
		d Y_t = b_t (Y_t)dt + \sigma_t ( Y_t) dB_t.
\end{equation}
Here $B=B(\omega)$ is a Brownian motion, the process $Y=Y(\omega)$  constitutes an important example of an It\^o process, e.g. \cite{RY99}. Crucially in this theory, coefficient fields like $\sigma_t (x) = \sigma(\omega,t,x)$ must be nonanticipating to enable the use of martingale methods.
In contrast, the purely deterministic theory of rough paths \cite{MR1654527} 
gives, under natural assumptions, well-posedness to rough differential equations (RDEs) of the form
 \begin{equation}\label{eq:RDE}
		dY_t=
		b_t (Y_t)dt + f_t(Y_t)d\X_t .
\end{equation}
Here $\X = (X,\XX)$ is a H\"older rough path, the solution $Y=Y^\X$ is an example of a {\em controlled rough path} \cite{MR2091358}, or simply {\em controlled} (w.r.t. $X$), in the sense that it looks like $X$ on small scales: $Y_t \approx Y_s + Y_s' (X_t - X_s)$, with $Y'_s := f_s (Y_s)$.
Crucially, the definition of $\int f (Y) d \X$, and then integral meaning to \eqref{eq:RDE}, requires $f(Y)$ itself to be controlled.
Many works on this subject, including \cite{FH20}, consider $f_t(\cdot) \equiv f(\cdot)$ so let us point out that non-regular time dependence requires extra considerations. 
A direct way to do so\footnote{But see \cite{bailleul2019rough} for a rough flow - and \cite{kelly2017deterministic} for a Banach RDE perspective.} 
 is to assume a controlled structure $f_t (\cdot) \approx f_s(\cdot) + f'_s (\cdot) (X_t -X_s)$.  Following \cite{MR2387018,FH20}, with focus on drift-free case $b \equiv 0$,  solutions come with a refined expansion
$$
         Y_t =Y_s + Y_s' (X_t - X_s) + Y''_s \XX_{s,t} + o(t-s), \quad Y''_s = ((Df_s) f_s+ f'_s) (Y_s),
$$
which characterizes RDE solutions. The inclusion of a drift poses no difficulties.\footnote{ ... and also appears as special case of \eqref{def.J} below.}

\medskip

{\em Brownian rough paths.} Important examples of rough paths come from the typical realization of a multidimensional Brownian motion enhanced with iterated (It\^o) integrals,
\[
	\X = (X, \XX) = \left(B(\omega), (\int B \otimes d B) (\omega) \right)=: \mathbf{B}^{\mathrm{It\^o}} (\omega).
\]
As is well-known, e.g. \cite[Ch.9]{FH20}, under natural conditions, $\bar{Y} (\omega) := Y^\X|_{\X = \mathbf{B}^{\mathrm{It\^o}} (\omega)}$ yields a beneficial version of the It\^o solution to \eqref{eq:SDE}; the Stratonovich case is similar.

\medskip

It has been an open problem for some time to provide a unified approach to SDEs and RDEs, such as to give intrinsic meaning and well-posedness to
{\em rough stochastic differential equations} (RSDEs), aiming for an adapted solution process $Y=Y^{\X} (\omega)$ to
\begin{equation}\label{eq:dRDE}
       d Y_t = b_t (Y_t)dt + \sigma_t ( Y_t) dB_t + f_t ( Y_t) d \X_t .
\end{equation}
At this stage, \eqref{eq:dRDE} is entirely formal. Making this equation meaningful and providing a satisfactory general solution theory under natural conditions is the main purpose of this work.

{\em Why RSDEs?} The interest in such a construction comes from a variety of applications and is the raison d'\^etre of several ad-hoc approaches, reviewed (together with their limitations) below.
For instance, recent progress on fast-slow systems, cf. \cite{pei2021averaging,hairer2021generating},
involves mixed dynamics of the form 
\begin{equation}\label{eq:mSDE}
        d Y_t = b (Y_t)dt + \sigma ( Y_t) dB_t + f ( Y_t) \circ dW^H;
\end{equation}
for some independent fractional Brownian noise $W^H$. This fits into \eqref{eq:dRDE} provided $W^H$ has a (canonical) rough path lift, which is well understood. Such dynamics also arise in quantitative finance
\cite{gatheral2018volatility, bayer2016pricing} where one naturally mixes $(dt,dB)$-modeled semimartingale dynamics (for tradable assets) with ``rough'' fractional $dW^H$-dynamics for volatility.
Perhaps the strongest case for \eqref{eq:dRDE}, which relates to works spanning over 4 decades (selected references below), comes in the form of ``doubly SDEs'' under conditioning; that is, 
\begin{equation}\label{eq:dSDE}
       d Y_t = b_t (Y_t)dt + \sigma_t ( Y_t) dB_t + f_t ( Y_t) \circ d W_t =  \tilde b_t (Y_t)dt + \sigma_t ( Y_t) dB_t + f_t ( Y_t) d W_t
\end{equation}
{\em conditionally} on some independent Brownian motion $W = W(\omega)$. Let us describe some concrete situations.
\medskip

\noindent (a) In the ``Markovian'' case of deterministic coefficient fields, with $ \cff^W_T = \sigma \{ W_t : t \le T\} $,
\begin{eqnarray}
u( s,y;\omega) &=& \E \left( h\left( Y^{s,y}_{T}\right) \exp \left(
\int_{s}^{T}c\left( Y^{s,y}_{t}\right) dt+\int_{s}^{T}\gamma \left( Y^{s,y}_{t}\right)
d W \right) \Big| \cff^W_T \right)  \label{preroughFKformula},
\end{eqnarray}%
yields the Feynman-Kac solution for the (terminal value) SPDE problem, $u(T,\cdot) = h$,
\begin{equation}\label{eq:bSPDE}
 - d_tu=\left(\tfrac{1}{2}\mathrm{Tr}\left( (\sigma \sigma ^{T})\left( x\right)
D^{2}u\right) + b\left( x\right) \cdot Du +c\left(
x\right) u\right)dt+ ( f\left( x\right) \cdot
Du +\gamma \left( x\right) u) \circ {dW_t},%
\end{equation}
with $\circ {dW}$ understood in backward Stratonovich sense. See e.g. \cite{pardoux1980stochastic, kunita1997stochastic, MR3746646} and many references therein. \\
(b) As noted explicitly by Bismut--Michel \cite{MR647076}, the conditional process $Y|W$ is not a semimartingale (since $W$ has a.s. locally infinite variation).
Their analysis then relies on auxiliary stochastic flows, obtained from $d \Phi_t (x) =  f ( \Phi_t (x) ) \circ d W_t, \Phi_0 (x) =x$, to remove the
troublesome ``$(...)d W$ {conditioned on} $W$'' term in \eqref{eq:dSDE}.

Conditional It\^o diffusions 
are also at the heart of stochastic filtering theory, e.g. \cite{el2012stochastic,MR2454694} and many references therein. Here $Y$ describes the signal dynamics, not necessarily Markovian, affected through the observation $W$ which, after a Girsanov change of measure, has Brownian statistics. The celebrated Kallianpur--Striebel formula expresses the filter (the conditional expectation of some observable $h$ of the signal, given the observation $W$) as the ratio $\pi_t(h) / \pi_t(1)$, where $\pi_t(h)$ has a similar form to \eqref{preroughFKformula}, with the exponential term coming from the
Girsanov theorem. %
Understanding the robustness of the filter with respect to $W$ is a classical question in filtering theory \cite{davis2011pathwise, MR3134732}.

\noindent (c) Stochastic flow transformations are also employed by \cite{MR2285722} where the authors have
controlled It\^o characteristics, $g \in \{ b, \sigma \}: g_t (Y_t) = g (Y_t, \eta_t (\omega))$,
for suitably non-anticipating controls $\ctrl(\cdot)$, and study the random value function, for $0 \le s \le T$,
\begin{equation}   \label{equ:value}
          v (s,y; \omega) :=\mathrm{essinf}_{\ctrl (\cdot) }  \E \left(  h (Y^{s,y,\ctrl}_T) +  \int_s^T \ell (t, Y^{s,y,\ctrl}_t; \ctrl_t) d t \Big| \cff^W_T \right).
\end{equation}
This {\em pathwise stochastic control} problem was first suggested in \cite{MR1647162}, albeit with constant $f (\cdot) \equiv f$ in dynamics \eqref{eq:dSDE}, as motivation for stochastic viscosity theory:
According to \cite{MR1647162, MR2285722, bhauryal2023pathwisestochasticcontrolclass} the value function defined in \eqref{equ:value} is a ``stochastic viscosity solution'' for a nonlinear stochastic PDE of the form
\begin{equation}   \label{equ:sHJB}
    - d_tv=  \inf_{\ctrl} \left (\tfrac{1}{2}\mathrm{Tr}\left(
    (\sigma \sigma ^{T})\left( x,  \ctrl\right)
    \left( x\right)
D^{2}v\right) + b\left( x, \ctrl \right) \cdot Dv \right) d t 
+ 
(f\left( x\right) \cdot Dv) \circ {dW_t}.
\end{equation}
Classical HJB (viscosity) equations, contained herein upon taking $f \equiv 0$, may not admit solutions with $\C^1$ spatial regularity,   
so there is little hope to give \eqref{equ:sHJB} a bona fide (backward It\^o/Stratonovich) stochastic integral meaning: Accordingly, \cite{MR1647162} propose a pathwise theory;
non-constant $f$ in \eqref{equ:sHJB} requires rough paths \cite{MR2765508, allan2020pathwise,chakraborty2024pathwiserelaxedoptimalcontrol}.

\noindent (d) Another motivating example comes from weakly interacting particle systems, driven by independent Brownian motions $B^1, ..., B^N$, subjected additionally to environmental (a.k.a. common) Brownian noise $W$.
Under suitable assumptions, one has conditional propagation of chaos, cf. \cite{CF16},
with the effective dynamics \eqref{eq:dSDE} of such a system governed by conditional McKean--Vlasov dynamics, with $g \in \{ b, \sigma, f \}: g_t (Y_t) = g (Y_t, \mathrm{Law} (Y_t | \cff^W_T  ))$. In a Markovian situation,
the law of this process follows a non-linear, non-local stochastic Fokker--Planck equation \cite{coghi2019stochastic, coghi2019rough}. The case of controlled McKean--Vlasov dynamics, $g_t (Y_t) = g (Y_t, \ctrl (\omega), \mathrm{Law} (Y_t | \cff^W_T  ))$ arises in the important area of mean-field games, e.g. \cite{carmona2018probabilistic}, with $W$ viewed as common noise. The conditional analysis of such equations, with common noise and progressive coefficients, is also central to \cite{lacker2022superposition}.

\medskip

Whether $W$ is interpreted as noise, observation, environment, or common noise, the importance of quantifying its impact on some stochastic model $Y | W$, or predication based thereon, is evident. We shall see that RSDEs, as developed in this work, do this in a satisfying way. Our (in SDE terminology) ``strong'' analysis not only removes tedious measure theoretical issues inherit to the conditional problems, but yields  a fundamental {\em partial} decomposition of It\^o-map: writing  $\bar{Y} \equiv Y$ for the SDE solution to \eqref{eq:dSDE}, driven by $(B,W)$ and with given initial data, we can decompose
\begin{equation} \label{equ:FDC}
	\bar{Y} (\omega) = (Y^\bullet) \circ \mathbf{W} (\omega), \quad (Y^\bullet): \X \mapsto Y^\X (\omega),
\end{equation}
into a (well understood) universal lifting map $\mathscr{L} : W \mapsto \mathbf{W} (\omega)$, and a robust RSDE solution  $Y^\bullet$. 
(For completeness, we show in Appendix \ref{app:RSDE} how this leads to the first equality in \eqref{equ:FDC}, together with a robust disintegration of $\law{(\bar{Y})}$, given $W$.)  This picture is reminiscent of
Lyons' original work, aiming to decompose SDE solutions as (deterministic) RDE solutions driven $\omega$-wise by a lifted Brownian motion. Yet, existing rough path tools are quite insufficient for our goals. Before commenting on the new techniques involved, we give a loose statement of our main result.

\begin{theorem}\label{thm.1}
  Under suitable regularity and boundedness assumptions on (possibly progressive) coefficient fields $b, \sigma, f$, consistent with those from It\^o SDE and RDE theories, there is
  a unique strong RSDE solution to \eqref{eq:dRDE}, to which we give intrinsic local and integral sense. The solution is exponentially integrable
  and comes with precise local Lipschitz estimates with respect to $\{Y_0, b, \sigma, f, \X \}$.
 \end{theorem}

This result allows to treat a variety of situations (ranging from mixed SDEs, pathwise filtering and stochastic control to common noise McKean--Vlasov and its particle approximations) in the desired (rough)pathwise fashion, that is, with $W^H$ or $W$ replaced by a deterministic rough path $\X$. Using the language of diffusions in random environments, we offer a {\em quenched} theory, with $\X$ seen as frozen environmental noise. At any stage, one can
return to the {\em annealed} (``doubly stochastic'') setting by randomization of $\X$, as discussed in Appendix \ref{app:RSDE}, based on \cite{FLZ25}.

\medskip
\cref{thm.1} is a loose summary of  \cref{thm.fixpoint} (existence, uniqueness), \cref{cor.expInt} (exponential integrability) 
and \cref{thm.stability_precise} (stability and local Lipschitz estimates).

 Central to our analysis is a new class of processes, {\em stochastic controlled rough paths} (s.c.r.p.), conceptually related to {\em rough semimartingales} \cite{friz2023rough} in their ability to mix martingales and adapted controlled processes, but analytically very different, based on an extension of
{\em stochastic sewing} \cite{le2018stochastic} to mixed $L_{m,n} (\Omega)$-spaces, cf. \cref{sub.ssl}. The resulting s.c.r.p.'s crucially involve two $\mathbb{P}$-integrability parameters which allow us to detangle
an (inevitable) loss of integrability (of the sort $L_{m,n} \to L_{m,n/2}$) upon composition of a s.c.r.p. with spatially regular (non-linear) $f$ and more general stochastic controlled vector fields (s.c.v.f.). While $m=n=\infty$ does not even accommodate Brownian motion, leave alone other reasonable classes of solution processes, it turns out that $m<n=\infty$ does.
After developing a rough integration theory for s.c.r.p. (\cref{sub.RSI}) we can close the loop in a fixed-point argument in our construction of a unique solution. 
Like s.c.r.p.'s, we should remark that s.c.v.f.'s (\cref{sub.stability}) 
have no counterpart in the deterministic rough paths literature.
While natural, our motivation for this kind of generality is rooted in the application to interacting particle systems with rough common noise with rough  (``quenched'') McKean--Vlasov limit, subject of our work \cite{friz2025McKean}. In this case $f$ not
only depends on $y = Y_t \equiv Y^\X_t$ but comes with a non-regular time dependence induced by  $t \mapsto \law (Y^\X_t ) \equiv \law (Y_t ; \X)$, 
 or a random approximation thereof, namely the empirical measure of the particle cloud.

RSDE well-posedness is complemented with precise estimate of local Lipschitz type in the data (\cref{thm.stability_precise}). In the so-called critical case local Lipschitz estimates are lost, but the problem remains well-posed (\cref{sub:uniqueness_when_}), thanks to a ``stochastic, rough Davie--Gr\"onwall-type lemma (\cref{sub.davie_gronwall_type_lemmas}) which may be useful in its own right.
\footnote{Readers familiar with previous (arXiv) versions of this article may note a simplified direct proof of the local Lipschitz estimates, without reliance on the technical Davie--Gr\"onwall lemma.}

In \cref{sec:ito} {\em rough It\^o processes} are introduced, which provide a flexible class, beyond the semi-martingale world, for which one has an It\^o-type formula. A rough stochastic calculus emerges, of which we can here only scratch at the surface: we introduce the {\em rough martingale problem} and further give an effective rough Fokker--Planck equation for RSDE, in a generality that also applies immediately to solutions of McKean--Vlasov SDEs with (rough) common noise, as provided by \cite{friz2025McKean}.

Closely related to the rough martingale problem, our final section \cref{sub:weak_existence} makes the point that the ``strong'' RSDEs theory of  \cref{thm.1} also has a ``weak'' counterpart.
Many natural questions emerge, starting with well-posedness for non-degenerate low regularity coefficients \`a la Stroock--Varadhan, with accompanying analytic questions for rough PDEs. We finally mention the possibility of a localized RSDE theory, a systematic study of which is left for a future note.\footnote{Partial results are contained in previous  (arXiv) versions of this article.}

\medskip

\noindent {\em Previously on RSDEs}. 
Assume $d \Phi^\X_t (x) =  f ( \Phi^\X_t (x) ) d\X_t, \Phi^\X_0 (x) =x$ is well posed, $\X$ is a rough geometric path. The {\em flow transformation} (FT) method for RSDEs amounts to {\em define} $Y_{\mathrm{FT}}^\X := \Phi^\X (\tilde Y)$, in terms of a distorted It\^o SDE for
\[
	\tilde Y (\omega) = (\Phi^\X)^{-1} (Y^\X_{\mathrm{FT}}), \quad d \tilde Y_t = \tilde b_t (Y_t; \X)dt + \tilde \sigma_t ( Y_t; \X) dB_t.
\]
This construction, classical in the SDE case, goes back to \cite{MR3134732} for RSDEs, where it was seen that, for $\X$ of Brownian regularity, it is necessary $f \in \C^{5^+}$ to have local Lipschitz dependence of
$\X \mapsto Y_{\mathrm{FT}}^\X$. (In contrast,  \cref{thm.1} gives this 
under the expected minimal $f \in \C^{2+}$ condition.)
Excessive regularity demands aside, FT methods are rather rigid and do not cope well with general $f=f_t (y, \omega)$, as is possible in \cref{thm.1}, and needed for instance in the common noise McKean--Vlasov situation described above. Even if one consents to a structural restriction like $f=f(Y_t)$, a flow-based definition of solution
lacks the intrinsic and local meaning that is relevant, for instance, when studying discretizations of RSDE dynamics.

As a concrete example, let $s,t$ be consecutive points in some partition $\pi$ of $[0,T]$ and consider the ``Euler-in-$B$, Milstein-in-$\X$'' scheme 
\begin{equation}   \label{equ:HybridScheme}
          Y^\pi_{t} =Y^\pi_{s} + b_s (Y^\pi_s) (t-s) + \sigma_s ( Y^\pi_s) (B_t - B_s)  + f_s ( Y^\pi_s)  (X_t - X_s) + F_s(Y^\pi_s) \XX_{s,t} ,
\end{equation}
with $F = (D_y f)f+f'$ where $f'$ accounts for possible $X$-controlled time dependence of $f$. A convergence analysis of this scheme would be tedious to carry out from a FT perspective. In contrast, local RSDE estimates as provided in \cref{prop.davie} make it at least plausible that this can be done efficiently in the framework of this work, content of forthcoming work. 
 When applied in the (rough) pathwise control setting, that is, the rough counterpart of \eqref{equ:value} with
\[
	v^\X (s,y) :=\mathrm{inf}_{\ctrl (\cdot) }  \E \left(  h (Y^{\X; s,y,\ctrl}_T) +  \int_s^T \ell (t, Y^{\X; s,y,\ctrl}_t; \ctrl_t) d t \right),
\]
this opens up to the possibility to study the finite difference of nonlinear stochastic PDEs of the form \eqref{equ:sHJB}, not implied (unless $f$ is constant) by presently available theory \cite{seeger2020approximation}. We also note in passing that the above expression for $v^\X$ in conjunction with precise RSDE estimates  (\cref{thm.stability_precise}) gives a direct approach to estimating H\"older space time regularity of such SPDEs, valid for every (rough path) realization of the driving noise. This is a powerful way to obtain regularity results for stochastic HJB equations (problem left open in \cite{MR2285722}) and can also be compared with recent work \cite{cardaliaguet2021holder}.

\medskip

A second previous approach, dubbed {\em random rough path} (RRP) method, amounts to define  $Y_{\mathrm{RRP}}^\X (\omega) := \hat Y^{\Z (\omega)}$ as the $\omega$-wise solution to the RDE
\begin{equation}   \label{equ:rRDE}
       d \hat Y_t=
		b_t (\hat Y_t)dt + (\sigma_t, f_t)(\hat Y_t)d\Z_t (\omega) ,
\end{equation}
driven by the random rough path $\Z (\omega)$ over $Z(\omega):=(B(\omega),\X)$, where the second level $\mathbb{Z}(\omega)$ is naturally specified via 4 blocks, given by
\[
	\int B \otimes dB,\, \int X \otimes dB,\, \int B \otimes dX := B\otimes X- \int (dB)\otimes X \text{ and } \XX.
\]
Here, all $dB$-integrals are in It\^o sense, $\mathbb{X}$ is the second level component of $\X$. This construction is due to \cite{MR3274695}, see \cite{friz2023rough} for the case of c\`adl\`ag martingale and $p$-rough paths.
It was also used in \cite{MR3746646} for intrinsic well-posedness of the RPDEs counterpart of \eqref{eq:bSPDE}, with $W$ replaced by $\X$, and most recently for McKean--Vlasov equations with (rough) common noise \cite{coghi2019rough}, to be distinguished from rough McKean--Vlasov (or mean field) equations \cite{bailleul2020,bailleul2020propagation} which also have a RRP flavor.\footnote{In these works, $Z(\omega)$ is a joint lift of $\X (\omega')$ and $\X (\omega'')$  for a suitable random rough path $\X$. No martingale structure is assumed.}
As a general remark, after the construction of a suitable joint lift $\Z=\Z(\omega)$, RPP methods rely on deterministic analysis and cannot benefit from the (partial) martingale structure inherent in RSDEs. This becomes
a serious issue for integrability, e.g. for exponential terms as seen in the rough counterpart (replace $W$ by $\X$) of \eqref{preroughFKformula}, 
and an insurmountability when it comes to
general progressive randomness in coefficients, 
a situation that cannot be dealt with by RRP methods. 
Indeed, the
It\^o coefficient field $\sigma = \sigma_t (\cdot, \omega)$ is now subject to the stringent space-time regularity and rigid controlledness conditions of vector fields in RDE theory.
This entails that the (minimal) Lipschitz-condition one expects from It\^o SDE theory has to be replaced by a suboptimal $\C^{1/\alpha}$-condition,
and further rules out general (progressive) time-dependence, as would be required to incorporate stochastic control aspects in \eqref{equ:rRDE}.
(All these limitations are removed by our \cref{thm.1}.)

\medskip

\noindent {\bf Summary and outlook}: 
Based on a complete intertwining of stochastic and rough analysis, the RSDE framework put forward in this work offers a powerful approach to many problems previously treated with flow transformation and/or random rough path methods. Immediate benefits include  the removal of excessive regularity demands seen in (all) such previous works, intrinsic (local) meaning to the equations of interest, and a significant relaxation of previously imposed structural assumptions (e.g. progressive vs. deterministic coefficients fields, as required in stochastic control). Concerning the outlook, our results and techniques  are of direct interest  for stochastic analysis (``partial'' Malliavin calculus, H\"ormander theory, random heat kernels ... ) of conditional processes, as studied by Bismut, Kunita, Nualart, and many others in the 80/90ties\footnote{After this work was made available, a first study of Malliavin calculus for RSDEs has been carried out in  \cite{bugini2024malliavincalculusroughstochastic}.}. We also have first evidence that our framework enables a ``robust'' conditional analysis of doubly stochastic backward SDEs \cite{pardoux1994backward, diehl2012backward}. Further uses can be expected in the vast field of mean field games (with common noise). Last not least, we envision extensions from rough SDEs to rough SPDEs, as may arise from the filtering of non-linear SPDEs. The present work is of foundational nature.
%
%
%

\medskip

\noindent {\bf Update}: At revision stage, we may point to \cite{bugini2025rough, FLZ24x, BFS24x, bugini2024malliavincalculusroughstochastic, friz2025McKean, BFS25, BFHL25} for direct developments based on this work. 

\noindent {\bf Acknowledgment}: 
PKF acknowledges support from DFG CRC/TRR 388 ``Rough Analysis, Stochastic Dynamics and Related Fields'' (project ID: 390685689), as well as a MATH+ Distinguished Fellowship from the Berlin Mathematics Research Center MATH+ (EXC-2046/1, project ID: 390685689). PKF also acknowledges initial support from the European Research Council, via the Consolidator Grant ``GPSART''. AH was supported by DFG CRC 910 ``Control of self-organizing nonlinear systems: Theoretical methods and concepts of application'', Project A10. KL was initially supported by an Alexander von Humboldt Research Fellowship and the above ERC grant. 
KL is now supported be the Engineering \& Physical Sciences Research Council (EPSRC), grant number EP/Y016955/1.
 KL thanks Máté Gerencsér for various interesting discussions related to Davie–Grönwall-type lemma. We all thank the referees for their careful reading and detailed feedback, which greatly improved clarity and presentation of this paper.	
 
\medskip

{\bf Frequently used notation.} 
For two extended real numbers \( a,b\in \mathbf R\cup\{\infty\} \) we write \( a\wedge b=\min\{a,b\} \) and \( a\vee b=\max\{a,b\} .\)
If more parameters are presented, \( \min\{\dots\} \) and \( \max\{\dots\} \) will be used instead.
The Borel-algebra of a topological space $\mathcal T$ is denoted by $\Bor(\mathcal T)$.
Throughout the manuscript we fix a finite deterministic time horizon $T>0$.  Accordingly the notation \( I\subset [0,T]\) refers to a (generic) compact interval and  we denote by \( |I| \) its length.
	The notation $F\lesssim G$ means that $F\le CG$ for some positive constant $C$; similarly, \( F \asymp G \) means \( F \lesssim G \) and \( G \lesssim F \).
	For two Banach spaces \( (\mathcal X, |\cdot|_{\mathcal X}) \), \( (\mathcal Y,|\cdot|_{\mathcal Y}) \) with \( \mathcal X \subset\mathcal Y\), we write
\( \mathcal X\hookrightarrow \mathcal Y \) if \( \mathcal X \) is continuously embedded into \( \mathcal Y \), in the sense that \( |\cdot|_{\mathcal Y}\lesssim |\cdot|_{\mathcal X} \).
By $V,\Vone,W, \bar{W}$ we denote real finite-dimensional Banach spaces. Their norms are denoted indistinctly by \( |\cdot| \). 
	The Banach space of linear maps from $V$ to $W$, endowed with the induced norm \( |K|:=\sup_{v\in V,|v|\le1}|Kv| \), is denoted by $\cll(V,W)$. Tensor products are equipped with a norm such that \( V\otimes W\simeq \cll(V,W)\) isometrically and, accordingly, we shall blur the difference between \( \cll(V,\cll(\bar V,W)) \), \( \cll(V\otimes \bar V;W) \) and bilinear maps from \( V\times \bar V\to W \). All instances of \( \otimes \) in this manuscript pertain to finite-dimensional spaces.

\section{Preparations} 
\label{sec.preliminaries}

\subsection{Framework}
\label{sec:framework}
We introduce some basic spaces and concepts. 
\subsubsection{Function spaces} \label{sec:fs}
	We denote by \( (\Cb,|\cdot|_{\infty}) \) the Banach space of continuous and bounded maps, namely
	\[
	 \Cb =\Cb(\mathcal T;W)=\left \{f\colon \mathcal T\to W\enskip\text{ continuous and s.t.\ }
	 |f|_\infty<\infty
	 \right \},\quad |f|_\infty:=\sup_{x\in \mathcal T}|f(x)|\,.
	\]
	For every $\alpha\in(0,1]$ and every function $g:V\to W$, we denote by $[g]_\alpha$ its H\"older seminorm, i.e.
	\[
		[g]_\alpha=\sup_{x,y\in V:x\neq y}\frac{|g(x)-g(y)|_W}{|x-y|_V^\alpha}.
	\]
	For $\kappa=N+\alpha$ where $N$ is a non-negative integer and $0<\alpha\le 1$, $\C^\kappa_b(V;W)$ denotes the Lipschitz space of bounded functions $f\colon V\to W$ such that $f$ has Fr\'echet derivatives up to order $N$, $D^jf$, $j=1,\ldots,N$ are bounded functions  and $D^Nf$ is globally H\"older continuous with exponent $\alpha$.
	Recall that for each $v\in V$, $Df(v)\in\cll(V,W)$, $D^2f(v)\in\cll(V\otimes V,W)$ and so on.
	Whenever clear from the context, we simply write  $\C^\kappa_b$ for $\C_b^\kappa(V;W)$.
	For each $f$ in $\C_b^\kappa$, we have seminorm and norm, respectively, given by
	\[
			[f]_\kappa=\sum\nolimits_{k=1}^N|D^kf|_\infty+[D^Nf]_\alpha
			\tand |f|_\kappa=|f|_\infty+[f]_\kappa.
	\]

	\subsubsection{Rough paths} \label{sec:rps}
	Given a compact interval \( I\subset[0,T] \) we shall work with the simplices $\Delta(I)$ and $\Deltatwo(I),$ defined as
	\[
	\begin{aligned}
	&\Delta (I) :=\{(s,t)\in I^2,\enskip \min I\le s\leq t\le \max I\},
	\\
	&\Deltatwo(I):=\{(s,u,t)\in I^3,\enskip \min I\leq s\leq u \leq t\leq \max I\}.
	\end{aligned}
	\]
	We write \( \Delta=\Delta(I) \) and \( \Deltatwo=\Deltatwo(I) \) whenever clear from the context.
	As is common in the rough path literature, given a path $Y=(Y_t)\colon I\to W$, we denote by $(\delta Y_{s,t})_{(s,t)\in\Delta }$ the increment of $Y$, which is the two-parameter map
	\begin{equation}  
	\label{nota:increments}
	\delta Y_{s,t}:=Y_t-Y_s,\quad \text{for every }\enskip  (s,t)\in\Delta .
	\end{equation}
	The (Banach) space $C^\alpha(I;V)$ then consists of all paths $Y: I \to V$ with finite semi-norm and norm 
	\begin{equation}	\label{holder_1}
 		[Y]_\alpha := |\delta Y| _\alpha := \sup_{s,t\in V:s\neq t}\frac{|\delta Y_{s,t}|}{|t-s|^\alpha} \tand |Y|_{\alpha }:= |Y|_\infty + |\delta Y|_{\alpha },
	\end{equation}	
	where $|Y|_\infty  = \sup _{t \in I}|Y_t| $. More generally, the (Banach) space $C_2^{\alpha}(I;V)$, consists of those two-parameter maps $A\colon\Delta \to V$ with finite norm
	\begin{equation}	\label{holder_2}
 		|A|_\alpha :=\sup_{(s,t)\in\Delta,s\neq t }\frac{|A_{s,t}|}{(t-s)^\alpha }. 
	\end{equation}
        Sewing arguments also require three-parameter maps $\delta A\colon\Deltatwo\to W$ given by
	\begin{equation}
	\label{nota:increments_2}
	\delta A_{s,u,t}:= A_{s,t}-A_{s,u}-A_{u,t},\quad \text{for every}\enskip (s,u,t)\in \Deltatwo\,.
	\end{equation}

We recall the definition of a (level-two) \( \alpha \)-H\"older rough path, as seen e.g. in \cite{FH20}.
	\begin{definition}	\label[definition]{def.RP}
		Write $\X=(X,\XX) \in \mathscr{C}^\alpha (I;V)$, provided
		\begin{enumerate}[(a)]
		 	\item $(X,\XX)$ belongs to $C^\alpha(I;V)\times C_2^{2 \alpha}(I;V\otimes V)$, 
		 	\item for every $(s,u,t)\in \Deltatwo$, Chen's relation holds
			\begin{equation}\label{chen}
	    \XX_{s,t}-\XX_{s,u}-\XX_{u,t}=\delta X_{s,u}\otimes \delta X_{u,t}\,.
			\end{equation}
		 \end{enumerate}
	\end{definition}	
	\reply{For $\alpha \in (\frac13,\frac12]$ this is exactly the definition of an $\alpha $-H\"older \textit{rough path} 
	on $I\subset[0,T]$, with restriction on the exponent due to classical sewing arguments that underly the integral formulation of rough differential equations. From a rough integration perspective, lower $\alpha$ is possible to the extent that the integrand has better regularity (cf. \cref{thm.Hroughint}).
	
	}
	For $\alpha' \in [0,\alpha]$ we will measure the distance of $\X,\bar\X\in\CC^\alpha$ with
	\begin{equation}
	\label{def.rho_metric}
		\rho_{\alpha,\alpha'}(\X,\bar\X) :=(|\delta X- \delta\bar X|_\alpha+|\XX-\bar\XX|_{\alpha+ \alpha'}), 
	\end{equation}
	often with $\alpha = \alpha'$ in which case we write \( \rho_{\alpha} = \rho_{\alpha,\alpha}\). 
	We also set $$\rho_{\alpha,\alpha'}(\X) :=|\delta X|_\alpha+|\XX|_{\alpha+ \alpha'}$$ and $ \rho_{\alpha}$ accordingly. 
	 Every smooth path $X\colon[0,T] \to V$ gives rise to a canonical rough path lift, with $\XX_{s,t} = \int_s^t \delta X_{s,r} \otimes d X_r$; we write  $\CC^{0,\alpha}_g$ for the closure of such canonically lifted smooth paths in $\CC^\alpha$.

\subsubsection{Stochastic setup} \label{sec:ss}
From now on, we work on a fixed complete probability space $(\Omega,\mathcal G,\P)$ equipped with a filtration $\{\cff_t\}$ with index set $[0,T]$, such that $\cff_0$ contains the $\P$-null sets. We also denote by $\W=(\Omega,\mathcal G, \P;\{\cff_t\})$ and call it a stochastic basis. Expectation with respect to $\P$ is denoted by $\E$.

\paragraph{Random variables and stochastic processes} Consider a generic, non-necessarily separable Banach space \( (\abx ,|\cdot|_{\abx})\).
	We call \( \xi\colon \Omega\to \abx \) a random variable if it is strongly \( \mathcal G/\Bor(\abx) \)-measurable; meaning that \( \xi \) is measurable relative to \( \mathcal G /\Bor(\abx)\) and separably valued.%
	\footnote{More precisely, \cite[Chap.~2]{MR1102015}) there exits a closed separable subspace \( \mathcal Y\subset \abx \) s.t.\ \( \mathbb P\circ\xi^{-1}\) is supported in \(\mathcal Y \); see also \cite[Appendix E]{cohn2013measure} or Chapter 1 in \cite{Hytnen2016} for more recent expositions.}
	Accordingly, we call $Y=Y(\omega,t)\in \abx $ a stochastic process if \( Y_t=Y(\cdot,t) \) forms a family of \( \abx \)-valued random variables. We call it adapted if for every $t \ge 0$, $Y_t$ is $\cff_t$-measurable, progressively measurable if $Y$, restricted to $\Omega \times [0,t]$, is strongly $\cff_t \otimes \Bor ([0,t])$-measurable.  
	


\paragraph{Lebesgue spaces} Write $L_0(\cgg;\abx)$ or $L_0(\abx)$ for the vector space space of all (strongly measurable)  random variables with values in $\abx$, with the usual convention that a.s. identical random variables are identified.  In particular, then $|\xi|_{\mathcal X} \in L_0^+(\R)$, the space of non-negative real-valued random variables. By definition, $L_m\left(\Omega,\mathcal G,\P;\abx\right)$ consists of all random variable $\xi$ with $|\xi|_{\mathcal X} \in L_m\left(\Omega,\mathcal G,\P;\R\right)$, any $m \le \infty$; we freely use shortened notation $L_m(\cgg;\abx), L_m^{
\cgg}(\abx),  L_m(\abx), L_m^{\cgg}, L_m$. 
For $\xi \in L_0 (\abx)$ we set
	 \[
	 	\|\xi\|_{m} := (\E |\xi|_{\mathcal X}^m )^{1/m} \in [0,\infty], \; m < \infty,
	 \]
	 and also $\|\xi\|_{\infty}:= \esssup |\xi|_{\mathcal X} \in [0,\infty]$. For $m \ge 1$, this makes $L_m(\abx)$ a Banach space.
	The notion of conditional expectation, classical treated for $L_1(\R)$-random variables, extends to $L_1(\abx)$-valued and also to $L_0^+(\R)$-valued random variables (see e.g. \cite[Ch.4]{Li2017} and \cite[Lem 3.1]{JP04}.)

\paragraph{Moment H\"older spaces}
	The increment notation \eqref{nota:increments}, \eqref{nota:increments_2} applies (pointwise) to stochastic processes $Y\colon\Omega\times I\to \abx$ and $A\colon\Omega\times\Delta(I) \to \abx$. We call $Y,A$ integrable ($L_m$-integrable) if $Y_t,A_{s,t}$ are integrable ($L_m$-integrable) for every $t \in I, (s,t)\in \Delta$, respectively.
	Adapting \eqref{holder_1}, \eqref{holder_2} to the process setting, we can define (when $m \ge 1$: Banach) spaces $C^{\alpha}L_{m}(I,\W;\abx)$ and $C_2^{\alpha}L_{m}(I,\W;\abx)$ where, respectively,
	\footnote{Note that $ \|Y\|_{0;m}
=\sup_{t\in I}\|Y_t\|_m
+\sup_{(s,t)\in\Delta}\|\delta Y_{s,t}\|_m
\asymp
\sup_{t\in I}\|Y_t\|_m.$}
	\begin{equation} \label{Yam} 
	\|Y\|_{\alpha,m}:=\sup_{t\in I}\|Y_t\|_{m} + \|\delta Y\|_{\alpha ,m}  < \infty \, ,
	\end{equation}	
	\begin{equation} \label{Aam} 
	\|A\|_{\alpha ,m}:= \sup_{(s,t)\in\Delta,s\neq t } \frac{\|A_{s,t}\|_{m}}{(t-s)^\alpha }  < \infty \,.
	\end{equation}
A generalization where $\|A_{s,t}\|_{m}$ is replaced by a ``mixed'' moment norm $\|\|A_{s,t}|\cff_s\|_m\|_n$ is given in Definition \ref{def.2amn} below and of central importance to this work. 	

	

\subsection{Spaces of mixed integrability} 
\label{sub.processes_with_finite_}
We 
introduce a family of integrable two-parameter processes with suitable regularity and integrability properties with respect to a fixed filtration. The linear spaces formed by these stochastic processes are the foundation of our analysis in later sections.

To this aim, an important concept that needs to be discussed is that of ``mixed integrability''. Even though it will be mostly used in the context of stochastic processes and a filtration, it is first better understood at the level of random variables, for which a sole sub-sigma algebra \( \mathcal F\subset \mathcal G \) suffices.
It is very likely that these spaces have been introduced before, however finding appropriate references turns out to be difficult. For instance, the \( L_{1,\infty} \)-norm appears implicitly in \cite[Appendix A]{MR2190038} (aimed at \( L_p \)-estimates for certain singular integral operators) but not exactly (compare the left hand side of  \cite[eq.~Appendix A(1.1)]{MR2190038} with the quantity introduced in \cref{def.C2mn}\ref{itm:C2mn} below).  

\subsubsection{Random variables of mixed integrability} \label{sec:mms}
Let \( 0 < m, n \le \infty\), and 
fix a sub-sigma-field $\cff \subset \cgg$.
For a $\abx$-valued r.v. \( \xi(\omega)\), we have $ | \xi |_{\abx}  \in L_0^+(\cgg;\R)$ so that   
\begin{equation}
	\label{Lmn_norm}
		 \| \xi | \cff \|_m := \left[ \E (| \xi |_{\abx}^m | \cff)
   \right]^{\frac{1}{m}} \in L_0^+(\cff;\R), \quad \| \xi \|_{m, n} \assign
   \left\{\begin{array}{l}
     \|\| \xi |\mathcal{F} \|_m \|_n \qquad \text{if } \| \xi | \cff \|_m \in L_n\\
     + \infty \text{\qquad \qquad  \quad  otherwise}
   \end{array}\right. 
\end{equation}
is well-defined. (As always, the cases $m$ or $n=\infty$ are understood in $\esssup$-sense.)  This yields a scale of  mixed moment spaces, 
\begin{equation}
	\label{Lmn_space}
L_{m,n} (\abx) = L_{m,n}^{\cff, \cgg} (\abx) = \{ \xi \in L_m:  \|\xi\|_{m,n} < \infty \}.
\end{equation}
The following simple remarks will be used tacitly in the rest of the paper.
\begin{remark}
\label[remark]{rem:nothing_happens}
(i) When $m=n$, by the tower property, we see $L_{m,m}^{\cff, \cgg} = L_m^\cgg$, so that there is no dependence on $\cff$.

(ii) For $m \ne n$, loosely speaking, $\cff$ modulates the level of integrability, as seen from the (immediate to verify) extreme cases
$ L_{m,n}^{\{ \emptyset,\Omega \} , \cgg} =L_m^{\cgg}, \quad
L_{m,n}^{\cgg , \cgg} =L_n^{\cgg}. $
(It is generally true that $L_{m \vee n} \subset L_{m,n} \subset L_m$; cf. 
\eqref{embeddings} 
below.) 

(iii) For any $\xi \in L_0 (\cff; \abx)$, we have
\( \|\xi|\mathcal F\|_m=|\xi|_{\abx} \)
and hence $\| \xi \|_{m,n} = \| \xi \|_{n}$.
\end{remark}
With definitions \eqref{Lmn_norm}, \eqref{Lmn_space} at hand, we record some basic properties.
\begin{proposition}
\label[proposition]{prop.Lmn}
Let \( 1\le m\le n \le \infty\). 
\begin{itemize}
\item[(i)]  \(\left ( L_{m,n} , \|\cdot\|_{m,n}\right )\) is Banach space, which coincides with \(\left ( L_{m} , \|\cdot\|_{m}\right )\)  when $m=n$. 

\item[(ii)] For all random variables \( \xi \in L_0 \), we have\footnote{Recall $\xi\mapsto\|\xi\|_{m} = + \infty$ when $\xi \in L_0 \backslash  L_m$}
\begin{equation}
\label{ineq.mn}
\|\xi\|_m\le\|\xi\|_{m,n} \le\|\xi\|_n \le \infty,
\end{equation}
and continuous embeddings
\begin{equation}
\label{embeddings}
L_n\hookrightarrow L_{m,n}\hookrightarrow L_m\,.
\end{equation}
\item[(iii)] Lower semicontinuity. Let  $(\zeta^k)_k \subset L_{m}$ with
a.s. limit $\zeta$. Then
\[ \| \zeta  \|_{m, n} \le \liminf_{k } \| \zeta^k \|_{m, n} \le\infty.
\]
\item[(iv)] Let $\xi^k \rightarrow \xi \in L_m$.
Then 
\[ \| \xi \|_{m, n} \le \limsup_{k } \| \xi^k \|_{m, n} \le
   \infty. \]
\end{itemize}
\end{proposition}
\begin{proof}
One has \( \|\|\xi|\mathcal F\|_m\|_m = \E(\E(|\xi|_{\abx}^m|\mathcal F)^{\frac mm})^{\frac1m}=\|\xi\|_m \) by tower property of conditional expectations, which shows $L_{m,m} = L_m$.
Next, we  show \eqref{ineq.mn}.
The left inequality is a simple consequence of Jensen inequality, which asserts that, since \( n\ge m \),
\begin{align*}
\|\xi\|_m=\|\|\xi|\cff\|_m\|_m\le \|\|\xi|\cff\|_m\|_n=\|\xi\|_{m,n}.
\end{align*}
For the right inequality, use Jensen inequality in conditional form to see
\begin{align*}
\|\xi\|_{m,n}=\|\|\xi|\cff\|_m\|_n\le \|\|\xi|\cff\|_n\|_n=\|\xi\|_n.
\end{align*}
Embeddings \eqref{embeddings} follow immediately from \eqref{ineq.mn}.

{\itshape Banach property.}
The space \( L_{m,n} \) is clearly linear.
To show completeness, suppose that $\{\xi^k\}_k$ is a Cauchy sequence in $L_{m,n}$. Since $L_{m}$ is complete, the left part of \eqref{ineq.mn} shows that we can find $\xi\in L_m$ such that $\lim_k \xi^k=\xi$ in $L_{m}$.
	For each $\varepsilon>0$, let $M_\varepsilon>0$ be such that
	\begin{equation*}
	\E \left[ \left(\E\Big(|\xi^k-\xi^l|_{\abx}^m\Big|\mathcal F\Big)\right) ^{\frac nm} \right] <\varepsilon \quad\forall k,l\ge M_\varepsilon\,.
	\end{equation*}
	Next, we choose a subsequence $\{l_i\}$ such that $l_i\ge M_\varepsilon$ and $\lim_{i}\xi^{l_i}=\xi$ a.s. Iterated use of (resp. conditional and classical) Fatou's lemma, shows that for each $k\ge M_\varepsilon$,
	\begin{align*}
	\E \left[\left(\E\Big(|\xi^k-\xi|_{\abx}^m\Big|\mathcal F\Big)\right)^{\frac nm} \right] \le \liminf_{i}\E \left[ \left(\E\Big(|\xi^k-\xi^{l_i}|_{\abx}^m\Big|\mathcal F\Big)\right)^{\frac nm} \right] <\varepsilon\,.
	\end{align*}
	Since $\varepsilon$ is arbitrarily small, we conclude that $\lim_k \|\|\xi^k-\xi|\mathcal F\|_m\|_n=0$.  This also shows that $\xi$ belongs to $L_{m,n}$, hence completeness. 

(iii) Iterated use of Fatou's lemma, as above. 

(iv) Thanks to $L_m$-convergence, we know $\xi^{k_j} \rightarrow \xi$ a.s.
along some subsequence. Applying (iii), we have
\[ \| \xi \|_{m, n}\le \liminf_{j } \| \xi^{k_j} \|_{m, n} \le \limsup_j \| \xi^{k_j}
   \|_{m, n} \le \limsup_k \| \xi^{k} \|_{m, n} . \]
\end{proof}


Next, we record a H\"older-type inequality for random variables of mixed integrability, related to a fixed a sub-sigma-field $\cff \subset \cgg$.
\begin{lemma}\label[lemma]{lem.holder_mixed}
  Assume $p, p', p'' \in (0, \infty]$ with $1 / p \ge 1 / p' + 1 / p''$ and
  similar for $q, q', q''$. Then for any two real-valued random variables \( A, B \in L_0 (\cgg)\),
  we have
  \begin{equation}  \label{lem:H1} 
   \| A B \|_{p, q} \le \| A \|_{p', q'} \times \| B \|_{p'', q''} .
  \end{equation}
  In particular, for any $m\in[2,\infty]$, we have 
   \begin{equation}  \label{lem:H2}
   \| \mathbb{E} (A B | \mathcal{F} ) \|_m
     \le \| A \|_m \times \| B \|_{2, \infty} \; ( \le \| A \|_m \times \| B \|_{m, \infty})  
     \end{equation}
  \end{lemma}

\begin{proof}
 W.l.o.g  $A,B \in L^+_0$, so that all mixed moment norm are well-defined with values in $[0,\infty]$.
  Conditional H{\"o}lder inequality gives, in a.s. sense,
  \[ \| A B | \mathcal{F}  \|_p \le \| A | \mathcal{F}
      \|_{p'} \times \| B | \mathcal{F}  \|_{p''} \]
  and the first inequality follows from taking $\| \cdot \|_q$ on both sides,
  followed by another application of H{\"o}lder inequality. The second inequality amounts to the special case  $(p, p', p'') = (1, m, 2)$ and $(q, q', q'') = (m, m, \infty)$.
\end{proof}


\subsubsection{Two parameter stochastic processes with mixed integrability}
Recall that $\{\cff_t\}$ is a filtration on a fixed complete probability space $(\Omega,\cgg,\P)$ and that we denote by \( \W=(\Omega,\cgg,\P;\{\cff_t\}) \).
For computational ease, we introduce the following shorthand notation for the rest of the paper:
	 \begin{equation}
	 \label{not:E_s}
	 \E_s = \E(\,\cdot\,|{\mathcal F_s})\quad \text{for all }s\in [0,T].
	 \end{equation}

In keeping with the previous considerations on random variables, we introduce a space of two parameter stochastic processes as follows.
\begin{definition}\label[definition]{def.C2mn}
Fix a compact interval \( I\subset [0,T] \).
Let $1\le m\le n\le \infty$ and 
\[
C_2L_{m,n}(I,\W;\abx)\quad
\]
 be the space of \( \abx \)-valued, \( 2 \)-parameter stochastic processes $(s,t)\mapsto A_{s,t}$ such that
	\begin{enumerate}[label=(\alph*)]
		\item $A\colon \Omega\times \Delta(I)\to \abx$ is strongly $\mathcal G\otimes\Bor(\Delta(I))/\Bor( \abx)$-measurable,
		\item $A\colon\Delta(I)\to L_m(\abx)$ is continuous, 
		\item\label{itm:C2mn}  $\|A\|_{\infty;{m,n}}:=\sup_{(s,t)\in \Delta(I)}\|\|A_{s,t}|\cff_s\|_{m}\|_{n}<\infty$.
	\end{enumerate}
\end{definition}
For notational ease, we will sometimes abbreviate this space as \( C_2L_{m,n} \).
Clearly, changing the filtration $\{\cff_t\}$ changes the corresponding space. When  $m=n$ however, the choice of filtration makes no difference (see \cref{rem:nothing_happens}). In that case we will contract the two integrability indices and further abbreviate by $C_2L_n:=C_2L_{n,n}$.

Similarly, we introduce further subclasses of such two-parameter stochastic processes as follows.
\begin{definition}\label{def.2amn}
	Fix \( I\subset [0,T] \),
	let $\kappa\in (0,1]$ and $1\le m\le n\le \infty$.
\begin{itemize}
\item  The space $C_2^{\kappa}L_{m,n}(I,\W;\abx)$ consists of two-parameter processes $(A_{s,t})_{(s,t)\in\Delta }$ in $C_2L_{m,n}$ such that
	\begin{align} \label{Akmn} 
	\|A\|_{\kappa;m,n}:=\sup_{s<t\in I}\frac{\|\|A_{s,t}|\cff_s\|_m\|_n}{|t-s|^{\kappa}} <\infty.
	\end{align}
\item 	Similarly, the space
	$C^{\kappa}L_{m,n}(I,\W;\abx)$ contains all stochastic processes $Y\colon \Omega\times I\to \abx$ such that $t\mapsto Y_t$ belongs to $C(I;L_m(\abx))$%
	 and $(s,t)\mapsto \delta Y_{s,t}$ belongs to $C_2^{\kappa}L_{m,n}$.

It is equipped with the norm%
	\begin{equation}  \label{Ykmn} 
	\|Y\|_{\kappa;m,n}:=\sup_{t\in I}\|Y_t\|_{m}+\|\delta Y\|_{\kappa;m,n},
	\end{equation}
which makes it Banach (proof omitted).
\end{itemize}
\end{definition}
We record the process version of  \cref{prop.Lmn}, we will use it later in the contraction argument (proof of Theorem \ref{thm.fixpoint}) with $n=\infty$.

\begin{proposition}\label[proposition]{prop.C2Lmn} Let $1\leq m\leq n\leq \infty$ and  \( \kappa\in(0,1] \). Then:
\begin{itemize}
\item[(i)] $C_2^\kappa L_{m,n}$ is a Banach space.
\item[(ii)] $C_2^\kappa L_{m',n'}\hookrightarrow C_2^{\kappa}L_{m,n}$ for every $m,n,m',n'\in[1,\infty]$  such that $m'\ge m$, $n'\ge n$, $m\le n$ and $m'\le n'$.
\item[(iii)] For each $A\in C_2^\kappa L_{m,n}$ such that $A_{s,t}$ is $\cff_s$-measurable $(s,t)\in \Delta$, then $A\in C_2L_{n,n}$ and $\|A\|_{\kappa;m,n}=\|A\|_{\kappa;n,n}$.
\item[(iv)] Let $A^k \rightarrow A$ in $C_2^{\kappa} L_m$. Then
\[ \|A\|_{\kappa ; m, n} \le \limsup_k \|A^k \|_{\kappa ; m, n} \le
   \infty . \]
\end{itemize}
The properties stated above remain true if each occurrence of \( C_2^\kappa L_{m,n} \) is replaced by \( C_2L_{m,n} \) and \( \|\cdot\|_{\kappa;m,n} \) by \( \|\cdot\|_{\infty;m,n} \).
\end{proposition}
\begin{proof}
The proofs for the first three properties are omitted as they essentially follow \cref{prop.Lmn,rem:nothing_happens}. 
For (iv) we first note that the assumed convergence in $C_2^{\kappa} L_m$, implies $\| A_{s, t}^k - A_{s, t} \|_m \rightarrow 0 $ for any fixed $(s, t)
\in \Delta$. By Proposition \ref{prop.Lmn},
\[   \|\|A_{s,t}|\cff_s\|_m\|_n 
     \le \limsup_k
     \|\|A_{s,t}^k|\cff_s\|_m\|_n
      \le \infty .
\]
Divide both sides by $|t - s|^{\kappa}$, estimate $\|\|A^k_{s,t}|\cff_s\|_m\|_n/
{|t - s|^{\kappa}}  \le \|A^k \|_{\kappa ; m, \infty}$, then take
${\sup_{\Delta}} $on left-hand side to conclude. 
\end{proof}

Lastly, we record an interesting exponential inequality, known as John--Nirenberg inequality.
\begin{proposition}[John--Nirenberg inequality]\label[proposition]{prop.JNineq}
  Let $Y\colon\Omega\times I\to\abx$ be adapted process such that $\delta Y$ belongs to $C^\kappa_2L_{1,\infty}(I,\W;\abx)$ for some $\kappa\in(0,1]$. Assume that $Y$ is a.s.\ continuous.
 Then there are  finite constants $C,c>0$ which are independent from $Y,\kappa,|I|,\W,\abx$ such that
  \begin{align}\label{est.JohnNirenberg}
    \E e^{\lambda\sup_{t\in[0,T]}|\delta Y_{0,t}|_\abx}\le C e^{c(\lambda\|\delta Y\|_{\kappa;1,\infty} )^{1/\kappa}T}
    \quad\text{ for every }\quad\lambda>0.
  \end{align}
\end{proposition}
While the classical John--Nirenberg inequality (see e.g. \cite[Excercise A.3.2]{MR2190038}) implies that $\E e^{\lambda\sup_{t\in[0,T]}|\delta Y_{0,t}|_\abx}$ is finite for some $\lambda>0$, the explicit right-hand side of \eqref{est.JohnNirenberg} follows from a more recent argument from \cite{le2022BMO,le2022maximal}. For the reader's convenience, we include a self-contained proof in \cref{sec:john_nirenberg_inequality}.

\subsection{Stochastic sewing revisited} 
\label{sub.ssl}
The stochastic sewing lemma was introduced in \cite[Theorem 2.1]{le2018stochastic}.
In  \cref{prop.SSL} we provide an extension compatible 
with the mixed $L_{m,n}$-norm required in our analysis and also
show, cf Part \ref{SSL1} below, that stochastic sewing limits are uniform (on compacts) in time.


	\begin{theorem}[Stochastic Sewing Lemma]
	\label{prop.SSL}
		Let $2\le m\le n\le \infty$ be fixed, $m<\infty$.
		Let $A=(A_{s,t})_{(s,t)\in \Delta}$ be a stochastic process in $W$ such that $A_{s,s}=0$ and $A_{s,t}$ is $\cff_t$-measurable for every $(s,t)\in \Delta$.
		\begin{enumerate}[label=(\roman*)]
		\item\label{SSL1}
		 Suppose that there are finite constants $\Gamma _1, \Gamma_2\ge0$ and $\varepsilon _1,\varepsilon _2>0$ such that for any $(s,u,t)\in\Deltatwo$,
		\begin{gather}\label{con.dA1}
		\|\E_s[\delta A_{s,u,t}]\|_n\leq \Gamma_1(t-s)^{1+\varepsilon_1}
		\\\shortintertext{and}\label{con.dA2}
			\|\|\delta A_{s,u,t}|\cff_s\|_{m}\|_n\leq \Gamma _2(t-s)^{\frac12+\varepsilon _2}\,.
		\end{gather}

		Then, there exists a unique stochastic process $\A$ with values in $W$ satisfying the following properties
		\begin{itemize}
			\item 
			 $\A_0=0$, $\A$ is $\{\cff_t\}$-adapted, $\A_t-A_{0,t}$ is $L_m$-integrable for each $t\in[0,T]$;
			\item 
			 there are positive constants $C_1=C_1(\varepsilon_1),C_2=C_2(\varepsilon_2)$ such that for every $(s,t)\in \Delta$,
			 \begin{equation}\label{est.A1}
			 	\left \|\left \|\A_t-\A_s-A_{s,t}|\cff_s\right \|_{m}\right \|_n\leq C_1 \Gamma_1(t-s)^{1+\varepsilon _1} +C_2 \Gamma_2(t-s)^{\frac12+\varepsilon _2}\,
			 \end{equation}
			 and
			\begin{equation}\label{est.A2}
				\|\E_s(\A_t-\A_s-A_{s,t})\|_n\leq C_1 \Gamma_1(t-s)^{1+\varepsilon _1}\,.
			\end{equation}
		\end{itemize}

		\item\label{SSL2}
		 Suppose furthermore that for each $s\in[0,T]$, the map $t\mapsto A_{s,t}$ is a.s.\ c\`adl\`ag (resp.\ continuous) on $[s,T]$, and  there are finite constants $\varepsilon_3>0$ and $\Gamma_3\ge0$ such that for any $(s,t)\in \Delta$
		\begin{gather}\label{con.dA3a}
			\bigg\|\Big\|\sup_{u\in[(s+t)/2,t]}|\delta A_{s,(s+t)/2,u}|\Big|\cff_s\Big\|_m\bigg\|_n\le \Gamma_3(t-s)^{\frac 1m+\varepsilon_3}
			\quad\forall(s,t)\in \Delta.
		\end{gather}
		Let $\cpp=\{0=t_0<t_1<\cdots<t_N=T\}$ be a partition of $[0,T]$ and define for each $t\in[0,T]$,
		\begin{equation} \label{Acpp}
			A^\cpp_{t}:=\sum\nolimits_{i:t_i\le t}A_{t_i,t_{i+1}\wedge t}\,.
		\end{equation}
		Then $\A$ has a c\`adl\`ag (resp.\ continuous) version, denoted by the same notation, and for this version, we have
		\begin{align}\label{est.unifAPA}
			\bigg\|\Big\|\sup_{t\in[0,T]}|A^\cpp_t-\A_t|\Big|\cff_0\Big\|_m\bigg\|_n
		\le C |\cpp|^{\varepsilon_1\wedge \varepsilon_2 \wedge \varepsilon_3} (\Gamma_1+\Gamma_2+\Gamma_3)
		\end{align}
		where $|\cpp|:=\sup_{i}|t_i-t_{i+1}|$ and $C$ is some constant depending on $T, m, \varepsilon_1, \varepsilon_2, \varepsilon_3$.
	\end{enumerate}
	\end{theorem}
	
	
		\subsubsection*{Proof of \cref{prop.SSL}}   
	Part \ref{SSL1}.
	We sketch the arguments for the convergence of $A^{\cpp^k}_T$ along the sequence of dyadic partitions $\cpp^k= \{t^k_i:i=0,\ldots,2^k\}$ of $[0,T]$. The argument is a variation 
	of \cite{le2018stochastic} with the key feature of handling mixed moments, $m \le n$. (For a stream-lined presentation and further generalizations to Banach  spaces see \cite{LeSSL2}.) 
	With  $Z^k_i=\delta A_{t^{k+1}_{2i},t^{k+1}_{2i+1},t^{k+1}_{2i+2}}$, we can write
	\begin{align*}
		A^{\cpp^k}_T-A^{\cpp^{k+1}}_T=\sum_{i=0}^{2^{k}-1}Z^k_i 
		=\sum_{i=0}^{2^{k}-1}\E_{t^k_i}Z^k_i+\sum_{i=0}^{2^{k}-1}(Z^k_i-\E_{t^k_i}Z^k_i).
	\end{align*}
	The former sum is estimated using \eqref{con.dA1} and the latter one using \eqref{con.dA1}-\eqref{con.dA2} in combination with a conditional Burkholder--Davis--Gundy inequality.\footnote{To wit, if  $\{ M_k : k \ge 0\}$ is a discrete martingale relative to some filtration $\{ \mathcal{G}_k \}$, then a $\mathcal{G}^{\circ}$-conditional BDG inequality holds provided
	$\mathcal{G}^{\circ} \subset  \mathcal{G}_0$, as seen by considering the martingales $M_k \mathbb{I}_G$, for $G \in \mathcal{G}^{\circ}$.} 	 This yields
	\begin{align*}
		&\|\|A^{\cpp^k}_T-A^{\cpp^{k+1}}_T|\cff_0\|_{m}\|_n\lesssim (\Gamma_1 T^{1+\varepsilon_1} +\Gamma_2 T^{\frac12+\varepsilon2}) 2^{-k(\varepsilon_1\wedge\varepsilon_2)},
		\\
		&\|\E_0(A^{\cpp^k}_T-A^{\cpp^{k+1}}_T)\|_n\lesssim \Gamma_1 T^{1+\varepsilon_1}  2^{-k\varepsilon_1}\,.
	\end{align*}
	This shows the existence of some limit $\A_T$ with $\A_T - A^{\cpp^k}_T \to 0$ in $L_m$. 
	We then write $\A_T-\A_0-A_{0,T}=\sum_{k=0}^\infty (A^{\cpp^{k+1}}_T-A^{\cpp^{k}}_T)$ and apply the above estimates to obtain \eqref{est.A1} and \eqref{est.A2} for $(s,t)=(0,T)$.
	A dyadic allocation argument, used in an essentially similar fashion as in \cite{le2018stochastic}, then provides the extension from dyadic to general partitions and yields (i) as stated. 
	
	The proof of Part \ref{SSL2} is given after Lemmas \ref{lem.AAonP}. \ref{lem.unifAA}, \ref{lem.AAonD}. We first give a corollary, useful for existence of weak RSDE solutions (cf. Lemma \ref{lem.lim}).
		

	\begin{corollary}\label[corollary]{cor.SSLk}
		Let $\{ (A^k_{s,t})_{(s,t)\in \Delta} : k \in \mathbb{N} \}$ be a family of stochastic processes that satisfies the hypotheses of \cref{prop.SSL} (ii), \reply{with respect to filtrations $\{(\cff^k_t)_{t}: k\in\mathbb{N}\}$ respectively,} with exponents and constants ($m,n,\Gamma_1,\Gamma_2,\Gamma_3,\varepsilon_1,\varepsilon_2,\varepsilon_3$) uniform in $k$,  and let $\A^k$ be the corresponding process. Suppose that for each $s\in[0,T]$, $\lim_{k}\|\sup_{t\in[s,T]}|A^k_{s,t}-A_{s,t}|\|_m=0$. Then $\lim_k\|\sup_{t\in[0,T]}|\A^k_t-\A_t|\|_m =0$.
	\end{corollary}
	\begin{proof}
		Let $\cpp$ be a partition of $[0,T]$. From \cref{prop.SSL}(ii), we have, for some $\varepsilon > 0$,
		\begin{align*}
			\sup_k \left \|\sup_{t\in[0,T]}\bigg|\A^k_t-\A_t-\sum_{[u,v]\in\cpp:u\le t}(A^k_{u,v\wedge t}-A_{u,v\wedge t})\bigg|\right\|_m\lesssim|\cpp|^{\varepsilon} .
		\end{align*}
		By assumption, we have
		\[
			\lim_{k}\left \|\sup_{t\in[0,T]}\bigg|\sum_{[u,v]\in\cpp:u\le t}(A^k_{u,v\wedge t}-A_{u,v\wedge t})\bigg|\right \|_m =0.
		\]
		By triangle inequality, we have
		\begin{align*}
			\Big \|\sup_{t\in[0,T]}\big|\A^k_t-\A_t\big|\Big \|_m\lesssim |\cpp|^\varepsilon+ \Big \|\sup_{t\in[0,T]}\Big|\sum_{[u,v]\in\cpp:u\le t}(A^k_{u,v\wedge t}-A_{u,v\wedge t})\Big|\Big \|_m
		\end{align*}
		From here, we send first $k\to\infty$ then $|\cpp|\to0$ to obtain the result.
	\end{proof}

	\subsubsection*{Preparatory lemmas}
        For the proof of \cref{prop.SSL}(ii), we prepare a few intermediate estimates. 
        All implicit constants in the following depend only on $m,\varepsilon_1,\varepsilon_2$ and $T$.
	\begin{lemma}\label[lemma]{lem.AAonP} In the setting of \cref{prop.SSL}(i), we have $$\|\|\sup_{t\in \cpp}|A^{\cpp}_{t}-\A_t||\cff_0 \|_m\|_n\lesssim |\cpp|^{\varepsilon_1\wedge \varepsilon_2}(\Gamma_1+\Gamma_2).$$
	\end{lemma}
	\begin{proof}
		For  each $j\ge1$, we write
		\begin{align*}
			A^{\cpp}_{t_j}-\caa_{t_j}
			=\sum_{i\le j} Z_i
			=\sum_{i\le j}( Z_i-\E_{t_{i-1}}Z_i)+\sum_{i\le j}\E_{t_{i-1}}Z_i,
		\end{align*}
		where for each $i$, $Z_i=A_{t_{i-1},t_i}- \delta\caa_{t_{i-1},t_i}\in\cff_{t_i}$. Note that the former sum is a discrete martingale indexed by $j$. Applying the conditional BDG inequality and the Minkowski inequality (see \cite[Eq. (2.5)]{le2018stochastic}), we have
		\begin{align*}
			\|\|\sup_{j}|A^{\cpp}_{t_j}-\caa_{t_j}||\cff_0\|_m\|_n
			\lesssim \left(\sum_{i}\|\|Z_i|\cff_0\|_m\|_n^2\right)^{1/2}+\sum_{i}\|\|\E_{t_{i-1}}Z_i|\cff_0\|_m\|_n.
		\end{align*}
		Using \eqref{est.A1} and \eqref{est.A2}, we can estimate the series on the right-hand side above, which yields the stated estimate.
	\end{proof}
	\begin{lemma}\label[lemma]{lem.unifAA} Under the assumptions of \cref{prop.SSL}(ii), for every $(s,t)\in \Delta$, we have
		\begin{align}\label{est.unif.AA}
			\Big\|\Big\|\sup_{r\in D(s,t) }|\delta\A_{s,r}-A_{s,r}|\Big|\cff_s\Big\|_m\Big\|_n\lesssim \Gamma_1(t-s)^{1+\varepsilon_1}+\Gamma_2(t-s)^{\frac12+\varepsilon_2}+\Gamma_3(t-s)^{\frac1m+\varepsilon_3} ,
		\end{align}
		where  $D(s,t)=\cup_k\cpp^k$, 
		 $\cpp^k$ the dyadic partitions of $[s,t]$, with uniform mesh size $2^{-k}(t-s)$.

	\end{lemma}
	\begin{proof}
%
		\reply{Fix $(s,t)\in \Delta$. Let $\cpp^k=\{r^k_i\}_{i=0}^{2^{k+1}-1}$ be the dyadic partition of $[s,t]$ with uniform mesh size $2^{-k}(t-s)$. Define  $\fl{r}_k=\sup\{\rho\in\cpp^k:\ \rho\le r\}$,  $A^k_{s,r}=\sum_{i:r^k_i\le r}A_{r^k_{i-1},r^k_i}$, $u^k_i=(r^k_i+r^k_{i-1})/2$ and $Z^k_i=-\delta A_{r^k_{i-1},u^k_i,r^k_i}$.
		For each $r\in[s,t]$, each integer $ h\ge1$, we write 
		}
		\begin{gather*}
			A^h_{s,\fl{r}_h}-A^{h-1}_{s,\fl{r}_{h-1}}=\sum_{i:r^{h-1}_i\le\fl{r}_{h-1}}Z^{h-1}_i+A_{\fl{r}_{h-1},\fl{r}_h},
		\end{gather*}
		so that \reply{for each $k\ge1$}
		\begin{align}\label{tmp.id.disAk}
			A^k_{s,\fl r_k}-A_{s,\fl r_k}=\sum_{h=0}^{k-1}\sum_{i:r^{h}_i\le\fl{r}_{h}}Z^{h}_i+\left(\sum_{h=0}^{k-1}A_{\fl{r}_{h},\fl{r}_{h+1}}-A_{s,\fl r_k}\right).
		\end{align}
		For the former sum, we further decompose
		\begin{align*}
			\sum_{h=0}^{k-1}\sum_{i:r^{h}_i\le\fl{r}_{h}}Z^{h}_i=\sum_{h=0}^{k-1} \left[I^{h}_1(\fl r_{h})+I^{h}_2(\fl r_{h})\right]
		\end{align*}
		where
		\begin{gather*}
			I^h_1(r)=\sum_{i:r^h_i\le r} \E_{r^h_{i-1}} Z^h_i
			\tand I^h_2(r)=\sum_{i:r^h_i\le r}( Z^h_i-\E_{r^h_{i-1}}Z^h_i).
		\end{gather*}
		Using triangle inequality, we have
		\begin{align*}
			\|\|\sup_{r\in[s,t]}|I^h_1(\fl r_h)||\cff_s\|_m\|_n\le\sum_{i:r^h_i\le t} \|\| \E_{r^h_{i-1}} Z^h_i|\cff_s\|_m\|_n\le \Gamma_1 (t-s)^{1+\varepsilon_1}2^{-h \varepsilon_1}.
		\end{align*}
		Note that for each $h$, $(I^h_2(u))_{u\in\cpp^h}$ is a discrete martingale. Applying the BDG inequality and the Minkowski inequality, we have for every $h<k$,
		\begin{align*}
			\|\|\sup_{r\in[s,t]}|I^h_2(\fl r_h)||\cff_s\|_m\|_n
			&\le \|\|\sup_{u\in\cpp^h}|I^h_2(u)||\cff_s\|_m\|_n
			\\&\lesssim \left(\sum_{i:r^h_i\le t}\|\|Z^h_i|\cff_s\|_m\|_n^2 \right)^{1/2}
			\lesssim \Gamma_2(t-s)^{\frac12+\varepsilon_2}2^{-h \varepsilon_2}.
		\end{align*}

 		For the later sum in \eqref{tmp.id.disAk}, we have
		\begin{align*}
			I^h_3(r):= |\sum_{h=0}^{k-1}A_{\fl r_h,\fl r_{h+1}}-A_{s,\fl r_k}|
			&=|\sum_{h=0}^{k-1}\delta A_{\fl r_h,\fl r_{h+1},\fl r_{k}}|
			\\&\le \sum_{h=0}^{k-1}\sup_{i=0,\ldots,2^h-1}\sup_{u\in[u^h_i,r^h_{i+1}]}|\delta A_{r^h_i,u^h_i,u}|.
		\end{align*}
		We note that
		\begin{align*}
			\E_s\sup_{i=0,\ldots,2^h-1}\sup_{u\in[u^h_i,r^h_{i+1}]}|\delta A_{r^h_i,u^h_i,u}|^m
			\le \sum_{i=0}^{2^h-1}\E_s\sup_{u\in[u^h_i,r^h_{i+1}]}|\delta A_{r^h_i,u^h_i,u}|^m.
		\end{align*}
		Hence using \eqref{con.dA3a} and the fact that $m\le n$, we obtain that
		\begin{align*}
			\bigg\|\Big\|\sup_{i}\sup_{u\in[u^h_i,r^h_{i+1}]}|\delta A_{r^h_i,u^h_i,u}|\Big|\cff_s\Big\|_m\bigg\|_n
			&\le \left(\sum_{i=0}^{2^h-1}\bigg\|\Big\|\sup_{u\in[u^h_i,r^h_{i+1}]}|\delta A_{r^h_i,u^h_i,u}|\Big|\cff_s\Big\|_m\bigg\|_n^m \right)^{\frac1m}
			\\&\lesssim \Gamma_3 2^{-h \varepsilon_3}(t-s)^{\frac1m+\varepsilon_3},
		\end{align*}
		which implies that
		\begin{align*}
			\|\|\sup\nolimits_rI^h_3(r)|\cff_s\|_m\|_n\lesssim \Gamma_3(t-s)^{\frac 1m+\varepsilon_3}.
		\end{align*}

		Applying the previous estimates altogether in \eqref{tmp.id.disAk}, we obtain that
		\begin{align*}
			\bigg\|\Big\|\sup_{r\in[s,t]}|A^k_{s,\fl r_k}-A_{s,\fl r_k}|\Big|\cff_s\Big\|_m\bigg\|_n
			\lesssim \Gamma_1 (t-s)^{1+\varepsilon_1}+\Gamma_2(t-s)^{\frac12+\varepsilon_2}
			+ \Gamma_3(t-s)^{\frac1m+\varepsilon_3}
		\end{align*}
		uniformly for every $k\ge1$.
		Hence, for every $k\le l$,
		\begin{align*}
			\bigg\|\Big\|\sup_{r\in\cpp^k}|A^l_{s,r}-A_{s,r}|\,\Big|\cff_s\Big\|_m\bigg\|_n
			\lesssim \Gamma_1 (t-s)^{1+\varepsilon_1}+\Gamma_2(t-s)^{\frac12+\varepsilon_2}
			+ \Gamma_3(t-s)^{\frac1m+\varepsilon_3}.
		\end{align*}
		Sending $l\to\infty$ then $k\to\infty$, using the fact that $\lim A^l_{s,r}=\delta \A_{s,r}$ in probability for each $s,r$, since we are in the setting of part (i),
		cf. Proof of \cref{prop.SSL} below, 	
		we obtain \eqref{est.unif.AA}.
	\end{proof}
	\begin{lemma}\label{lem.AAonD}
		Then under the assumptions of \cref{prop.SSL}(ii), we have
		\begin{align*}
			\Big\|\sup_{t\in D}|A^{\cpp^k}_t-\A_t|\Big\|_m\lesssim 2^{-k (\varepsilon_1\wedge \varepsilon_2 \wedge \varepsilon_3)} (\Gamma_1+\Gamma_2+ \Gamma_3),
		\end{align*}
		where  $D=\cup_k\cpp^k$, 
		$\cpp^k$ the dyadic partitions of $[0,T]$, with uniform mesh size $2^{-k}T$.
		\end{lemma}

	
	\begin{proof}
		For any $t$, define $\fl{t}_k=\sup\{r\in\cpp^k:r\le t\}$.
		For any $l\ge k$, we have
		\begin{align*}
			\sup_{t\in \cpp^l}|A^{\cpp^k}_t-\A_t|
			&\le \sup_{t\in \cpp^l}|A^{\cpp^k}_{\fl{t}_k}-\A_{\fl{t}_k}|+\sup_{t\in \cpp^l}| A_{\fl{t}_k,t}-\delta\A_{\fl{t}_k,t}|
			\\&= \sup_{t\in \cpp^k}|A^{\cpp^k}_t-\A_t|+\sup_{s\in\cpp^k}\sup_{t\in \cpp^l\cap[s,s+2^{-k}T]}| A_{s,t}-\delta\A_{s,t}|.
		\end{align*}
		By \cref{lem.AAonP}, we have $\|\sup_{t\in \cpp^k}|A^{\cpp^k}_t-\A_t|\|_m\lesssim 2^{-k (\varepsilon_1\wedge \varepsilon_2)}(\Gamma_1+\Gamma_2)$.
		For the second term, we put $\zeta_s=\sup_{t\in \cpp^l\cap[s,s+2^{-k}T]}| A_{s,t}-\delta\A_{s,t}|$ and use \cref{lem.unifAA} to obtain that
		\begin{align*}
			\E|\sup_{s\in\cpp^k}\zeta_s |^m\le\sum_{s\in\cpp^k} \E|\zeta_s|^m
			\lesssim  2^{-k\min\left\{(1+\varepsilon_1)m-1,(\frac12+\varepsilon_2)m-1,m \varepsilon_3\right\} }(\Gamma_1+\Gamma_2+\Gamma_3)^m.
		\end{align*}
		Since $m\ge2$, the above exponent of $2^{-k}$ is positive, hence we have $\|\sup_{s\in\cpp^k}\zeta_s \|_m\lesssim 2^{-k( \varepsilon_1\wedge \varepsilon_2\wedge \varepsilon_3)} (\Gamma_1+\Gamma_2+\Gamma_3)$.
		These estimates yield
		\begin{align*}
			\Big\|\sup_{t\in \cpp^l}|A^{\cpp^k}_t-\A_t|\Big\|_m\lesssim 2^{-k (\varepsilon_1\wedge \varepsilon_2\wedge \varepsilon_3)} (\Gamma_1+\Gamma_2+\Gamma_3).
		\end{align*}
		Since $l\ge k$ is arbitrary, this implies the result.
	\end{proof}

 	\subsubsection*{Proof of \cref{prop.SSL}} (Part \ref{SSL2}, locally uniform convergence) By Part \ref{SSL1} we have pointwise convergence,
	that is $A_t^{\mathcal{P}^k} \rightarrow \mathcal{A}_t$ in probability, $t \in
[0, T]$. By Lemma \ref{lem.AAonD} and the triangle inequality it is clear that
\[ \| \sup_{t \in D} | A_t^{\mathcal{P}^k} - A_t^{\mathcal{P}^l} | \|_m
   \rightarrow 0, \tmop{as} k, l \rightarrow \infty ; \]
by the assumed c{\`a}dl{\`a}g (continuity) assumption, one replaces $D$ by
$\bar{D} = [0, T]$. With a Cauchy argument we see that that there is a
c{\`a}dl{\`a}g (continuous) process $\tilde{\mathcal{A}} $ so that
\[ \| \sup_{t \in [0,T]} | A_t^{\mathcal{P}^k} - {\tilde{\mathcal{A}}_t} | \|_m
   \rightarrow 0, \tmop{as} k, l \rightarrow \infty . \]
Clearly then, $\mathcal{A}_t = \tilde{\mathcal{A}}_t$ a.s. and for each $t \in
[0, T]$, so that $\tilde{\mathcal{A}} $ is the desired c{\`a}dl{\`a}g
(continuous) version. 
	\hfill\qed

\section{Rough stochastic analysis}
\label{sec.RSA}
In the current section, we define and establish basic properties of the integration $\int Zd\X$ where $Z$ is an adapted process and $\X=(X,\XX)$ is an $\alpha$-H\"older rough path.
The frequently used class of \( X \)-controlled rough paths in rough-path theory (\cite{FH20}) turns out to be too restrictive to contain solutions to RSDEs (as introduced later in \eqref{eqn.srde}). This has led us to the concept of stochastic controlled rough paths (introduced in \cref{def.SCRP}) and rough stochastic integrations, which are described herein.
Throughout the section, $\W=(\Omega,\cgg,\P;\{\cff_t\})$ is a stochastic basis whose underlying probability space is complete. We assume, as usual, that $\cff_0$ contains $\P$-null sets, which ensures that modifications of adapted processes are still adapted.

\subsection{Stochastic controlled rough paths} 
\label{sub.regular_stochastic_controlled_rough_paths}
In the sequel we let $2\le m<\infty ,$ while $m\le n\le \infty$.
The parameters  $\alpha,\beta,\beta'\in(0,1]$ are subject to \( \alpha +\beta+\beta'>1\).
Unless stated otherwise,
\( I\subset [0,T] \) is an arbitrary compact interval and we let for convenience
\[
o=\min I.
\]

For any 2-parameter stochastic process \( A_{s,t}(\omega)\), we introduce the quantity
	\begin{equation}
	\label{shorthand_bigcdot}
	\E_\bigcdot A=
(s,t;\omega) \mapsto \E_s (A_{s,t} )(\omega)  ,
	\end{equation}
where we recall that \( \E_s=\E(\cdot|\mathcal F_s) \).
A progressively measurable stochastic process \( Z_s(\omega)\in W\) is called \textit{stochastically controlled} with respect to \( X \) if there exists another such \( Z_s'(\omega)\in \mathcal L(V,W) \) so that for every \( (s,t)\in \Delta(I) \), with probability one
\begin{equation}
\label{stochastic_controlled}
|\E_s\delta Z_{s,t}-Z_s'\delta X_{s,t}| \le C_{s,t}|t-s|^{\beta+\beta'}\,,
\end{equation}
where \( C_{s,t}=C_{s,t}(\omega)\) is a uniformly \( L_{n}\)-integrable, two-parameter family of random variables.
By a common abuse of language, we call \( Z' \) the (generalized) Gubinelli derivative of \( Z \) even though it is not unique in general.

It turns out that rough stochastic integrals can be defined for stochastically controlled processes that are subject to additional regularity assumptions in the spaces \( C^\kappa L_{m,n}\). One of these subclasses is that of stochastic controlled rough paths, as defined here. As will be seen in \cref{sec.rough_stochastic_differential_equations}, it contains solutions to RSDEs of the form \eqref{eqn.srde} for reasonable coefficients.
	\begin{definition}[Stochastic controlled rough paths]
	\label[definition]{def.SCRP}
		We say that $(Z,Z')$ is a stochastic controlled rough path of $(m,n)$-integrability and $(\beta,\beta')$-H\"older regularity with values in $W$ with respect to $\{\cff_t\}$
		if the following are satisfied
		\begin{enumerate}[label=(\alph*)]
		 	\item\label{good1} $Z\colon \Omega\times  I\to W$ and $Z'\colon\Omega \times I\to \mathcal L(V,W)$
		 	are $\{\mathcal F_t\}$-progressively measurable;
			\item\label{goodgood}
			$\delta Z$ belongs to $C^{\beta}_2 L_{m,n}( I,\W;\cll(V,W))$;
		 	\item \label{itm:remainder} \( Z \) is stochastically controlled with Gubinelli derivative \( Z' \). Said otherwise, putting
		 	\begin{equation*}
		 	\begin{aligned}
		 	R^Z_{s,t}
		 	&=\delta Z_{s,t}-Z'_s \delta X_{s,t}
		 	,\quad\text{for}\enskip  (s,t)\in \Delta ,
			\end{aligned}
		 	\end{equation*}
		 	we have that $\E_{\bigcdot} R^Z$ belongs to $C^{\beta+\beta'}_2L_{n}( I,\W;W) $;%
		 	\footnote{In that case \cref{rem:nothing_happens} implies that \( \|\EE R^Z\|_{\beta+\beta';n}=\|\EE R^Z\|_{\beta+\beta';m,n} \).}
			\item\label{good2}
			$Z'$ belongs to $C^{\beta'} L_{m,n}( I,\W;\cll(V,W))$
			and \(\sup_{t\in I}\|Z'_t\|_n<\infty\).
		\end{enumerate}
		The class of such processes is denoted by $\D_X^{\beta,\beta'}L_{m,n}( I,\W;W)$, or simply
	 $\D_X^{\beta,\beta'}L_{m,n}$ whenever clear from the context. Additionally, we write $\D_X^{2 \beta}L_{m,n}=\D_X^{\beta,\beta}L_{m,n}$.
	\end{definition}
	Being stochastically controlled is fundamentally a statement about the increments \( \delta Z \) rather than the path \( Z \) itself. No integrability assumption is required on the ground value \( Z_o\in L_0(\mathcal F_o) \), as long as property \eqref{stochastic_controlled} is not altered. 

	For a process $(Z,Z')$ in $\D_X^{\beta,\beta'}L_{m,n}$, we define the  seminorms
	\begin{equation}\label{def.bracketDL}
		\bk{(Z,Z')}_{X,\beta,\beta';m,n} =	
			\|\delta Z\|_{\beta;m,n}  
		+ \| \delta Z'\|_{\beta';m,n} + \|\E_{\bigcdot} R^Z\|_{\beta+\beta';n},
\end{equation}
and 
	\begin{equation}\label{def.normDbar}
		\|(Z,Z')\|_{X,\beta,\beta';m,n}
		= \bk{(Z,Z')}_{X,\beta,\beta';m,n}  + \| Z' \|_{\infty ; n} .
	\end{equation}
This is not a norm as it assigns zero to any $(Z,Z') \equiv (z,0)$, any $z\in W$, accordingly does not induce a metric but only a pseudometric on $\D_X^{\beta,\beta'}L_{m,n}$. Although 
$(\bar Z,\bar Z') \in  \D^{\beta,\beta'}_{\bar X}L_{m,n}$, for $\bar X \ne X$, lives in a different space, we can define a meaningful distance\footnote{Because of Proposition \ref{prop.compose_stab} we will use this distance only in case $m=n$.}
%
	\begin{equation}
	\label{def.scrp.bracket}
	\begin{aligned}
		\bk{Z,Z';\bar Z,\bar Z'} _{X,\bar X;\beta,\beta';m,n}
		&= \|\delta Z- \delta \bar Z\|_{\beta;m,n}
	+\|\delta Z'- \delta \bar Z'\|_{\beta';m,n}+\|\E_\bigcdot R^{Z}-\E_\bigcdot \bar R^{\bar Z}\|_{\beta+\beta';n},
	\end{aligned}
	\end{equation}
	where $\bar R^{\bar Z}_{s,t}:=\delta\bar Z_{s,t}-\bar Z'_s\delta\bar X_{s,t}$, and then
	\begin{equation}
	\label{def.scrp.metric}
		\|Z,Z';\bar Z,\bar Z'\| _{X,\bar X;\beta,\beta';m,n} = \bk{Z,Z';\bar Z,\bar Z'}_{X,\bar X;\beta,\beta';m,n} +  \| Z' - \bar Z' \|_{\infty ; n}.
	\end{equation}
	As previously, the subscript $(X,\bar X;\beta,\beta';m,n)$ will be notationally condensed in case of $X=\bar X, \beta = \beta'$ or $m=n$, in which case we simply write $X,\beta$ or $m$, respecively, rather than repeating the concerned variables. The subadditivity property for \eqref{def.normDbar}, with induced triangle inequality, then extends to 
\begin{multline}
\label{ineq.super_additive}
\|Z+S,Z'+S';\bar Z+S,\bar Z'+S'\|_{X,\bar X;\beta, \beta';m,n}
\\\le\|Z,Z';\bar Z,\bar Z'\|_{X,\bar X;\beta, \beta';m,n}+\|S,S';\bar S,\bar S'\|_{X,\bar X;\beta, \beta';m,n},
\end{multline}
for any \( (Z,Z'), (S,S')\in \D_{X}^{\beta,\beta'}L_{m,n}\) and \((\bar Z,\bar Z'), (\bar S,\bar S')\in  \D^{\beta,\beta'}_{\bar X}L_{m,n} \), and 
\begin{equation}
\label{triangle_inequality}
\|Z,Z';\bar Z,\bar Z'\|_{X,\bar X;\beta, \beta';m,n}
\le \|Z,Z';\tilde Z,\tilde Z'\|_{X,\tilde X;\beta, \beta';m,n}
+\|\tilde Z,\tilde Z';\bar Z,\bar Z'\|_{\tilde X,\bar X;\beta,\beta';m,n},
\end{equation}
for any \( \tilde X \in C^\alpha(V)\) and  \( (\tilde Z,\tilde Z')\in \D_{\tilde X}^{\beta,\beta'}L_{m,n}\). Similar inequalities hold for the brackets \eqref{def.scrp.bracket}.

\begin{remark}
\label[remark]{rem:martingales}
Martingales ``have'' zero generalized Gubinelli derivative,
in the sense that letting \( M_s'=0 \) yields \( \E_sR^M_{s,t}=\E_s\delta M_{s,t}\equiv 0\), which is in \(C^{\beta'}_2L_n \).
Additionally \( (M,0) \) forms a stochastic controlled rough path in \( \D_X^{\beta,\beta'}L_{m,n} \) if and only if \( t\mapsto \delta M_{0,t} \) lies in \( C^\beta L_{m,n} \) (with identical semi-norms). This is clear from \cref{def.SCRP}.
\end{remark}

\subsubsection*{Doob--Meyer decomposition}
Other examples of stochastic controlled rough paths would be \textit{$X$-controlled} rough paths taking values in the Banach space \( \abx=L_m \), as defined in \cite{FH20}  (denoted by \(\mathscr D_X^{\beta,\beta'}(\abx)\) therein). 
Recall that for two continuous paths \( Y\colon  I\to \abx \) and \( Y'\colon  I\to \mathcal L(V,\abx) \), the pair \( (Y,Y') \) is called \( X \)-controlled whenever
\begin{equation}
\label{strong_cancellation}
|\delta Y_{s,t}-Y'_s\delta X_{s,t}|_{\abx}\lesssim (t-s)^{\beta+\beta'}\,.
\end{equation}
Now, because of the contraction property for conditional expectation, we have that any progressively measurable pair $(Y,Y')$, which is also in \( \mathscr D_X^{\beta,\beta'}(L_m)\), in fact yields an element in \(\D_X^{\beta,\beta'} L_{m,m}\).%

Under mild regularity and integrability assumptions, it is clear from \cref{rem:martingales} and linearity  that the sum of a martingale and an \( X \)-controlled process 
will yield a stochastic controlled rough path as in \cref{def.SCRP}.
Similar to Doob--Meyer, we show a converse statement. It asserts that a stochastic controlled rough path can be written as a martingale (endowed with zero Gubinelli derivative as above) plus an additional path subject to \eqref{strong_cancellation}.
This is formulated in the following result.
	\begin{theorem}[Doob--Meyer decomposition]\label{thm.mj}
		Suppose that $(Z,Z')$ is a stochastic controlled rough path with respect to $\{\cff_t\}$ in $\D_X^{\beta,\beta'}L_{m,n}$ with $(\alpha\wedge\beta)+\beta'>\frac12$, $m\in[2,\infty)$ and $n\in[m,\infty]$.
		Then, there are uniquely characterized processes $M,Y$ such that
		\begin{enumerate}[label=(\roman*)]
			\item\label{mj1} $Z_t=M_t+Y_t$ a.s.\ for every $t\in I$;
			\item $M$ is an $\{\cff_t\}$-martingale, $M_o=0$;
			\item\label{mj3} $Y$ is $\{\cff_t\}$-adapted and \( X \)-controlled in the sense that
			\begin{equation}\label{est.jz}
				\|\|\delta Y_{s,t}-Z'_s \delta X_{s,t}|\cff_s\|_m\|_n\lesssim(\|\delta  Z'\|_{\beta';m,n}|\delta X|_\alpha+\|\E_\bigcdot R^Z\|_{\beta+\beta';n}) |t-s|^{(\alpha\wedge \beta)+\beta'}
			\end{equation}
			for every $(s,t)\in \Delta$.
		\end{enumerate}
	  Letting \( Y'=Z' \), estimate \eqref{est.jz} implies moreover that $(Y,Y')$ belongs to $\D_X^{ (\alpha\wedge \beta),\beta'}L_{m,n}$ (in particular, $M$ belongs to $C^{\alpha\wedge \beta}L_{m,n}$).
	\end{theorem}

	\begin{proof} This can be seen as application of the Doob-Meyer type decomposition Theorem 2.2 in \cite{le2018stochastic}, but in the context of mixed moments, $m \le n$, also revisited in \cite{LeSSL2}. Using Theorem 3.3 therein, consider 
		 $A_{s,t}=\delta Z_{s,t}$, which is integrable. Since $\delta A\equiv0$, conditions \eqref{con.dA1} and \eqref{con.dA2} of \cref{prop.SSL} are trivially satisfied, and $\caa_t=Z_t-Z_o$ for each $t$. From the definition of the spaces $\D_X^{\beta,\beta'}L_{m,n}$ and the fact that $Z'_s$ is $\cff_s$-measurable, we have
		\begin{align*}
		 (\E_s-\E_u)\delta Z_{u,t}
		 &=(\E_s-\E_u)[Z'_{u}\delta X_{u,t}]+(\E_s-\E_u)R^Z_{u,t}
		 \\&=(\E_s-\E_u)[\delta Z'_{s,u}]\delta X_{u,t}+(\E_s-\E_u)R^Z_{u,t}\,.
		\end{align*}
		Hence, we infer that
		\[
		\|\|(\E_s-\E_u )\delta Z_{u,t}|\cff_s\|_{m}\|_n\lesssim(\|\delta  Z'\|_{\beta';m,n}|\delta X|_\alpha+\|\E_\bigcdot R^Z\|_{\beta+\beta';n}) |t-s|^{(\alpha\wedge \beta)+\beta'}\,.
		\]
		But since $(\alpha\wedge \beta)+\beta'>\frac12$, the conditions of \cite[Theorem 3.3]{LeSSL2} are met. Hence, $\caa=\cmm+\cjj$ where $\cmm,\cjj$ satisfy the conclusions of  that theorem. 
		We set $M=\cmm$ and $Y=\cjj+Z_o$ so that $Z=M+Y$. We see from \cite[Eqn. (3.9)]{LeSSL2} that
		\[
			\|\| \delta Y_{s,t}-\E_s \delta Z_{s,t}|\cff_s\|_m\|_n\lesssim(\|\delta  Z'\|_{\beta';m,n}|\delta X|_\alpha+\|\E_\bigcdot R^Z\|_{\beta+\beta';n})|t-s|^{(\alpha\wedge \beta)+\beta'}\,.
		\]
		Next, writing $\E_s \delta Z_{s,t}=Z'_s \delta X_{s,t}+\E_sR^Z_{s,t}$ and applying the triangle inequality, we obtain
		\begin{align*}
			\|\|\delta Y_{s,t}-Z'_s \delta X_{s,t}|\cff_s\|_m\|_n\le\|\E_sR^Z_{s,t}\|_n+\|\| \delta Y_{s,t}-\E_s \delta Z_{s,t}|\cff_s\|_m\|_n\,.
		\end{align*}
		Combined with the former estimate, we obtain \eqref{est.jz}.
		Uniqueness of $(M,Y)$ given $(Z,Z')$ follows from \cite[Theorem 3.3(vi)]{LeSSL2}.
	\end{proof}

\subsection{Rough stochastic integrals}\label{sub.RSI}


We will now tackle the heart of the matter by defining the rough stochastic integral of
a progressively measurable process $Z\colon \Omega\times I\to W$ against a rough path \( \X=(X,\XX)\in \CC^\alpha(I;V) \), assuming the former is stochastically controlled with respect to \( X \). For that purpose though, extra regularity assumptions are required and we shall see in particular that these are fulfilled whenever the corresponding pair $\left( Z,Z^{\prime }\right)$ forms a stochastic controlled rough path in $\D_X^{\beta,\beta'}L_{m,n}( I,\W;\cll(V,W))$. Its rough stochastic integral $\int_o^t Zd\X$ is then well-defined as the limit in probability of the Riemann sums
\begin{equation}\label{riemann}
\sum_{[u,v]\in\cpp:o\le u < t}\left(Z_u \delta X_{u,v\wedge t}+Z'_u\XX_{u,v\wedge t}\right)
\end{equation}
as the mesh-size of $\cpp$ goes to $0$, for each $t\in I$
(here $\cpp$ is any partition of $I$).
Here, in writing $Z'_u\XX_{u,v}$, we have used the isomorphism $\cll(V,\cll(V,W))\simeq\cll(V\otimes V,W)$ (recall that \( V,W \) are finite-dimensional).
The resulting integration theory is self-consistent in the sense that $(\int Zd\X,Z)$ shares all of the properties of stochastic controlled rough paths, except for the fact that its second component, namely $Z$, is not necessarily bounded uniformly in $L_n$.

Although \eqref{riemann} has the same form as the defining Riemann sums for rough integrals (\cite{FH20}), the convergence of \eqref{riemann} only takes place in probability. This is due to the fact that the class of stochastic controlled rough paths contains not only controlled rough paths, but also nontrivial martingales (for which \eqref{riemann} fails to converge a.s.).
This alludes that the usual sewing lemma is not applicable. Instead, we rely on the stochastic sewing lemma, \cref{prop.SSL}, to obtain such convergence.

We now state our main result on rough stochastic integration in which the reader may assume $\beta=\beta'=\alpha$ at the first reading.
For any rough path $\X \in \mathscr{C}^\alpha$; $\beta,\beta'\in(0,1)$; $m,n\in[1,\infty]$;
and integrable stochastic processes $(S_{s,t}), (A_{s,t})$
we introduce the quantities%
\begin{align}
\label{nota.Gamma_1}
\Gamma_1^{\beta, \beta';m,n}(\X,S,A;I):=\rho_{\alpha,\alpha\wedge\beta;I}(\X) \big( \|\E_{\bigcdot} A\|_{\alpha\wedge\beta+\beta';n;I} + \|\E_{\bigcdot} S\|_{\beta';n;I}\big)\, ,
\\
\label{nota.Gamma_2}
\Gamma_2^{\beta,\beta';m,n}(\X,S,A;I):=\rho_{\alpha,\alpha\wedge\beta;I}(\X) \big(\|A\|_{\alpha\wedge\beta;m,n;I} + \| S\|_{\infty;m,n;I}\big)\,.
\end{align}
	\begin{theorem}[Rough stochastic integral]\label{thm.Hroughint}
		Let $\alpha\in(\frac14,\frac12]$, $\beta,\beta'\in(0,1]$, $\alpha+\beta>\frac12$, $\alpha+(\alpha\wedge \beta)+\beta'>1$, $m\in[2,\infty)$, $n\in[m,\infty]$ and $\X=(X,\XX) \in \mathscr{C}^\alpha([0,T];V)$.
Suppose that $Z,Z'$ are $\{\mathcal F_t\}$-progressively measurable processes  such that\footnote{Think of \eqref{con.roughintGamma} as the weakest possible condition to apply stochastic sewing, satisfied in particular by stochastic controlled rough paths, in sense of Definition \ref{def.SCRP}, introduced because of their good behaviour under composition, discussed in Section \ref{sub.stability}.}
\begin{align}\label{con.roughintGamma}
 \max_{i=1,2}\Gamma_i^{\beta,\beta';m,n}(\X,\delta Z',R^Z;I) <\infty\,.
\end{align}
	Then $A_{s,t} := Z_s \delta X_{s,t} + Z'_s\XX_{s,t}$
	defines a two-parameter stochastic process which satisfies the hypotheses of the stochastic sewing lemma. We define the rough stochastic integral $\int_o^\cdot Z d\X$ by the continuous process supplied by \cref{prop.SSL}(ii).
		In particular, $\int_o^\cdot Zd\X$ is the continuous process which corresponds to the limit in probability of \eqref{riemann} uniformly in time.
Moreover, the corresponding integral remainder
\( J_{s,t}=\int_s^tZd\X -Z_s\delta X_{s,t}-Z'_s\XX_{s,t}\) depends on \( (\X,Z,Z') \) in a Lipschitz fashion. More precisely, let
 \( (\bar \X ;\bar Z,\bar Z')\) denote another tuple subject to \( \max_{i=1,2}\Gamma_i^{\beta,\beta';m.n}(\bar\X,\delta\bar Z',\bar R^{\bar Z};I)<\infty \).
Then, defining \( \bar J\) accordingly, we have the estimates 
\begin{multline}
\label{est.Hcondz}
\|\EE
[J -\bar J ]\|_{\alpha+\alpha\wedge\beta+\beta';n;I}
\lesssim
\Gamma_1(\X-\bar\X,\delta\bar Z',\bar R^{\bar Z};I) + \Gamma_1(\X,\delta Z'-\delta \bar Z', R^Z-\bar R^{\bar Z};I)\,,
\end{multline}
\begin{equation}
\label{est.Hzm}
\begin{aligned}
\|J-\bar J\|_{\alpha+\alpha\wedge \beta;m,n;I}
&\lesssim
(|I|^{\beta'}\Gamma_1 + \Gamma_2)(\X-\bar\X,\delta\bar Z',\bar R^{\bar Z};I)
\\&\quad
+(|I|^{\beta'}\Gamma_1+\Gamma_2)(\X,\delta Z'-\delta \bar Z',R^Z-\bar R^{\bar Z};I)\,,
\end{aligned}
\end{equation}
and similarly 
\begin{multline}
\label{est.Hsup}
\Big\|\big\|\sup_{r\in I}|J_{o,r}-\bar J_{o,r}|\big|\cff_o\big\|_m\Big\|_n
\lesssim
\Big[(\Gamma_1+\Gamma_2)(\X-\bar\X,\delta\bar Z', \bar R^{\bar Z};I)
\\
+ (\Gamma_1+\Gamma_2)(\X,\delta Z'-\delta \bar Z',R^Z-\bar R^{\bar Z};I)\Big]
|I|^{\alpha+\alpha\wedge\beta-\frac12\vee(1-\beta')}.
\end{multline}
In the above inequalities, we have abreviated $\Gamma_i=\Gamma_i^{\beta,\beta';m,n}$ for $i=1,2$ and all hidden constants  depend on $\alpha,\beta,\beta',m$ and \( T \), but are independent of \( \X,\bar\X,Z,Z',\bar Z,\bar Z' \).
\end{theorem}

\begin{proof}
		Using the Chen's relation \eqref{chen}, we easily arrive at the identity
		\begin{equation}
		\label{delta_AA}
			-\delta A_{s,u,t}
			= R^Z_{s,u}\delta X_{u,t}
			+ \delta  Z'_{s,u}\XX_{u,t}
		\end{equation}
		for every \( (s,u,t)\in \Deltatwo(I) \). This implies that
		\begin{align*}
			\|\E_s\delta A_{s,u,t}\|_n
			&\le (t-s)^{\alpha+\alpha\wedge \beta+\beta'}\Big(|\delta X|_{\alpha;I}\|\E_{\bigcdot} R^Z\|_{\alpha\wedge\beta+\beta';n;I}
			+|\XX|_{\alpha+\alpha\wedge\beta;I}\|\E_{\bigcdot} \delta Z'\|_{\beta';n;I}\Big),
		\\&\lesssim (t-s)^{\alpha+\alpha\wedge\beta+\beta'} \Gamma_1^{\beta,\beta';m,n}(\X,\delta Z',R^Z;I) 
		\end{align*}
		and similarly
		\begin{align*}
			\bigg\| \bigg\|\sup_{\tau\in[u,t]}|\delta A_{s,u,\tau}|\bigg|\cff_s\bigg\|_m\bigg\|_n
			\le (t-s)^{\alpha+\alpha\wedge\beta} \Gamma_2^{\beta,\beta';m,n}(\X,\delta Z',R^Z;I),
		\end{align*}
		showing the conditions \eqref{con.dA1}, \eqref{con.dA2} and \eqref{con.dA3a} of \cref{prop.SSL} with \( \varepsilon_1=\alpha+\alpha\wedge\beta+\beta'-1>0 \), \( \varepsilon_2=\alpha+\alpha\wedge\beta-1/2>0 \) and \( \varepsilon_3=\alpha+\alpha\wedge\beta -\frac1m>0\).
That $t\mapsto A_{s,t}$ is a.s.\ continuous for each $s$ is evident. 
It follows that 
the process $\int_o^{\cdot} Zd\X:=\mathcal A$ is well-defined.
The convergence of \eqref{riemann} also follows from \cref{prop.SSL}. 

Next, we show \eqref{est.Hcondz}--\eqref{est.Hsup}.
When \( \bar Z=0 \) and \( \bar\X=\X \), the claimed estimates follow directly by \eqref{est.A1}, \eqref{est.A2}, \eqref{est.unifAPA} (with the trivial partition) and the previous bounds.
The general case follows similar arguments based this time on the identity
\begin{multline*}
		-\delta(A-\bar A )_{s,u,t}
			= \bar R^{\bar Z}_{s,u}(\delta X-\delta \bar X)_{u,t}
			+ \delta \bar Z'_{s,u}(\XX-\bar \XX)_{u,t}
			\\
			+(R^{Z}-\bar R^{\bar Z})_{s,u}\delta X_{u,t}
			+ (\delta Z' - \delta \bar Z')_{s,u}\XX_{u,t} \,,
		\end{multline*}
		where \( \bar A_{s,t}=\bar Z_s\delta \bar X_{s,t} + \bar Z_s'\mathbb{\bar X}_{s,t} \).
We leave the details of these bounds to the reader.
	\end{proof}

We now state an important corollary concerning integrability of rough stochastic controlled paths as per \cref{def.SCRP}, as well as the continuity of the integration map in that context. 
\begin{corollary}[Continuity of integration map]
\label[corollary]{cor.integral}
Fix \(m,n, \alpha,\beta,\beta' \) as in \cref{thm.Hroughint}.\footnote{Both cases, $\alpha \le \beta$ and $ \beta \le \alpha$ are interesting. For instance, rough Brownian sample paths is $(1/2-\varepsilon)$-H\"older, whereas Brownian motion as moment space valued path has H\"older exponent $1/2$. On the other hand, in our later Picard argument for RSDEs it will be important to take $\beta < \alpha$.} 
Let \( (Z,Z'), (\bar Z,\bar Z') \) be  stochastic controlled rough paths respectively in \( \D_{X}^{\beta,\beta'}L_{m,n} \) and \( \D_{\bar X}^{\beta,\beta'}L_{m,n} \).

\noindent  (i)   For $J,\bar J$ as in \cref{thm.Hroughint}, we have
\begin{multline}\label{est.RSIE}
	\|\EE(J-\bar J)\|_{\alpha+\alpha\wedge\beta+\beta';n;I}
	+\|J-\bar J\|_{\alpha+\alpha\wedge\beta;m,n;I}
	+\Big\|\big\|\sup_{r\in I}|J_{o,r}-\bar J_{o,r}|
	\big|\cff_o\big\|_m\Big\|_n
	\\
	\lesssim
	\reply{\rho_\alpha(\X)(1+\rho_\alpha(\X))}
	\|Z,Z';\bar Z,\bar Z'\|_
	{X,\bar X;\alpha\wedge\beta,\beta';m,n;I}
	\\
	+
	\reply{
	\|(\bar Z,\bar Z')\|_
	{\bar X;\alpha\wedge\beta,\beta';m,n;I}
	}
	\rho_{\alpha,\alpha\wedge\beta}(\X,\bar\X).
\end{multline}
(ii) Assuming that \( \|Z\|_{\infty;n}<\infty \) and \( \|\bar Z\|_{\infty;n}<\infty \), we have that
\begin{multline}
\label{est.integral}
\normlr{\int_o^{\cdot} Zd\X,Z;\int_o^{\cdot} \bar Z d\bar\X,\bar Z}_{X,\bar X;\alpha,\alpha\wedge\beta;m,n}
+\left\|\bigg\|\sup_{t\in I}\bigg|\int_o^t Zd\X-\int_o^t \bar Z d\bar\X\bigg|\ \bigg|\cff_o\bigg\|_{m}\right\|_n
\\\lesssim
C'\rho_{\alpha,\alpha\wedge\beta}(\X,\bar\X)
+(1+C)(\|Z-\bar Z\|_{\infty;n}+
 \|Z,Z';\bar Z,\bar Z'\|_{X,\bar X;\alpha\wedge\beta,\beta';m,n})
\end{multline}
where \( C = \rho_\alpha(\X) (1+\rho_\alpha(\X))\) and 
\( C'= \sup_{t}\|\bar Z_t\|_n +\|(\bar Z,\bar Z')\|_{\bar X;\beta,\beta';m,n} \).

\noindent (iii) For fixed  \( \X\in \mathscr C^{\alpha}(V) \), the integration map
\[
\begin{aligned}
\D_X^{\beta,\beta'}L_{m,n}\cap \{(Z,Z')\colon \|Z\|_{\infty;n}<\infty 
\}
&\longrightarrow \D_{X}^{\alpha,\alpha\wedge\beta}L_{m,n}
\\
(Z,Z')&\longmapsto \left (\int_o^{\cdot} Zd\X,Z\right )
\end{aligned}
\]
is well-defined, linear and bounded in the sense that\footnote{Actually,
$\bk{
  (\int_o^{\cdot} Zd\X,Z)
  }_{X;\alpha,\alpha\wedge\beta;m,n}
\lesssim \| \delta Z \|_{\alpha\wedge\beta;m,n} + 
 C \{ \| Z \|_{\infty;n}
 + \|(Z,Z')\|_{X;\alpha\wedge\beta,\beta';m,n}\}$, which upon adding $\| Z \|_{\infty;n}$ leads to \eqref{est.integral_lin}. This form of the estimate
 is more aligned with estimates for deterministic rough integrals, where one often uses $C \lesssim T^\delta$, assuming $\X$ to be $\alpha+\delta$-H\"older.}
\begin{equation}
\label{est.integral_lin}
\left \|(
\int_o^{\cdot} Zd\X,Z)\right \|_{X;\alpha,\alpha\wedge\beta;m,n}
\lesssim (1+C)(\|Z\|_{\infty;n} + \|(Z,Z')\|_{X;\alpha\wedge\beta,\beta';m,n})\,.
\end{equation}
\end{corollary}
\begin{proof}
	We  observe that \( \|\EE R^Z\|_{\alpha\wedge\beta+\beta';n}\lesssim \|\EE R^Z\|_{\beta+\beta';n} \) while 
\( \|\EE \delta Z'\|_{\beta';m,n} \le \|\delta Z'\|_{\beta';m,n}\). 
This yields that 
\begin{equation}\label{tmp.g1}
	\Gamma_1^{\beta,\beta';m,n}(\X,\delta Z',R^Z;I)\lesssim \rho_{\alpha,\alpha\wedge\beta}(\X) \bk{(Z,Z')}_{X;\alpha\wedge\beta,\beta';m,n}.
\end{equation}
Furthermore, using triangle inequality $\|R^Z\|_{\alpha\wedge\beta;m,n}\le \|\delta Z\|_{\alpha\wedge\beta;m,n} + |\delta X|_{\alpha\wedge\beta}\|Z'\|_{\infty;n}$ and the trivial bound $\|\delta Z'\|_{\infty;m,n}\le 2\|Z'\|_{\infty;n}$, we also have that 
\begin{equation}\label{tmp.g2}
	\Gamma_2^{\beta,\beta';m,n}(\X,\delta Z',R^Z;I)\lesssim (1+|\delta X|_{\alpha\wedge\beta})\rho_{\alpha,\alpha\wedge\beta}(\X)(\bk{(Z,Z')}_{X;\alpha\wedge\beta,\beta';m,n}+\|Z'\|_{\infty;n}).
\end{equation}
In particular, \cref{thm.Hroughint} asserts that \( (Z,Z') \) has a well-defined rough stochastic integral.
\reply{
Since
\(Z'_s\delta X_{s,t}\) is \(\cff_s\)-measurable, we have
}
\[
\reply{
R^Z_{s,t}-\E_sR^Z_{s,t}
=
\delta Z_{s,t}-\E_s\delta Z_{s,t}.
}
\]
\reply{Hence, by conditional Jensen's inequality,}
\[
\reply{
\|R^Z\|_{\alpha\wedge\beta;m,n;I}
\le
2\|\delta Z\|_{\alpha\wedge\beta;m,n;I}
+
|I|^{\beta'}
\|\EE R^Z\|_{\alpha\wedge\beta+\beta';n;I}.
}
\]
\reply{The same argument gives}
\[
\reply{
\|R^Z-\bar R^{\bar Z}\|_{\alpha\wedge\beta;m,n;I}
\le
2\|\delta Z-\delta\bar Z\|_{\alpha\wedge\beta;m,n;I}
+
|I|^{\beta'}
\|\EE(R^Z-\bar R^{\bar Z})\|_{\alpha\wedge\beta+\beta';n;I}.
}
\]
\reply{
Moreover,
\[
\|\delta Z'\|_{\infty;m,n;I}
\le
|I|^{\beta'}
\|\delta Z'\|_{\beta';m,n;I},
\]
and likewise for \(\delta\bar Z'\) and
\(\delta Z'-\delta\bar Z'\).
}
Similarly, we have
\begin{align}
	\max_{i=1,2}
	\Gamma_i^{\beta,\beta';m,n}
	(\X-\bar\X,\delta\bar Z',\bar R^{\bar Z};I)
	&\lesssim
	\norm{(\bar Z,\bar Z')}_{
	\reply{\bar X};\alpha\wedge\beta,\beta';m,n}
	\rho_{\alpha,\alpha\wedge\beta}(\X,\bar\X),
	\label{equ.GG1}
\\	\max_{i=1,2}\Gamma_i^{\beta,\beta';m,n}(\X,\delta Z'-\delta \bar Z',R^{Z}-\bar R^{\bar Z};I)
	&\lesssim C
	\norm{Z,Z';\bar Z,\bar Z'}_{X,\bar X;\alpha\wedge\beta,\beta';m,n}, \label{equ.GG2}
\end{align}
where $C$ is as in the statement.

Part (i) is a  direct consequence of \eqref{est.Hcondz}, \eqref{est.Hzm} and \eqref{est.Hsup} with $\Gamma_1, \Gamma_2$ bounded as in  \eqref{equ.GG1}, \eqref{equ.GG2}.

(ii) Define \( (Y,Y')=(\int Zd\X,Z) \) and similarly for \( (\bar Y,\bar Y') \).
In the notations of \cref{thm.Hroughint}, we have
\( \|\E_s(R^{Y}-\bar R^{\bar Y})_{s,t}\|_n
		\le \|\bar Z'_s(\XX_{s,t}-\bar\XX_{s,t})\|_n+\|(Z'_s-\bar Z'_s)\XX_{s,t}\|_n +\|\E_s (J-\bar J)_{s,t}\|_{n} \).
Applying \eqref{est.RSIE} and the preious estimate, we obtain that
		\[\begin{aligned}
			\|\EE(R^{Y}-\bar R^{\bar Y})\|_{\alpha+\alpha\wedge \beta ;n}
		&\lesssim
C'\rho_{\alpha,\alpha\wedge\beta}(\X, \bar\X)
+ C\norm{Z,Z';\bar Z,\bar Z'}_{X,\bar X;\alpha\wedge\beta,\beta';m,n}.
		\end{aligned}\]
We proceed similarly for the difference between increments, starting this time from
\[
\begin{aligned}
\|(\delta Y-\delta \bar Y)_{s,t}\|_{m,n}&
\le \|\bar Z_s(\delta X_{s,t} - \delta\bar X_{s,t})+\bar Z'_s(\XX_{s,t}-\bar\XX_{s,t}) \|_{m,n}
\\&	\quad +\|(Z_s-\bar Z_s) \delta  X_{s,t}+(Z'_s-\bar Z'_s)\XX_{s,t}\|_{m,n}
+\|J_{s,t}-\bar J_{s,t}\|_{m,n}.
\end{aligned}
\]
Using \eqref{est.RSIE} to treat the last term above,  it follows that
\begin{multline*}
\|\delta Y-\delta \bar Y\|_{\alpha;m,n}
\lesssim 
C'\rho_{\alpha,\alpha\wedge\beta}(\X, \bar\X)
\\
+C\bigg(\|Z-\bar Z\|_{\infty;n}
+ \norm{Z,Z';\bar Z,\bar Z'}_{X,\bar X;\alpha\wedge\beta,\beta';m,n}
\bigg).
\end{multline*}
After summing up these contributions, we obtain the desired bound for 
\[\normlr{\int_o^{\cdot} Zd\X,Z;\int_o^{\cdot} \bar Z d\bar\X,\bar Z}_{X,\bar X;\alpha,\alpha\wedge\beta;m,n}\] as in \eqref{est.integral}. The estimate for the second term with $\sup_{t\in I}$ in \eqref{est.integral} follows from \eqref{est.RSIE} and similar arguments. 

(iii) 
Taking \( \X=\bar\X \) and \( (\bar Z,\bar Z')=(0,0) \) in \eqref{est.integral} entails \eqref{est.integral_lin}. It is straightforward to check that the pair \( (\int _o^{\cdot}Z d\X,Z) \) satisfies all properties of stochastic controlled rough paths in \cref{def.SCRP}.
\end{proof}
\begin{remark}
\label{rem.integral_gene}
The averaged form in which the Gubinelli derivative appears in \eqref{nota.Gamma_1} shows that integration makes sense for a larger class of ``extended'' stochastic controlled rough paths, as introduced later in \cref{sec:ito}. Since that class is not stable by composition (contrary to stochastic controlled rough paths, at least when \( n=\infty \), see \cref{prop.compose_stab} below), it does not play a prominent role as far as we are concerned with RSDEs of the form \eqref{eqn.srde}. That is why we postpone its introduction until later sections.
\end{remark}

\subsection{Stochastic controlled vector fields} 
\label{sub.stability}
	Unless stated otherwise, in the sequel we work with a H\"older path $X$ in $C^\alpha(I;V)$ for some $\alpha\in(0,1]$ and a compact interval \( I\subset [0,T] \).
	Recall that $V,W$ are finite-dimensional Banach spaces, we also set here $\bar W := \cll(V,W)$, another (finite-dimensional) Banach space.

	Herein we introduce a class of random, time-dependent and progressively measurable vector fields\footnote{Our terminology comes from viewing $f$ as collection of $d$ (time-dependent, random) vector fields when \(\mathrm{dim}V=d \). }
	\begin{equation}
	\label{fa_view}
	\begin{aligned}
		f&\colon\Omega\times I\longrightarrow \mathcal X\hookrightarrow\mathcal C_b(W,\bar W)
		\\
		f'&\colon \Omega\times I\longrightarrow \mathcal Y\hookrightarrow\mathcal C_b\big(W,\mathcal L(V,\bar W)\big)
	\end{aligned}
\end{equation}
 	for well-chosen Banach spaces of functions \( \mathcal X,\mathcal Y \), and where strong%
	\footnote{This particular detail matters in the discussion since none of the natural target spaces \( \Cb\) or \(\mathcal C_b^\gamma \) for \(\gamma>0\) is separable (see however \cref{rem:separability}).}
	\( \mathcal G /\Bor(\abx)\)-measurability (resp.\ strong \( \mathcal G /\Bor(\mathcal Y)\)-measurability) is assumed, see \cref{sec:framework}.
	Our main purpose here is to investigate a natural composition operation of that class with stochastic controlled rough paths and show that it yields a similar object, subject to explicit local-Lipschitz estimates.

	As is well-known in the absence of time and sample parameters, the pair \( (f^{\circ}(Y),Df^\circ(Y)Y') \) forms an \( X \)-controlled rough path if \( (Y,Y') \) shares that property (with common H\"older regularity exponent \( \alpha \), say) provided that
\[
 f^{\circ}\in \C^\gamma_b(W;\bar W),\quad \quad \text{for some }\gamma>\frac1\alpha\,.
\]
Given the functional analytic viewpoint laid out in \eqref{fa_view},
it is then tempting to let \( \abx=\mathcal C_b^{\gamma}(W,\bar W) \), \( \mathcal Y=\mathcal C_b^{\gamma}(W,\mathcal L(V,\bar W))\simeq \mathcal L\big(V,\abx\big) \) and simply define stochastic controlled vector fields as stochastic controlled rough paths with values in \( \mathcal X \) (note that \cref{def.SCRP} extends trivially to infinite-dimensional state spaces).
Although doing the job, this description would be too demanding regularity-wise as it fails to capture possible tradeoffs between space and time regularities at the level of the vector fields. A much better definition, which we employ in the rest of the paper, is the following.
	\begin{definition}[Stochastic controlled vector fields]\label[definition]{def.scvec}
		Let $\beta, \beta' \in (0,1]$ and $\gamma\in(1,\infty)$, $m\in[2,\infty)$ and $n\in[m,\infty]$ be some fixed parameters.
		We call $(f,f')$ {\em stochastic controlled vector field on \( W \)}
		of $(m,n)$-integrability and $(\gamma,\beta,\beta')$ regularity with respect to $\{\cff_t\}$
		if the following conditions are satisfied.
		\begin{enumerate}[label=(\alph*)]
			\item\label{scvec.f}
			The pair
			\[
			(f,f')\colon\Omega \times  I  \to \C^\gamma_{b}(W, \bar W) \times \C^{\gamma-1}_{b} (W,\mathcal L(V,\bar W))
			\]
			is  progressively measurable in the strong sense and uniformly $n$-integrable i.e.
			\[
			\sup_{s\in I}\||f_s|_{\gamma}\|_n +  \sup_{s\in I}\||f'_s|_{\gamma-1}\|_n < \infty\,.
			\]

			\item\label{scvec.brakets}
			Letting
			\begin{align}\label{def.bk}
				\bk{Z}_{\kappa;m,n}:=\sup_{(s,t)\in \Delta(I):s\neq t}\frac{\Big\|\big\|\sup_{x\in W}|Z_{s,t}(x)|\,\big|\,\cff_s \big\|_m\Big\|_n}{(t-s)^\kappa}\,,
			\end{align}
			the quantities
			$\bk{\delta f}_{\beta;m,n}$, $\bk{\delta f'}_{\beta';m,n}$, \( \bk{\delta Df}_{\beta';m,n} \) are finite.%

			\item\label{scvec.remainder} The map $(s,t)\mapsto \E_s R^f_{s,t}=\E_sf_t-f_s-f'_s \delta X_{s,t}$ belongs to \( C_2^{\beta+\beta'}L_{n}(\Cb) \)%
			\footnote{Similar to \cref{def.SCRP}, it is equivalent to say that \( \EE R^f \) belongs to \( C_2^{\beta+\beta'}L_{m,n}(\Cb) \)
			},
			namely
			\[
			\bk{\E_\bigcdot R^f}_{\beta+\beta';n}
			=\sup_{(s,t)\in \Delta(I):s\neq t}\frac{\|\sup_{y\in W}|\E_sR^{f}_{s,t}(y)|\|_n}{(t-s)^{\beta+\beta'}}
			<\infty\,.
			\]
		\end{enumerate}
		The class of such vector fields is denoted by $\D_X^{\beta,\beta'}L_{m,n}\C^\gamma_b( I,\W;W)$, or simply
	 $\D_X^{\beta,\beta'}L_{m,n}\C^\gamma_b$ whenever the tuple $( I,\W;W)$ is clear from the context. 
	 Additionally, we write $\D_X^{2 \beta}L_{m,n}\C^\gamma_b=\D_X^{\beta,\beta}L_{m,n}\C^\gamma_b$.

		We call $(f,f')$ $L_{m,\infty}$-integrable$, (\gamma,\alpha,\alpha')$-space-time-regular stochastic controlled vector fields if
		$(f,f')\in \D^{\alpha,\alpha'}_XL_{m,\infty}\C^\gamma_b$, and $(Df,Df') \in \D^{\alpha,\alpha'}_XL_{m,\infty}\C^{\gamma-1}_b$. (Write $2\alpha$ instead of $\alpha,\alpha'$ in case they are equal.)
	\end{definition}
	For stochastic controlled vector fields as above, we introduce moreover the quantities
	\begin{equation}
		\label{def.norms_scvf}
		\left \{\begin{aligned}
			[(f,f')]_{\gamma;n}&:=\sup_{s\in I}\left(\|[f_s]_\gamma\|_n+\||f'_s|_{\gamma-1}\|_n\right) ,
			\\
			\|(f,f')\|_{\gamma;n}&:=[(f,f')]_{\gamma;n}    +  \sup_{s\in I} \|  |f_s|_\infty  \|_n ,
			\\
			\bk{(f,f')}_{X;\beta,\beta';m,n}
			&:=\bk{\delta f}_{\beta;m,n}+\bk{\delta Df}_{\beta';m,n}+\bk{\delta f'}_{\beta';m,n}+\bk{\E_\bigcdot R^f}_{\beta+\beta';n},
		\end{aligned}
	\right .
	\end{equation}
	which is abbreviated as $\bk{(f,f')}_{X;\beta,\beta';m}$ if $m=n$, as $\bk{(f,f')}_{X;2\beta;m,n}$ if \( \beta=\beta' \) and as $\bk{(f,f')}_{X;2\beta;m}$ when both conditions are met.

Similarly, if $(\bar f,\bar f')\in \D^{\beta,\beta'}_{\bar X}L_{m,\infty}\C^\gamma_b$ for another such \(\bar X\in C^\alpha(V) \), we define\footnote{One immediately defines a distance
$\bk{ -  ; - }_{X,\bar X;\beta, \beta';m,n}$ relative to mixed moments, however only $m=n$ will turn out to be relevant.}
 	\begin{equation}
	\label{def.bk_metric}
	\begin{aligned}
		\bk{f,f';\bar f,\bar f'}_{X,\bar X;\beta, \beta';m}&=\bk{\delta f-\delta \bar f}_{\beta;m}+\bk{\delta f'-\delta \bar f'}_{\beta';m}+\bk{\delta Df-\delta D\bar f}_{\beta';m}
		\\&\quad+\bk{\E_\bigcdot R^f-\E_\bigcdot \bar R^{\bar f}}_{\beta+\beta';m}.
	\end{aligned}
	\end{equation}

	\begin{remark}
		\label[remark]{rem:separability}
		By nature, the Lipschitz spaces $\C^{\gamma}_b (W)$ are non-separable (similar to the space of continuous bounded function on the real line), so that ``strong'' measurability assumptions are in order (cf. Section \ref{sec:ss}). Non-separability of such H\"older spaces is also a well-know feature of (H\"older) rough path - and model spaces (in regularity structures), typically with $W$ replaced by some interval of finite-dimensional torus, respectively. Following \cite{FrizVictoirNoteOnGRP,hairer2015large} one can usually work with separable subspace obtained by the closure of smooth (rough) paths and models, respectively. This simplification is not available to us, since a compact state-space for the solution process of RSDEs would entail a significant loss of generality. (The situation is even worse in \cite{friz2025McKean} where we encounter s.c.v.f.s on infinite-dimensional spaces, auxiliary moments spaces from a Lions lifting construction.)
   \end{remark}

The concept introduced in \cref{def.scvec} seems new even when the underlying pair \( (f,f') \) is deterministic (to our best knowledge) for an application. 
It is also worth noticing that the condition \( (f,f')\in \D_X^{\beta,\beta'}L_{m,n}\Cb^{\gamma} \) reduces to ``\( f\in \Cb^{\gamma} \)'' when there is no time nor sample parameter.
In the next example, we give a natural recipe to build genuinely random elements.

	\begin{example}
		\label{ex.v0}
Let \( (f,f')\in \D_X^{\beta,\beta'}L_{m,n}\mathcal C_b^{\gamma}([0,T],\W;W)\) for $\beta,\beta',\gamma,m,n$ as in \cref{def.scvec}.
Suppose that \( W=W_1\times W_2 \) and take $(Y,Y')$ a stochastic controlled rough path in \( \D^{\beta,\beta'}_XL_{m,\infty}([0,T],\W;W_2) \).
We can construct another stochastic controlled vector field on \( W_1 \) through the formula
		\begin{equation}
		\label{formula_compose}
		\left \{\begin{aligned}
		&g_t(\cdot)=f_t(\cdot ,Y_t)
		\\
		&g'_t(\cdot)=D_2 f_t(\cdot,Y_t)Y'_t + f_t'(\cdot, Y_t)\,,
		\end{aligned}\right .
		\end{equation}
where \( D_2 \) is the derivative with respect to the argument in \( W_2 \).
It can be observed that \( (g,g') \) belongs to \( \D_X^{\beta,\min\{(\gamma-2)\beta,\beta'\}}L_{m,n}\C^{\gamma}_b([0,T],\W;W_1) \)
(see \cite[Section 5]{friz2025McKean} for details).
\end{example}

The following lemma is useful to obtain estimates for the composition of a stochastic controlled vector field with a stochastic controlled rough path.
\begin{lemma}\label[lemma]{lem:substitute}
Let $\mathcal{F}\subset \mathcal G$ be a sub-$\sigma$-field. Let $E$ be a Banach space and  
$f:E\times\Omega\to\R$ be a bounded function that is jointly measurable and continuous in its first argument.
Then,  for every strongly measurable random variable $X:(\Omega,\cff)\to E$,  $f(X(\cdot),\cdot):(\Omega,\cgg)\to \R$ is measurable and we have 
\begin{align}\label{est.conditionalfx}
	|\E[f(X(\cdot),\cdot)|\cff]|(\omega)\le \sup_{x\in E}|\E[f(x,\cdot)|\cff]|(\omega) \quad \text{for a.s. } \omega\in \Omega.
\end{align}
\end{lemma}
\begin{proof}
	Because $X$ is strongly measurable, $X=\lim_nX_n$ a.s. for some sequence of random variables $X_n$ such that each $X_n$ has finitely many values. 
	We can write $X_n=\sum_{i=1}^{k_n} x_i \1_{A_i}$ for some finite integer $k_n$, $x_i\in E$ and disjoint partition $\{A_i\}\subset\cff$ of $\Omega$. 
	Then, we can write
	\begin{align*}
		f(X_n(\omega),\omega)=\sum_{i=1}^{k_n} f(x_i,\omega)\1_{A_i}(\omega),
	\end{align*}
	which is measurable. By continuity of $f$, we have $f(X_n(\cdot),\cdot)\to f(X(\cdot),\cdot)$ a.s. as $n\to\infty$. This shows that $f(X(\cdot),\cdot)$ is  measurable.
	Furthermore, taking conditional expectation with respect to $\cff$, we have
	\begin{align*}
		|\E[f(X_n(\cdot),\cdot)|\cff](\omega)|\le \sum_{i=1}^{k_n} \1_{A_i}(\omega)\sup_{x\in E} |\E[f(x,\cdot)|\cff](\omega)|=\sup_{x\in E} |\E[f(x,\cdot)|\cff](\omega)|.
	\end{align*}
	As \( n\to\infty \), applying the dominated convergence theorem, we obtain \eqref{est.conditionalfx}.
\end{proof}

		\begin{lemma}\label[lemma]{lem.composecvec}
			Let $\beta,\beta'\in(0,1]$, $\gamma\in(2,3]$, and $m\in[2,\infty)$ and $n\in[m,\infty]$.
			Let $(f,f')$ be a stochastic controlled vector field in $ \D^{\beta,\beta'}_XL_{m,\infty} \C^{\gamma-1}_b$ and $(Y,Y')$ be a stochastic controlled rough path in $ \D^{\beta,\beta'}_XL_{m,n}$. Define $\beta''=\min\{(\gamma-2)\beta,\beta'\}$ and  $(Z,Z')=(f(Y),Df(Y)Y'+f'(Y))$. Then $(Z,Z')$ is a stochastic controlled rough path in $ \D^{\beta,\beta''}_XL_{m,\frac n{\gamma-1}}$ with
			\begin{align}\label{est.ZinD}
				\|(Z,Z')\|_{X;\beta,\beta'';m,\frac n{\gamma-1}}\lesssim ([(f,f')]_{\gamma-1;\infty}+\bk{(f,f')}_{X;\beta,\beta';m,\infty})( 1+ \|(Y,Y')\|_{X;\beta,\beta';m,n}^{\gamma-1}),
			\end{align}
			for an implicit constant depending only on $|I|$.
		\end{lemma}
		\begin{proof}
			That $Z_t,Z'_t$ are well-defined random variables for each $t$ follows from \cref{lem:substitute}. 
To show the result, setting $\|Df_s\|_\infty=\|\sup_x |Df_s(x)|\|_\infty$, we are going to establish the following estimates for each \( (s,t)\in \Delta (I) \):
			\begin{equation}\label{est.cfy}
				\|\|\delta Z_{s,t}|\cff_s\|_m\|_n\le(\bk{\delta f}_{\beta;m,n}+ \|Df_s\|_\infty\|\delta Y\|_{\beta;m,n})|t-s|^\beta\,,
			\end{equation}
			\begin{multline}\label{est.cERfY}
				\|\E_s R^{Z}_{s,t}\|_{\frac n{\gamma-1}}\le \|[Df_s]_{\gamma-2}\|_\infty \|\delta Y\|_{\beta;m,n} ^{\gamma-1}|t-s|^{(\gamma-1)\beta}
				\\+(\bk{\delta Df}_{\beta';m,\infty}\|\delta Y\|_{\beta;m,n}+ \|Df_s\|_\infty \|\E_{\bigcdot} R^Y\|_{\beta+\beta';n}+\llbracket \E_\bigcdot R^f\rrbracket_{\beta+\beta';\frac n{\gamma-1}})|t-s|^{\beta+ \beta'}\,,
			\end{multline}
			and
			\begin{equation}\label{est.cddf}
			\begin{aligned}
				&\|\|\delta Z'_{s,t}|\cff_s\|_m\|_{\frac n{\gamma-1}}
				\\&\le
				 (\|[Df_s]_{\gamma-2}\|_\infty\|\delta Y\|^{\gamma-2}_{\beta;m,n}\sup_{r}\|Y'_r\|_{n}+\||f'_s|_{\gamma-2} \|_\infty\|\delta Y\|_{\beta;m,n}^{\gamma-2} ) |t-s|^{(\gamma-2)\beta}
				\\&\quad+(\|Df_t\|_\infty \|\delta Y'\|_{\beta';m,n}+\bk{\delta Df}_{\beta';m,\frac n{\gamma-2}}\sup_r\|Y'_r\|_n+\bk{\delta f'}_{\beta';m,\frac n{\gamma-1}})|t-s|^{\beta'}\,.
			\end{aligned}
			\end{equation}
Since we have $\|Df_t(Y_t)Y'_t\|_n\le\|Df_t\|_\infty\|Y'_t\|_n$ and $\|f'_t(Y_t)\|_n\le\||f'_t|_\infty\|_\infty$, it is obvious on the other hand that $Z'$ is uniformly \( L_n \)-integrable 
and $\norm{Z'}_{\infty;n}\le [(f,f')]_{\gamma-1;\infty}(1+\norm{(Y,Y')}_{X;\beta,\beta';m,n})$. Thus, estimates \eqref{est.cfy}-\eqref{est.cddf} will be sufficient to show that \( (Z,Z')\in \D_X^{\beta,\beta''}L_{m,\frac{n}{\gamma-1}} \), as claimed.\smallskip

			Now, the first of these inequalities is trivial, since by triangle inequality
			\begin{align*}
				|f_t(Y_t)-f_s(Y_s)|&\le|f_t(Y_t)-f_s(Y_t)|+|f_s(Y_t)-f_s(Y_s)|
				\\&\le |\delta f_{s,t}|_\infty+|Df_s|_\infty|\delta Y_{s,t}|\,,
			\end{align*}
			leading to \eqref{est.cfy}. To treat $R^Z$, we write
			\begin{align}
				R^Z_{s,t}&=f_s(Y_t)-f_s(Y_s)-Df_s(Y_s)Y'_s \delta X_{s,t}
				\nonumber\\&\quad+f_t(Y_s)-f_s(Y_s)-f'_s(Y_s)\delta X_{s,t}
				\nonumber\\&\quad+ f_{t}(Y_t)- f_{t}(Y_s)-f_s(Y_t)+f_s(Y_s)
				\nonumber\\&= R^{f_s(Y)}_{s,t}+R^f_{s,t}(Y_s)+\left(\delta f_{s,t}(Y_t)-\delta f_{s,t}(Y_s)\right) .
				\label{tmp.357}
			\end{align}
			By the fundamental theorem of calculus,
			\begin{equation*}
			 	R^{f_s(Y)}_{s,t}=\left(\int_0^1\left[Df_s(Y_s+\theta \delta Y_{s,t})-Df_s(Y_s)\right]d \theta\right) \delta Y_{s,t}+Df_s(Y_s)R^Y_{s,t}\,,
			\end{equation*}
			which yields 
			\begin{align}\label{tmp.Efs}
				|\E_sR^{f_s(Y)}_{s,t}|\le [Df_s]_{\gamma-2} |\E_s|\delta Y_{s,t}|^{\gamma-1}|+|Df_s|_\infty|\E_sR^Y_{s,t}|\,.
			\end{align}
			Applying the $L_{n/(\gamma-1)}$-norm and triangle inequality gives
			\begin{align*}
				\|\E_sR^{f_s(Y)}_{s,t}\|_{\frac n{\gamma-1}}
				&\le \|[Df_s]_{\gamma-2}\|_\infty \|\|\delta Y_{s,t}|\cff_s\|_{\gamma-1}\|_{n}^{\gamma-1}+\|Df_s\|_\infty\|\E_sR^Y_{s,t}\|_{\frac n{\gamma-1}}
				\\
				&\le \|[Df_s]_{\gamma-2}\|_\infty \|\|\delta Y_{s,t}|\cff_s\|_{m}\|_{n}^{\gamma-1}+\|Df_s\|_\infty\|\E_sR^Y_{s,t}\|_{n}.
			\end{align*}
			From here, we obtain
			\begin{align}
				\|\E_s R^{f_s(Y)}_{s,t}\|_{\frac n{\gamma-1}}\le \|[Df_s]_{\gamma-2}\|_\infty \|\delta Y\|_{\beta;m,n} ^{\gamma-1}|t-s|^{(\gamma-1)\beta}
				+\|Df_s\|_\infty \|\E_{\bigcdot} R^Y\|_{\beta+\beta';n}|t-s|^{\beta+ \beta'}\,.
			\end{align}
			Using \eqref{est.conditionalfx},
			the second term in \eqref{tmp.357} is easily estimated by 
			\[
				\|\E_sR^f_{s,t}(Y_s)\|_{\frac n{\gamma-1}}\le\||\E_sR^f_{s,t}|_{\infty}\|_{\frac n{\gamma-1}}\le \bk{\E_\bigcdot R^f}_{\beta+\beta';\frac n{\gamma-1}}|t-s|^{\beta+\beta'}.
			\]

	For the last term in \eqref{tmp.357}, we use the Lipschitz estimate $|\delta f_{s,t}(Y_t)- \delta f_{s,t}(Y_s)|\le |\delta Df_{s,t}|_\infty|\delta Y_{s,t}|$
			and H\"older inequality to obtain that
			\begin{align*}
				\|\E_s(\delta f_{s,t}(Y_t)-\delta f_{s,t}(Y_s))\|_{\frac n{\gamma-1}}
				&\le \|\E_s(|\delta Df_{s,t}|_\infty|\delta Y_{s,t}|)\|_{\frac n{\gamma-1}}
				\\&\le \|\||\delta Df_{s,t}|_\infty|\cff_s\|_{m} \|_\infty \|\|\delta Y_{s,t}|\cff_s\|_m\|_{\frac n{\gamma-1}}
				\\&\le \bk{\delta Df}_{\beta';m,\infty}\|\delta Y\|_{\beta;m,\frac n{\gamma-1}}(t-s)^{\beta+\beta'}.
			\end{align*}
			Putting these estimates in \eqref{tmp.357}, we obtain \eqref{est.cERfY}.

			Next, from the identity
			\begin{align*}
				Df_t(Y_t)Y'_t-Df_s(Y_s)Y'_s&=[Df_s(Y_t)-Df_s(Y_s)]Y'_s
				\\&\quad+[Df_t(Y_t)-Df_s(Y_t)]Y'_s +Df_t(Y_t)\delta Y'_{s,t}
			\end{align*}
			we deduce that
			\begin{align*}
				\left \|Df(Y_t)Y'_t-Df(Y_s)Y'_s|\cff_s\right \|_m
				&\le\|[Df_s]_{\gamma-2}\|_\infty |Y'_s|\big\|\delta Y_{s,t}\big|\cff_s\big\|_{m(\gamma-2)}^{\gamma-2}
				\\&\quad+\||\delta Df_{s,t}|_\infty|\cff_s\|_m|Y'_s| +\|Df_t\|_\infty\|\delta Y'_{s,t}|\cff_s\|_m\,.
			\end{align*}
			To treat the first two terms on the above right-hand side, we apply $L_{n/(\gamma-1)}$-norm and use H\"older inequalities
			\[
				\|A B^{\gamma-2}\|_{\frac n{\gamma-1}}\le\|A\|_n\|B\|_n^{\gamma-2},
				\quad\|A B\|_{\frac n{\gamma-1}}\le\|A\|_{\frac n{\gamma-2}}\|B\|_n.
			\]
			This yields
			\begin{multline*}
				\left\|\|Df_t(Y_t)Y'_t-Df_s(Y_s)Y'_s|\cff_s\|_m\right\|_{\frac n{\gamma-1}}
				\\\le\|[Df_s]_{\gamma-2}\|_\infty \sup_r\|Y'_r\|_n\|\delta Y\|_{\beta;m(\gamma-2),n}^{\gamma-2}|t-s|^{(\gamma-2)\beta}
				\\+(\bk{\delta Df}_{\beta';m,\frac n{\gamma-2}}\sup_r\|Y'_r\|_n+\|Df_t\|_\infty\|\delta Y'\|_{\beta';m,\frac n{\gamma-1}})|t-s|^{\beta'}\,.
			\end{multline*}
			Similarly
			\begin{align*}
				|f'_t(Y_t)-f'_s(Y_s)|&\le |(f'_t-f'_s)(Y_t)|+|f'_s(Y_t)-f'_s(Y_s)|
				\\&\le |\delta f'_{s,t}|_\infty+[f'_s]_{\gamma-2}|\delta Y_{s,t}|^{\gamma-2}
			\end{align*}
			and hence,
			\begin{multline}\label{tmp.080725}
				\|\|f'_t(Y_t)-f'_s(Y_s)|\cff_s\|_m\|_{\frac n{\gamma-1}}\le \bk{\delta f'}_{\beta';m,\frac n{\gamma-1}}|t-s|^{\beta'}
				\\+\left\||f'_s|_{\gamma-2}\right\|_\infty\|\delta Y\|^{\gamma-2}_{\beta;m(\gamma-2),n}|t-s|^{(\gamma-2)\beta}.
			\end{multline}
			We arrive at \eqref{est.cddf} after observing that $\|\delta Y\|_{\beta;m(\gamma-2),n}\le\|\delta Y\|_{\beta;m,n}$ and $\|\delta Y'\|_{\beta';m,\frac n{\gamma-1}}\le \|\delta Y'\|_{\beta';m,n}$.
		\end{proof}

		\begin{remark}
			\label[remark]{rem:loss}
			Unless $n=\infty$ or $Df \equiv 0$, the estimate \eqref{tmp.Efs}, and more precisely the term $\E_s|\delta Y_{s,t}|^{\gamma-1}$ therein, inevitably causes a loss of integrability from $L_{m,n}$ to $L_{m,\frac n{\gamma-1}}$ in the composition map $(Y,Y')\mapsto (f(Y),Df(Y)Y'+f'(Y))$.
		\end{remark}

	We now discuss in more detail the stability of stochastic controlled rough paths under compositions, so as to obtain local-Lipschitz estimates. Let $X$ and $\bar X$ be two $\alpha$-H\"older paths, \(\alpha\in(\frac 13,\frac12] \).
	\begin{proposition}[Stability of composition]
		\label[proposition]{prop.compose_stab}
		Let $m\in[2,\infty)$; $\gamma\in(2,3]$; $\alpha\in(\frac13,\frac12]$; $\alpha',\alpha'',\beta,\beta'\in(0,1]$  be fixed numbers. Let $X$ and $\bar X$ be two $\alpha$-H\"older paths.
		Let $(Y,Y')$ and $(\bar Y,\bar Y')$ be two elements in $ \D^{\beta,\beta'}_{X}L_{m,\infty}$ and $ \D^{\beta,\beta'}_{\bar X}L_{m,\infty}$ respectively. 
		Assume that
		\[
		\|(Y,Y')\|_{X;\beta,\beta';m,\infty}\vee\|(\bar Y,\bar Y')\|_{\bar X;\beta,\beta';m,\infty}\le M<\infty.
		\]
    Let $\kappa\in(0, \min\{\alpha,\alpha',\beta\}]$ and $\kappa'\in(0,\min\{\kappa,\alpha',\alpha'',(\gamma-2)\beta,\beta'\}]$.
		Let $(f,f')$, $(\bar f,\bar f')$ be controlled vector fields in $ \D^{\alpha,\alpha'}_XL_{m,\infty} \C^\gamma_b$ and
		$ \D^{\kappa,\kappa'}_{\bar X}L_{m,\infty} \C^{\gamma-1}_b$ respectively.
		Assume that $(Df,Df')$ belongs to  $ \D^{\alpha',\alpha''}_XL_{m,\infty} \C^{\gamma-1}_b$.
		Define
		\[
			(Z,Z')=(f(Y),Df(Y)Y'+f'(Y))
		\] and similarly for $(\bar Z,\bar Z')$. Then, recalling notations \eqref{def.scrp.metric} and \eqref{def.bk_metric}, we have the estimate
		\begin{multline}\label{est.zz}
			\norm{Z-\bar Z}_{\infty;m}+ \|Z,Z';\bar Z,\bar Z'\|_{X,\bar X;\kappa, \kappa';m}
			\lesssim
			\||Y_0-\bar Y_0|\wedge 1\|_m+\|Y,Y';\bar Y,\bar Y'\|_{X,\bar X; \kappa,\kappa';m}
			\\+\bk{f,f';\bar f,\bar f'}_{X,\bar X;\kappa,\kappa';m}+ \|(f-\bar f,f'-\bar f')\|_{\gamma-1;m}\,,
		\end{multline}
for an implicit constant which depends on \( M,T, \alpha,\alpha',\alpha'',\beta,\beta',\kappa,\kappa' \), $\bk{(f,f')}_{X;\alpha,\alpha';m,\infty}$, $\|(f,f')\|_{\gamma;\infty}$ and $\bk{(Df,Df')}_{X;\alpha',\alpha'';m,\infty}$.
	\end{proposition}
	\begin{remark} The proof relies crucially on \cref{lem.holder_mixed}, \eqref{lem:H2}, which explains why the Lipschitz estimate \eqref{est.zz} is only given in $L_m$, rather than $L_{m,\infty}$-sense.
	\end{remark} 
	\begin{proof}
		Despite its length, the proof is elementary.
		We put $\tilde Y=Y-\bar Y$, $\tilde Z=Z-\bar Z$, $\tilde f=f-\bar f$ and similarly for $\tilde Y', \tilde Z', \tilde f'$. \reply{All implicit constants herein depend on \( M,T, \alpha,\alpha',\alpha'',\beta,\beta',\kappa,\kappa' \), $\bk{(f,f')}_{X;\alpha,\alpha';m,\infty}$, $\|(f,f')\|_{\gamma;\infty}$ and $\bk{(Df,Df')}_{X;\alpha',\alpha'';m,\infty}$.}
		\smallskip

		\textit{Step 1.}	We show that
		\begin{align}
			\label{est.cdz}
			&\| Z-  \bar Z\|_{\kappa;m}\lesssim\bk{\delta\tilde f}_{\kappa;m}+\sup_{s}\||\tilde f_s|_1 \|_m+\||\tilde Y_0|\wedge1\|_m+ \|\delta\tilde Y\|_{\kappa;m},
			\\&\|Df(Y)-D\bar f(\bar Y)\|_{\kappa';m}\lesssim\bk{\delta D\tilde f}_{\kappa';m}+\sup_{s}\||D\tilde f_s|_{\gamma-2}\|_m+\||\tilde Y_0|\wedge1\|_m+\|\delta\tilde Y\|_{\kappa;m},
			\label{tmp.1108}
			\\&
			\|f'(Y)-\bar f'(\bar Y)\|_{\kappa';m}\lesssim\bk{\delta \tilde f'}_{\kappa';m}+\sup_{s}\||\tilde f'_s|_{\gamma-2} \|_m+\||\tilde Y_0|\wedge1\|_m+\|\delta\tilde Y\|_{\kappa;m}.
			\label{tmp.1109}
		\end{align}

		By triangle inequality
		\begin{align*}
			|f_s(Y_s)-\bar f_s(\bar Y_s)|
			&\le|f_s(Y_s)-f_s(\bar Y_s)|+|\tilde f_s(\bar Y_s)|
			\\&\le|f_s|_1(|Y_s-\bar Y_s|\wedge1)+|\tilde f_s|_\infty,
		\end{align*}
		which gives
		\begin{equation}\label{est.cz2}
			\|Z_s-\bar Z_s\|_m\le C(\||\tilde f_s|_\infty \|_m+\||\tilde Y_s|\wedge1\|_m).
		\end{equation}

		From $\tilde Z_t=(f_t(Y_s)-f_t(\bar Y_s))+(\delta f_t(Y)_{s,t}-\delta f_t(\bar Y)_{s,t})+\tilde f_t(\bar Y_t)$,
		we have
		\begin{equation}
		\label{decomp.delta_Z}
		\begin{aligned}
			\delta \tilde Z_{s,t}
			&= \left(\delta f_{s,t}(Y_s)-\delta f_{s,t}(\bar Y_s)\right)
			\\&\quad+ \left(f_t(Y_t)-f_t(\bar Y_t)-f_t(Y_s)+f_t(\bar Y_s)\right)
			\\&\quad+\left(\tilde f_t(\bar Y_t)-\tilde f_s(\bar Y_s)\right)=:I_1+I_2+I_3.
		\end{aligned}
		\end{equation}
		It is easy to see that $|I_1|\le|\delta Df_{s,t}|_\infty|\tilde Y_s|$ and $|I_1|\le2 |\delta f_{s,t}|_\infty$ so that
		\begin{align}\label{tmp.1133new}
			|I_1|\le 2(|\delta Df_{s,t}|_\infty+|\delta f_{s,t}|_\infty)(|\tilde Y_s|\wedge1).
		\end{align}
		Since $\tilde Y_s$ is $\cff_s$-measurable, we have
		\begin{align*}
			\|I_1\|_m\le 2\big\|\||\delta Df_{s,t}|_\infty+|\delta f_{s,t}|_\infty|\cff_s\|_m\big\|_\infty \||\tilde Y_s|\wedge1\|_m
			\lesssim\||\tilde Y_s|\wedge1 \|_m(t-s)^{\alpha\wedge \alpha'}.
		\end{align*}
		Using the elementary estimate 
		\begin{align*}
			|g(a)-g(b)-g(c)+g(d)|\le |Dg|_1(|a-c|+|b-d|)(|c-d|\wedge 1)
			+|Dg|_\infty|a-b-c+d|
		\end{align*}
		we see that
		\begin{align}\label{tmp.1056}
			|I_2|\le |Df_t|_1(|\delta Y_{s,t}|+|\delta \bar Y_{s,t}|)(|\tilde Y_s|\wedge1)+|Df_t|_\infty|\delta\tilde Y_{s,t}|.
		\end{align}
		Hence, $\|I_2\|_m\lesssim\||\tilde Y_s|\wedge1\|_m(t-s)^\beta +\|\delta\tilde Y_{s,t}\|_m$.

		It is easy to see that
		\begin{align}\label{tmp.I3new}
			|I_3|\le|\delta \tilde f_{s,t}|_\infty+|\tilde f_s|_1|\delta\bar Y_{s,t}|.
		\end{align}
		Since $|\tilde f_s|_1$ is $\cff_s$-measurable, we have
		\[
			\||\tilde f_s|_1|\delta\bar Y_{s,t}\|_m\le \||\tilde f_s|_1|\|_m\|\|\delta\bar Y_{s,t}|\cff_s\|_m\|_\infty,
		\] and hence
		\begin{align*}
			\|I_3\|_m\lesssim \bk{\delta \tilde f}_{\kappa;m}|t-s|^{\kappa}+\||\tilde f_s|_1\|_m|t-s|^\beta.
		\end{align*}

		Combining the estimates for $I_1,I_2,I_3$ and \eqref{est.cz2},
    we obtain that
    \begin{align*}
      \|\delta\tilde Z_{s,t}\|_m
      &\lesssim\left[\bk{\delta f}_{\alpha;m,\infty}+\bk{\delta Df}_{\alpha';m,\infty}+\||Df_t|_1 \|_\infty \right]\| |\tilde Y_s|\wedge1\|_m(t-s)^{\alpha\wedge \alpha'\wedge \beta}
      \\&\quad+\||Df_t|_1 \|_\infty\|\delta\tilde Y_{s,t}\|_m
      +\bk{\delta \tilde f}_{\kappa;m}|t-s|^{\kappa}+\||\tilde f_s|_1\|_m|t-s|^\beta.
    \end{align*}
    Noting that for every $s$
        \[
            \||\tilde Y_s|\wedge1\|_m\le \||\tilde Y_0|\wedge1\|_m+\|\delta \tilde Y\|_{\kappa;m}|I|^{\kappa},
        \]
    we derive \eqref{est.cdz} from the previous estimate. The estimates \eqref{tmp.1108}, \eqref{tmp.1109} are obtained analogously.
		(Here, it is necessary to replace $|Df_t|_1$ in \eqref{tmp.1056}  by $\| |D^2f_t|_{\gamma-2}\|_\infty $ and $\| |Df'_t|_{\gamma-2}\|_\infty $ respectively; replace $|\tilde f_s|_1$  in \eqref{tmp.I3new} by $|D\tilde f_s|_{\gamma-2}$ and $|\tilde f'_s|_{\gamma-2}$ respectively. This also justifies the restriction $\kappa'\le \min\{\kappa,\alpha',(\gamma-2)\beta\}$.) 
		\smallskip

		\textit{Step 2.} We show that 
		\begin{multline}\label{est.czp}
			\|\tilde Z'\|_{\kappa';m}\lesssim \bk{\delta D\tilde f}_{\kappa';m}+\bk{\delta \tilde f'}_{\kappa';m}+\sup_{s}(\||D\tilde f_s|_{\gamma-2} \|_m+\||\tilde f'_s|_{\gamma-2} \|_m)
			\\ +\| |\tilde Y_0|\wedge1\|_m  +\|\delta \tilde Y\|_{\kappa;m}+\|\tilde Y'\|_{\kappa';m}.
		\end{multline}

		It is elementary to verify that
		\begin{align}\label{est.mult}
			\|\eta \zeta \|_{\kappa';m}\le\| \eta\|_{\kappa';m}(\sup_s\|\zeta_s\|_\infty+\|\delta \zeta\|_{\kappa';m,\infty}).
		\end{align}
		From the identity
		\[
			Df( Y) Y'-D\bar f(\bar Y)\bar Y'=(Df(Y)-D\bar f(\bar Y)) Y'+D\bar f(\bar Y)(Y'-\bar Y'),
		\]
		applying \eqref{tmp.1108}, \eqref{est.mult} and the fact that $\sup_s\| Y'_s\|_\infty+\|\delta  Y'\|_{\kappa';m,\infty} $ and $\sup_s\|D\bar f_s(\bar Y_s)\|_\infty+\|\delta D\bar f( \bar Y)\|_{\kappa';m,\infty} $ are finite (by assumptions and analogous argument to \eqref{tmp.080725}),
		 we obtain
		\begin{multline*}
			\|Df( Y) Y'-D\bar f(\bar Y)\bar Y'\|_{\kappa';m}\lesssim \bk{\delta D\tilde f}_{\kappa';m}+\sup_{s}\||D\tilde f_s|_{\gamma-2} \|_m
			\\+\| |\tilde Y_0|\wedge1\|_m+\|\delta\tilde Y\|_{\kappa;m}+\|\tilde Y'\|_{\kappa';m}.
		\end{multline*}
		This estimate and \eqref{tmp.1109} yield \eqref{est.czp}.\smallskip

    \textit{Step 3.} We show that
    \begin{multline}\label{est.ERERbar}
      \bk{\E_\bigcdot R^Z-\E_\bigcdot \bar R^{\bar Z}}_{\kappa+ \kappa';m}\lesssim\||\tilde Y_0|\wedge1 \|_m+ \|\delta\tilde Y\|_{\kappa;m}+\sup_s\|\tilde Y'\|_m
      \\+\|\E_\bigcdot R^Y-\E_\bigcdot \bar R^{\bar Y}\|_{\kappa+ \kappa';m}
      +\sup_s\||\tilde f_s|_{\gamma-1} \|_m+\bk{\delta D\tilde f}_{\kappa';m}+\bk{\E_\bigcdot R^f-\E_\bigcdot\bar R^{\bar f}}_{\kappa+ \kappa';m}.
    \end{multline}
    Similar to \eqref{tmp.357}, we write
    \begin{equation}\label{id.Rfeta}
      \begin{aligned}
        R^{Z}_{s,t}
        =\Taylor f_s(Y_s,Y_t) +Df_s(Y_s)[R^{Y}_{s,t}]
        +R^{ f}_{s,t}(Y_s)
        +(\delta f_{s,t}(Y_t)-\delta f_{s,t}(Y_s))
      \end{aligned}
    \end{equation}
    where
    \begin{align*}
      \Taylor h(\xi,\eta)= h(\eta)- h(\xi)-D h(\xi)[\eta- \xi].
    \end{align*}
    We decompose $R^{\bar Z}_{s,t}$ in an analogous way.
    We estimate separately the differences of the corresponding terms on the right-hand sides of the two decompositions.

    We have
    \begin{align*}
      \Taylor f_s(Y_s,Y_t)-\Taylor\bar f_s(\bar Y_s,\bar Y_t)=\Taylor\tilde f_s(\bar Y_s,\bar Y_t)+\Taylor f_s(Y_s,Y_t)-\Taylor f_s(\bar Y_s,\bar Y_t).
    \end{align*}
    By Taylor's expansion, it is evident that
    $  |\Taylor\tilde f_s(\bar Y_s,\bar Y_t)|\le |\tilde f_s|_{\gamma-1}|\delta\bar Y_{s,t}|^{\gamma-1}$,
    and hence,
    \begin{align*}
      \|\E_s\Taylor\tilde f_s(\bar Y_s,\bar Y_t)\|_m\le \||\tilde f_s|_{\gamma-1}\|_m(t-s)^{(\gamma-1)\beta}.
    \end{align*}
    Next, we put $ Y^\theta=\theta Y+(1- \theta)\bar Y$. 
	We apply the fundamental theorem of calculus to get that
    \begin{align*}
      &\Taylor f_s(Y_s,Y_t)-\Taylor f_s(\bar Y_s,\bar Y_t)
      =\int_0^1\frac{d}{d \theta}\left( f_s( Y^\theta_t)- f_s( Y^\theta_s)-Df_s( Y^\theta_s)[\delta  Y^\theta_{s,t}] \right)d \theta
      \\&=\int_0^1\left( Df_s( Y^\theta_t)[\tilde Y_t]-Df_s( Y^\theta_s)[\tilde Y_s]-Df_s( Y^\theta_s)[\delta\tilde Y_{s,t}]-D^2 f_s( Y^\theta_s)[\tilde Y_s,\delta  Y^\theta_{s,t}]\right)d \theta
      \\&=\int_0^1\left(Df_s( Y^\theta_t)[\tilde Y_s]-Df_s( Y^\theta_s)[\tilde Y_s]-D^2 f_s( Y^\theta_s)[\tilde Y_s,\delta  Y^\theta_{s,t}]\right)d \theta
      \\&\quad+\int_0^1\left( Df_s( Y^\theta_t)[\delta\tilde Y_{s,t}]-Df_s( Y^\theta_s)[\delta\tilde Y_{s,t}]\right)d \theta.
    \end{align*}
    Using the fact that $Df_s \in \C^{\gamma-1}_b$, we get that
    \begin{align*}
      |\Taylor f_s(Y_s,Y_t)-\Taylor f_s(\bar Y_s,\bar Y_t)|
      &\lesssim |Df_s|_{\gamma-1}\int_0^1\left(|\tilde Y_s||\delta  Y^\theta_{s,t}|^{\gamma-1}+|\delta  Y^\theta_{s,t}||\delta \tilde Y_{s,t}| \right)d \theta
      \\&\lesssim|\tilde Y_s|(|\delta Y_{s,t}|^{\gamma-1}+|\delta \bar Y_{s,t}|^{\gamma-1})+|\delta \tilde Y_{s,t}|(|\delta Y_{s,t}|+|\delta \bar Y_{s,t}|).
    \end{align*}
    On the other hand, we also have $|\Taylor f_s(Y_s,Y_t)-\Taylor f_s(\bar Y_s,\bar Y_t)|\lesssim |\delta Y_{s,t}|^{\gamma-1}+|\delta \bar Y_{s,t}|^{\gamma-1}$.
    Thus, we have
    \begin{align*}
      |\Taylor f_s(Y_s,Y_t)-\Taylor f_s(\bar Y_s,\bar Y_t)|
      &\lesssim(|\tilde Y_s|\wedge1)(|\delta Y_{s,t}|^{\gamma-1}+|\delta \bar Y_{s,t}|^{\gamma-1})+|\delta \tilde Y_{s,t}|(|\delta Y_{s,t}|+|\delta \bar Y_{s,t}|)
    \end{align*}
    which, in view of \cref{lem.holder_mixed} and the assumed regularity of $Y,\bar Y$, implies that
    \begin{align*}
      \|\E_s(\Taylor f_s(Y_s,Y_t)-\Taylor f_s(\bar Y_s,\bar Y_t))\|_m
      \lesssim \||\tilde Y_s|\wedge1\|_m(t-s)^{(\gamma-1)\beta}+\|\delta \tilde Y_{s,t}\|_m(t-s)^{\beta}.
    \end{align*}
    It follows that
    \begin{multline*}
      \|\E_s(\Taylor f_s(Y_s,Y_t)-\Taylor\bar f_s(\bar Y_s,\bar Y_t))\|_m
      \\\lesssim (\||\tilde f_s|_{\gamma-1}\|_m+\||\tilde Y_s|\wedge1\|_m)(t-s)^{(\gamma-1)\beta}+\|\delta \tilde Y_{s,t}\|_m(t-s)^\beta.
    \end{multline*}

    For the difference corresponding to the second term in \eqref{id.Rfeta}, we note that
    \begin{multline*}
      \E_s\left(Df_s(Y_s)[R^Y_{s,t}]-D{\bar f}_s(\bar Y_s)[\bar R^{\bar Y}_{s,t}]\right)
      \\=D \tilde f_s(\bar Y_s)[\E_s\bar R^{\bar Y}_{s,t}]
      +\left(Df_s(Y_s)[\E_sR^Y_{s,t}]-Df_s(\bar Y_s)[\E_s\bar R^{\bar Y}_{s,t}]\right).
    \end{multline*}
	Noting that $D \tilde f_s$ is bounded and $Df_s$ is Lipschitz and bounded, we have
    \begin{multline*}
      |\E_s(Df_s(Y_s)[R^Y_{s,t}]-D{\bar f}_s(\bar Y_s)[\bar R^{\bar Y}_{s,t}])|
      \lesssim|D \tilde f_s|_\infty|\E_s\bar R^{\bar Y}_{s,t}|
      \\+|Df_s|_{1}(|\tilde Y_s|\wedge1)|\E_s\bar R^{\bar Y}_{s,t}|+|Df_s|_\infty|\E_s(R^Y_{s,t}-\bar R^{\bar Y}_{s,t})|.
    \end{multline*}
    Taking into account the regularity of $Y,\bar Y$, we have
    \begin{multline*}
      \|\E_s(Df_s(Y_s)[R^Y_{s,t}]-D{\bar f}_s(\bar Y_s)[\bar R^{\bar Y}_{s,t}])\|_m
      \\\lesssim(\||D \tilde f_s|_\infty \|_m+\||\tilde Y_s|\wedge1\|_m)(t-s)^{\beta+\beta'}+\|\E_s(R^Y_{s,t}-\bar R^{\bar Y}_{s,t})\|_m.
    \end{multline*}

    For the difference corresponding to the third term in \eqref{id.Rfeta}, we write
    \begin{align*}
      R^{ f}_{s,t}(Y_s)-\bar R^{{\bar f}}_{s,t}(\bar Y_s)
      &=(R^{ f}_{s,t}-\bar R^{{\bar f}}_{s,t})(\bar Y_s)
       +R^{ f}_{s,t}(Y_s) -R^{ f}_{s,t}(\bar Y_s).
    \end{align*}
    We note that
    \begin{align*}
      |\E_s(R^{ f}_{s,t}(Y_s) -R^{ f}_{s,t}(\bar Y_s))|\le 2\bk{\E_\bigcdot R^f}_{\alpha+\alpha';\infty} (t-s)^{\alpha+\alpha'}
    \end{align*}
    and by the fundamental theorem of calculus that
    \begin{align*}
      |\E_s(R^{ f}_{s,t}(Y_s) -R^{ f}_{s,t}(\bar Y_s))|
      &=\Big| \E_s \int_0^1R^{Df}_{s,t}(\theta Y_s+(1- \theta)\bar Y_s)[\tilde Y_s]d \theta\Big|
      \\&\le \bk{\E_\bigcdot R^{Df}}_{\alpha'+\alpha'';\infty}
       (t-s)^{\alpha'+\alpha''}|\tilde Y_s|.
    \end{align*}
    Combining the previous inequalities, we obtain that
    \begin{align*}
      \||\E_s(R^{ f}_{s,t}(Y_s)-\bar R^{{\bar f}}_{s,t}(\bar Y_s))|_\infty \|_m
      \lesssim \|\E_s(R^{ f}_{s,t}-\bar R^{{\bar f}}_{s,t}) \|_m
      +\||\tilde Y_s|\wedge1\|_m(t-s)^{\alpha'+\alpha\wedge \alpha''}.
    \end{align*}

    For the difference corresponding to the last term in \eqref{id.Rfeta}, we write
    \begin{multline*}
      \delta f_{s,t}(Y_t)-\delta f_{s,t}(Y_s)
      -(\delta {\bar f}_{s,t}(\bar Y_t)-\delta {\bar f}_{s,t}(\bar Y_s))
      \\=\delta \tilde f_{s,t}(\bar Y_t)-\delta \tilde f_{s,t}(\bar Y_s)
      +\left[\delta f_{s,t}(Y_t)-\delta f_{s,t}(Y_s)
          -(\delta f_{s,t}(\bar Y_t)-\delta f_{s,t}(\bar Y_s))\right].
    \end{multline*}
    Similar to \eqref{tmp.1056}, we have
    \begin{multline*}
      |\delta f_{s,t}(Y_t)-\delta f_{s,t}(Y_s)
      -(\delta {\bar f}_{s,t}(\bar Y_t)-\delta {\bar f}_{s,t}(\bar Y_s))|
      \\\lesssim |\delta D \tilde f_{s,t}|_{\infty}|\delta \bar Y_{s,t}|+|\delta D f_{s,t}|_{1}(|\delta Y_{s,t}|+|\delta\bar Y_{s,t}|)(|\tilde Y_s|\wedge1)+|\delta D f_{s,t}|_{\infty}|\delta \tilde Y_{s,t}|.
    \end{multline*}
	Applying \cref{lem.holder_mixed} (mixed H\"older estimates) to estimate the conditional moments of the right-hand side, we obtain that
	\begin{align*}
		\|\E_s[|\delta D \tilde f_{s,t}|_{\infty}|\delta \bar Y_{s,t}|]\|_m&\le \||\delta D \tilde f_{s,t}|_{\infty}\|_m \times \|\|\delta \bar Y_{s,t}|\cff_s\|_m\|_\infty
		,\\\E_s[|\delta D f_{s,t}|_{1}(|\delta Y_{s,t}|+|\delta\bar Y_{s,t}|)]
		&\le \|\||\delta D f_{s,t}|_{1}|\cff_s\|_m\|_\infty \times \|\|(|\delta Y_{s,t}|+|\delta\bar Y_{s,t}|)|\cff_s \|_m\|_\infty
		,\\\|\E_s|\delta D f_{s,t}|_{\infty}|\delta \tilde Y_{s,t}|\|_m&\le \|\||\delta D f_{s,t}|_{\infty}|\cff_s\|_m\|_\infty \times \|\delta \tilde Y_{s,t}\|_m. 
	\end{align*}
    Taking into account  the estimate $ |{\delta Df_{s,t}}|_1
  \le  |\delta D^2 f_{s,t}|_\infty
      +  |\delta Df_{s,t}|_\infty$ and  the assumped regularity of $f$ and $Y,\bar Y$, we deduce that
    \begin{multline*}
      \|\E_s[\delta f_{s,t}(Y_t)-\delta f_{s,t}(Y_s)
              -(\delta {\bar f}_{s,t}(\bar Y_t)-\delta {\bar f}_{s,t}(\bar Y_s))]\|_m
      \\\lesssim \||\delta D \tilde f_{s,t}|_{\infty}\|_m(t-s)^\beta +\||\tilde Y_s|\wedge1\|_m(t-s)^{\alpha'\wedge\alpha''+\beta}+\|\delta \tilde Y_{s,t}\|_m(t-s)^{\alpha'}.
    \end{multline*}

    Summing up the estimates for all the differences, we obtain that
    \begin{equation}\label{tmp.erer}
      \begin{aligned}
      \|\E_s R^Z_{s,t}&-\E_s\bar R^{\bar Z}_{s,t}\|_m
      \lesssim
      \||\tilde Y_s|\wedge1\|_m(t-s)^{\min((\gamma-1)\beta,\beta+\beta',\alpha'+\alpha\wedge \alpha'',\alpha'\wedge\alpha''+\beta)}
     \\&\quad +\|\delta \tilde Y_{s,t}\|_m(t-s)^{\min(\alpha',\beta)}
      +\|\E_s(R^Y_{s,t}-\bar R^{\bar Y}_{s,t})\|_m
      \\&\quad+\||\tilde f_s|_{\gamma-1}\|_m(t-s)^{\min((\gamma-1)\beta,\beta+\beta')}
      +\||\delta D \tilde f_{s,t}|_{\infty}\|_m(t-s)^\beta +\|\E_s(R^{ f}_{s,t}-\bar R^{{\bar f}}_{s,t}) \|_m,
      \end{aligned}
    \end{equation}
        which implies \eqref{est.ERERbar}.\smallskip

		\textit{Conclusion.} Combining \eqref{est.cdz}, \eqref{est.czp}, \eqref{est.ERERbar} we obtain \eqref{est.zz}.
	\end{proof}


\section{Rough stochastic differential equations} 
\label{sec.rough_stochastic_differential_equations}
	Let $\W=(\Omega,\mathcal G,\P;\{\cff_t\})$ be a stochastic basis, $B$ be a standard $\{\cff_t\}$-Brownian motion in $\Vone$, $\X=(X,\XX) $ be a deterministic rough path in $\mathscr{C}^\alpha(V)$  with $\alpha\in(\frac13,\frac12]$. We consider the rough stochastic differential equation
	\begin{equation}\label{eqn.srde}
		dY_t(\omega)=b_t(\omega,Y_t(\omega))dt+\sigma_t(\omega,Y_t(\omega))dB_t(\omega)+(f_t,f'_t)(\omega,Y_t(\omega))d\X_t,\quad t\in[0,T].
	\end{equation}
	We are given a drift vector field
	$b\colon \Omega\times [0,T]\times W\to W$, and vector fields
	$\sigma\colon\Omega\times[0,T]\times W\to \mathcal L(\Vone,W)$,
	$f\colon \Omega\times[0,T]\times W\to \mathcal L(V,W)$, $f'\colon \Omega\times[0,T]\times W\to \mathcal L(V\otimes V,W)$.
	We assume further that $b,\sigma,f,f'$ are progressively measurable (as functions of $\omega,t$, fixed $y$) and joint measurability in $(\omega,t,y)$, 
	and that for each $t$ and a.s. $\omega$, $y\mapsto f_t(\omega,y)$ is differentiable with derivative $Df_t(\omega,y)$.
	In what follows, we omit the $\omega$-dependence in the coefficients $\sigma,b,f,f',Df$.
	We assume moreover
	that $\sigma,b$ are random bounded continuous functions, in the sense of the next definition.
	\begin{definition}\label[definition]{def.randomcoef}
		Let $W,\bar W$ be some finite dimensional Euclidean spaces and fix a Borel set \( S\subset W \). Let \( (t,\omega)\mapsto g_t(\omega,\cdot) \) be a progressively measurable stochastic process from \( \Omega\times [0,T] \to \Cb(S;\bar W)\) (in the sense of a family of strongly measurable random variables as defined in \cref{sec.preliminaries}).
We say that:
		\begin{enumerate}[label=(\alph*)]
			\item $g$ is \textit{random bounded continuous} if 
			is uniformly bounded, namely, there exists a deterministic constant $\|g\|_\infty$ such that
				\[
					\sup_{t\in[0,T]}\esssup_{\omega\in \Omega}\sup_{x\in S}|g_t(\omega,x)|\le\|g\|_\infty.
				\]
			\item $g$ is \textit{random bounded Lipschitz} if it is random bounded continuous, progressively measurable from \( \Omega\times[0,T]\to \Cb^1 (S;\bar W) \) and uniformly bounded in the sense that
				\[
					\sup_{t\in[0,T]}\esssup_{\omega\in \Omega}\sup_{x,\bar x\in S}\frac{|g_t(\omega,x)-g_t(\omega,\bar x)|}{|x-\bar x|}\le \|g\|_\lip
				\]
				for some constant \( \|g\|_{\lip} \).
		\end{enumerate}
	\end{definition}
	We give the definition of $L_{m,n}$-integrable solutions, make in particular use of the space of stochastic controlled rough paths from \cref{def.SCRP}.
	\begin{definition}[Integrable solutions]
		\label[definition]{def.soln}
		Let $m,n$ be (extended) real numbers such that $m\in[2,\infty)$ and $n\in[m,\infty]$.
		An $L_{m,n}$-integrable solution of \eqref{eqn.srde}  over $[0,T]$ is a continuous $\{\cff_t\}$-adapted process $Y$ such that the following conditions are satisfied
		\begin{enumerate}[label=(\alph*)]
			\item\label{def.conbs} $\int_0^T |b_r(Y_r)|dr$ and $\int_0^T |(\sigma \sigma^\dagger)_r(Y_r)|dr$ are finite a.s.;
			\item\label{def.confY} $(f(Y),Df(Y)f(Y)+f'(Y))$ belongs to $\D_X^{\bar\alpha,\bar\alpha'}L_{m,n}([0,T],\W;\cll(V,W))$
			for some%
			\footnote{Take $\alpha = \bar\alpha =\bar\alpha'$ at first reading.}
			\[
			\bar\alpha,\bar\alpha'\in(0,1]:\enskip
 \alpha+(\alpha\wedge \bar\alpha)>\frac12,\enskip
 \alpha+(\alpha\wedge\bar\alpha)+\bar\alpha'>1;\]
			\item\label{def.Davie.expansion} $Y$ satisfies the following stochastic Davie-type expansion
			\begin{equation}\label{est.defJ}
				\|\|J_{s,t}|\cff_s\|_m\|_n=o(t-s)^{1/2}
				\tand\|\E_sJ_{s,t}\|_n=o(t-s)
			\end{equation}
			for every $(s,t)\in \Delta$, where
			\begin{align}
				J_{s,t}
				&=\delta Y_{s,t}-\int_s^tb_r(Y_r)dr-\int_s^t \sigma_r(Y_r)dB_r
				\nonumber\\&\quad-f_s(Y_s)\delta X_{s,t}-\left(Df_s(Y_s)f_s(Y_s)+f'_s(Y_s)\right)\XX_{s,t}.\label{def.J}
			\end{align}
		\end{enumerate}
		When the initial datum $Y_0=\xi$ is specified, we say that $Y$ is a solution starting from $\xi$.
	\end{definition}

	We begin by showing that a solution to \eqref{eqn.srde} satisfies an integral equation, therefore, providing a dynamical description which is equivalent to the local description of \cref{def.soln}.
	The deterministic counterpart of this characterization appears in \cite{MR2387018}.
	\begin{proposition}\label[proposition]{prop.davie}
		$Y$ is an $L_{m,n}$-integrable solution of \eqref{eqn.srde} if and only if \ref{def.conbs}-\ref{def.confY} of \cref{def.soln} hold and for $\P$-a.s. $\omega$,
		\begin{equation}\label{eqn.int.srde}
			Y_t=Y_0+\int_0^t b_r(Y_r)dr+\int_0^t \sigma_r(Y_r)dB_r+\int_0^tf_r (Y_r)d\X_r \text{ for all }t\in[0,T].
		\end{equation}
		Furthermore, in this case, we have for $J$ as in \cref{def.soln} and for any $(s,t)\in\Delta$,
		\begin{equation}\label{est.defJex}
			\|\|J_{s,t}|\cff_s\|_m\|_n\lesssim |t-s|^{\alpha+(\alpha\wedge \bar\alpha)}
			\tand\|\E_sJ_{s,t}\|_n\lesssim |t-s|^{\alpha+(\alpha\wedge \bar\alpha)+\bar\alpha'}.
		\end{equation}
	\end{proposition}
	\begin{proof}
	Assume first that \ref{def.conbs} and \ref{def.confY} of \cref{def.soln} hold.
	Define $A_{s,t}=f_s(Y_s)\delta X_{s,t}+(Df_s(Y_s)f_s(Y_s)+f'_s(Y_s))\XX_{s,t}$ and
	\[
		Z_t=Y_t-Y_0-\int_0^tb_r(Y_r)dr-\int_0^t \sigma_r(Y_r)dB_r \,.
	\]
	Since  $(f(Y),Df(Y)f(Y)+f'(Y))$ belongs to $ \D^{\bar\alpha,\bar\alpha'}_XL_{m,n}$, we can apply \cref{thm.Hroughint} to define the rough stochastic integral $\caa_\cdot:=\int_0^\cdot f(Y)d\X$ which then satisfies
	\begin{align}\label{tmp.1224}
	 	\|\|\delta\caa_{s,t}-A_{s,t}|\cff_s\|_m\|_{n}\lesssim |t-s|^{\alpha+(\alpha\wedge \bar\alpha)}
	 	\text{ and }
	 	\|\E_s (\delta\caa_{s,t}-A_{s,t})\|_{n}\lesssim|t-s|^{\alpha+(\alpha\wedge \bar\alpha)+\bar\alpha'}
	\end{align}
	for every $(s,t)\in \Delta$. Now, suppose that $Y$ is a $L_{m,n}$-integrable solution.
	We can combine \eqref{tmp.1224} with \eqref{est.defJ} and \eqref{def.J} to obtain that
	\begin{align*}
		\|\|\delta Z_{s,t}- \delta\caa_{s,t}|\cff_s\|_m\|_{n}=o(|t-s|)^{\frac12}
		\quad\textrm{and}\quad
		\|\E_s(\delta Z_{s,t}-\delta\caa_{s,t})\|_{n}=o(|t-s|)\,.
	\end{align*}
	The previous estimates imply (by \cite[Lemma 3.5]{LeSSL2}) that $Z_t=\caa_t$ a.s. for every $t\in[0,T]$. Since both $Z$ and $\caa$ are continuous, they are indistinguishable, which means that \eqref{eqn.int.srde} holds. This shows the necessity.

	Sufficiency is evident from the fact that if \eqref{eqn.int.srde} holds then together with \eqref{tmp.1224}, it implies \eqref{est.defJ}. That $Y$ is a.s. continuous is evident from \eqref{eqn.int.srde}. Hence, we have shown that $Y$ is an $L_{m,n}$-integrable solution. At last, observing that \eqref{tmp.1224} implies \eqref{est.defJex}, we conclude the proof.
	\end{proof}
	\begin{remark}\label[remark]{rmk.JwhenBounded}
		When $\X$ belongs to $\CC^\alpha$, $\alpha\in(\frac13,\frac12]$ and
		\begin{align*}
		\sup_{s\in[0,T]}\|g_s(Y_s)\|_n<\infty ,\quad\forall g\in\{b,\sigma \sigma^\dagger,f,Dff,f'\},
		\end{align*}
		the estimates in \eqref{est.defJex} and \eqref{def.J} imply that
		\begin{equation}\label{k3b}
		\|\|\delta Y_{s,t}|\cff_s\|_m \|_{n}\lesssim|t-s|^{ \alpha}
		\tand
		\left \|\E_s\left(\delta Y_{s,t}-f_s(Y_s)\delta X_{s,t}\right)\right \|_{n} \lesssim|t-s|^{2\alpha}\,.
		\end{equation}
		In this case, any $L_{m,n}$-integrable solution to \eqref{eqn.srde} satisfying $f(Y)\in C^\beta L_{m,n}$ for some $\beta\in(0,\alpha]$ necessarily belongs to $ \D^{\alpha,\beta}_XL_{m,n}$.
	\end{remark}

	Each $L_{m,\infty}$-solution is bounded in the following sense.
	\begin{proposition}[A priori estimates]\label[proposition]{prop.apri}
			Suppose that $b,\sigma$ are random bounded continuous  and $(f,f')$ belongs to $ \D^{\beta,\beta'}_XL_{m,\infty}\C^{\gamma-1}_b$ with $\beta\in(0,\alpha]$, $\beta'\in(0,1]$ and $\gamma\in(2,3]$ such that $\alpha+(\gamma-1)\beta>1$ and $\alpha+\beta+\beta'>1$.
			Let $Y$ be an $L_{m,\infty}$-solution to \eqref{eqn.srde}
			and take any finite constant $M$ such that
			\begin{align*}
				&|f|_{\gamma-1;\infty;[0,T]}+\bk{\delta f}_{\beta;m,\infty;[0,T]}+\bk{\EE R^f}_{\beta+\beta';m,\infty;[0,T]}\le M,
				\\
				\shortintertext{and}
				&|f'|_{\gamma-2;\infty;[0,T]}+\bk{\delta f'}_{\beta';m,\infty;[0,T]}\le M^2.
			\end{align*}
			Define
			\begin{align*}
				K=&1+\|b\|_\infty+\|\sigma\|_\infty+M|\delta X|_\alpha+M^2|\XX|_\beta.
			\end{align*}
			Then, there exists a constant $C$ depending only on $T,m,\alpha,\beta,\beta',\gamma$ such that
			\begin{gather}\label{est.apri.m}
				\|\delta  Y\|_{\alpha;m,\infty;[0,T]}+\|\E_\bigcdot R^Y\|_{\alpha+\beta;\infty;[0,T]}
				\le CK^{2+2/\beta''},
			\end{gather}
			where $\beta''=\min\{(\gamma-2)\beta,\beta'\}$ and $R^{Y}=\delta Y-f(Y)\delta X$.
			Furthermore, we have
			\begin{gather}
				\|(Y,Y')\|_{X;\alpha,\beta;m,\infty}\le \bar C(1+M) K^{2+2/\beta''},
				\label{est.apri.YinD}
				\\
				\|(Z,Z')\|_{X;\beta,\beta'';m,\infty}\le \bar CM(1+M)^\gamma K^{(\gamma-1)(2+2/\beta'')},
				\label{est.apriZinD}
			\end{gather}
			where  $Y'=Z=f(Y)$, $Z'=(Df(Y)f(Y)+f'(Y))$ and $\bar C$  is a constant depending only on $T,m,\alpha,\beta,\beta',\gamma$.
	\end{proposition}
		\begin{proof}
			Inequality \eqref{est.apri.YinD} is a consequence of \eqref{est.apri.m} and 
			\begin{equation*}
				\|\|\delta Z_{s,t}|\cff_s\|_m\|_\infty\le(\bk{\delta f}_{\beta;m,\infty}+ \|Df_s\|_\infty\|\delta Y\|_{\beta;m,\infty})|t-s|^\beta\,.
			\end{equation*}
			Inequality \eqref{est.apriZinD} is a direct consequences of \eqref{est.apri.YinD} and \eqref{est.ZinD}. Hence, it suffices to show \eqref{est.apri.m}.
			To this aim, our strategy is to obtain a closed argument from \cref{thm.Hroughint,lem.composecvec}.  Without loss of generality, we can and will assume that $\beta'\le \beta$\footnote{When $\beta'\ge \beta$,  the condition $\alpha+\beta+\beta'>1$ is implied by $\alpha+(\gamma-1)\beta>1$ and $\gamma\le 3$}. Moreover, by working with $(MX,M^2\XX)$ and $(f/M,f'/M^2)$ instead of $(X,\XX)$ and $(f,f')$, we can also assume that $M=1$. In this case, we have $\bk{(f,f')}_{X;\beta;m,\infty}+\|(f,f')\|_{\gamma-1;\infty}\le 1$ and $K=1+\rho_{\alpha,\beta}(\X)+\|b\|_\infty+\|\sigma\|_\infty$. 
			All implicit constants herein depend only on $T,m,\alpha,\beta,\beta',\gamma$.\smallskip

\reply{Before starting the proof, let us briefly explain the bootstrap argument used
below, adapted from \cite[Chapter 8.4]{FH20}. The local estimate obtained in Step~1 below is
quadratic in the relevant H\"older seminorm: writing $x_h$ for (a normalization of) this
seminorm restricted to intervals of length at most $h$, one arrives at an inequality of
the form $x_h\le K+\varepsilon_hx_h^2$ with $\varepsilon_h\downarrow0$ as $h\downarrow0$.
Plotting both sides as functions of $x_h$, this constrains $x_h$ to lie either below a
finite threshold close to $K$, or above a second threshold diverging to $\infty$ as
$h\downarrow0$: the corresponding quadratic inequality has both a bounded and an
unbounded admissible branch. Since $x_h$ cannot jump between the two branches --
by an elementary doubling estimate $x_h\le2x_{h/2}$ -- and is forced onto the bounded
branch at sufficiently fine scales (where $x_h$ stays finite while the unbounded
threshold diverges), it must remain on the bounded branch at every scale up to some
$h_0$, yielding the quantitative estimate below. We note a related analysis of polynomial 
inequalities was recently formalized as the ``Flying Fish'' lemma of \cite[Lemma~4.6]{LY25}.
}

\smallskip

			\textit{Step 1: local estimates.}
			As noted in \cref{rmk.JwhenBounded}, the fact that $Y$ is a solution together with regularity of coefficients implies that $(Y,f(Y))$ belongs to \( \D^{\alpha,\beta}_{X}L_{m,\infty}\subset \D^{\beta,\beta'}_{X}L_{m,\infty}\). 
			Applying \cref{lem.composecvec}, we see that \((Z,Z')=(f(Y),Df(Y)f(Y)+f'(Y))\) is a stochastic controlled rough path in \( \D^{\beta,\beta''}_{X}L_{m,\infty}\) with \( \beta''=\min((\gamma-2)\beta,\beta')\).
			Additionally, noting $\gamma\le 3$,  we obtain from the estimates \eqref{est.cfy}-\eqref{est.cddf} in the proof of \cref{lem.composecvec} that 
			\begin{align*}
				&\|\delta Z\|_{\beta;m,\infty}\lesssim(1\vee \|\delta Y\|_{\beta;m,\infty}),
				\\
				&\|\EE R^Z\|_{\beta+\beta'';\infty} \lesssim   \|\EE R^Y\|_{\beta+\beta';\infty} +(1\vee \|\delta Y\|_{\beta;m,\infty})^{2},
				\\ &\|\delta Z'\|_{\beta'';m,\infty}\lesssim(1\vee \|\delta Y\|_{\beta;m,\infty}),
			\end{align*}
			and hence, noting that $\|Z'_t\|_{\infty}\le 2$, 
			\begin{align*}
				\|R^Z\|_{\beta;m,\infty}
				&\lesssim (1\vee \|\delta Y\|_{\beta;m,\infty})+|\delta X|_\alpha.
			\end{align*}
			These estimates imply that 
			\begin{multline}
				\Gamma_1([s,t]):=\Gamma_1^{\beta,\beta'';m,\infty}(\X,\delta Z',R^Z;[s,t])
				\\\lesssim 
				K
				\big(\|\E_{\bigcdot} R^Y\|_{\beta+ \beta';\infty;[s,t]}
				+(1\vee\|\delta Y\|_{\beta;m,\infty;[s,t]})^{2}\big)
				\label{est.Gamma1}	
			\end{multline}
			and 
			\begin{align}
				\Gamma_2([s,t]):=\Gamma_2^{\beta,\beta'';m,\infty}(\X,\delta Z',R^Z;[s,t])\lesssim K(1\vee \|\delta Y\|_{\beta;m,\infty;[s,t]})+K^2.
				\label{est.Gamma2}
			\end{align}

			Next,  we estimate $\|\EE R^Y\|_{\beta+\beta';\infty;[s,t]}$. We put 
			\begin{align*}
				J:=\int Zd\X-Z\delta X-Z'\XX=R^Y-Z'\XX-\int b(Y)dr-\int \sigma(Y)dB_r,
			\end{align*}
			where the second identity follows from \eqref{eqn.int.srde} (applicable because our assumptions on $\beta,\beta'$ ensure that $\alpha+\beta>1/2$ and $\alpha+\beta+\beta''>1$).
			We have
			\begin{align*}
				\|\EE R^Y\|_{\alpha+\beta;\infty;[s,t]}\lesssim \|\EE J\|_{\alpha+\beta+\beta'';\infty;[s,t]}(t-s)^{\beta''}+K.
			\end{align*}
			By \cref{thm.Hroughint} (with $\beta'=\beta''$ therein) and \eqref{est.Gamma1}, we have 
			\begin{align*}
				\|\EE J\|_{\alpha+\beta+\beta'';m,\infty;[s,t]}\lesssim \Gamma_1([s,t])
				\lesssim K
				\big(\|\E_{\bigcdot} R^Y\|_{\beta+ \beta';\infty;[s,t]}+(1\vee\|\delta Y\|_{\beta;m,\infty;[s,t]})^{2}
				\big)
			\end{align*}
			Hence, we obtain
			\begin{equation}\label{tmp.imp0}
				\|\E_\bigcdot R^{Y}\|_{\alpha+\beta;\infty;[s,t]}
				\lesssim K(t-s)^{\beta''}(\|\E_\bigcdot R^{Y}\|_{\beta+\beta';\infty;[s,t]}+(1\vee\|\delta Y\|_{\beta;m,\infty;[s,t]})^{2})
				+K.
			\end{equation}
			So if $t-s\le \ell$ for some sufficiently small $\ell\in(0,1)$ such that
			\begin{equation}\label{con.smallness0}
				K\ell^{\beta''}\ll1,
			\end{equation}
			we derive from \eqref{tmp.imp0} that
			\begin{equation}\label{tmp.1149}
				\|\E_\bigcdot R^{Y}\|_{\beta+\beta';\infty;[s,t]}\lesssim
				(1\vee\|\delta Y\|_{\beta;m,\infty;[s,t]})^{2}+K.
			\end{equation}

			From the defining identity for $J$, 
			we apply the bounds
			\begin{align*}
				|\int_{s}^{t}b_r(Y_r)dr|\le\|b\|_\infty|t-s|,
				\quad
				\Big\|\Big\|\int_{s}^{t}\sigma_r(Y_r)dB_r\Big|\cff_{s}\Big\|_m\Big\|_\infty\lesssim\|\sigma\|_\infty|t-s|^{\frac12}
			\end{align*}
			to obtain that
			\begin{align*}
				\|\delta Y\|_{\alpha;m,\infty;[s,t]}
				\lesssim \|J\|_{\alpha+\beta;m,\infty;[s,t]}(t-s)^\beta+ K.
			\end{align*}
			By \cref{thm.Hroughint}, \eqref{est.Gamma1}, \eqref{est.Gamma2} and \eqref{tmp.1149}, we have 
			\begin{align*}
				\|J\|_{\alpha+\beta;m,\infty;[s,t]}
				&\lesssim \Gamma_1([s,t])+\Gamma_2([s,t])
				\\&\lesssim K
				(1\vee\|\delta Y\|_{\beta;m,\infty;[s,t]})^{2}+ K^2.
			\end{align*}
			Altogether, using \eqref{con.smallness0}, we see that
			\begin{align*}
				(\|\delta Y\|_{\alpha;m,\infty;[s,t]}\vee1)
				\lesssim(\|\delta Y\|_{\alpha;m,\infty;[s,t]}\vee1)^{2} K \ell^{\beta''}+K
			\end{align*}
			for every $(s,t)\in \Delta$ satisfying $0\le t-s\le\ell$.
			\reply{
			Define
\begin{equation*}
	x_h
	:=
	1\vee
	\sup_{\substack{0\leq u<v\leq T\\ v-u\leq h}}
	\frac{
		\big\|\|\delta Y_{u,v}|\cff_u\|_m\big\|_\infty
	}{
		|v-u|^\alpha
	}.
\end{equation*}
Taking \(\ell=h\) in the preceding estimates and then taking the supremum
over all intervals of length at most \(h\), we obtain, whenever
\(c_*Kh^{\beta''}\leq 1\), with suitable constant \(c_*\geq 1\), 
\begin{equation}\label{est.quadratic.bootstrap}
	x_h
	\leq
	c_*Kh^{\beta''}x_h^2+c_*K \equiv \varepsilon_h x_h^2 + A
\end{equation}
It is easy to see that
\begin{equation}\label{est.doubling.bootstrap}
x_h\leq 2^{1-\alpha}x_{h/2}\leq 2x_{h/2}.	
\end{equation}
Choose \(c_0>0\), depending only on
\(T,m,\alpha,\beta,\beta'\) and \(\gamma\), sufficiently small, and further set $\ell_0:=c_0K^{-2/\beta''}$.
We may choose \(c_0\) so that \(\ell_0\leq T\wedge1\) and
\begin{equation}\label{est.branch.smallness}
	A\varepsilon_h
	\leq
	A\varepsilon_{\ell_0}
	\leq
	\frac{1}{16},
	\qquad
	0<h\leq\ell_0.
\end{equation}
In particular, \eqref{con.smallness0} is satisfied on these scales. Write $r_\pm (h)$ for the roots of the quadratic polynomial corresponding to \eqref{est.quadratic.bootstrap}. Plug therein $2A$ and $8A$ to see
%
%
\begin{equation}\label{est.branch.bounds}
	r_-(h)<2A
	\qquad\text{and}\qquad
	r_+(h)>8A,
	\qquad 0<h\le\ell_0.
\end{equation}
We claim that only the lower branch can occur. Fix \(h\leq\ell_0\).
Since \(x_{h2^{-n}}\leq x_T<\infty\), while $r_+(h2^{-n})\longrightarrow\infty$, as $n\to\infty$,
we must have 
\begin{equation*}
	x_{h2^{-n}}
	\leq
	r_-(h2^{-n})
	\leq
	2A.
\end{equation*}
Moreover, if the lower-branch bound holds at scale \(h/2\), then
\eqref{est.doubling.bootstrap} gives
\begin{equation*}
	x_h
	\leq
	2x_{h/2}
	\leq
	4A
	<
	8A
	\leq
	r_+(h).
\end{equation*}
The alternative then again forces
\begin{equation*}
	x_h\leq r_-(h)\leq 2A.
\end{equation*}
Iterating this implication from the scale \(h2^{-n}\) up to \(h\), we obtain
\begin{equation*}
	x_h\leq 2A\lesssim K,
	\qquad
	0<h\leq\ell_0.
\end{equation*}
Consequently,
\begin{align}\label{est.yonsmall}
	\big(
		\|\delta Y\|_{\alpha;m,\infty;[s,t]}\vee1
	\big)
	\lesssim K
	\qquad\text{whenever}\qquad
	0\leq t-s\leq\ell\leq\ell_0=c_0K^{-2/\beta''},
\end{align} }
	%
			and plugging \eqref{est.yonsmall} in  \eqref{tmp.1149}, we obtain
			\begin{align}\label{est.Ryonsmall}
				\|\E_\bigcdot R^Y\|_{\alpha+\beta;\infty;[s,t]}\lesssim K^2
			\end{align}
whenever \( 0\le t-s\le\ell\le\ell_0 \).
\smallskip

			\textit{Step 2: extension over the whole interval $[0,T]$.}

			If $s\le t$ are fixed, then for any partition $\{\tau_i\}_{i=0}^N$ of $[s,t)$,
			we have by a telescopic argument (using $\delta R^{Y}_{s,u,t}=-\delta Z_{s,u}\delta X_{u,t}$) that
			\[
			R^{Y}_{s,t}=\sum_{i=0}^{N-1}\left (R^{Y}_{\tau_i,\tau_{i+1}} - \delta Z_{\tau_i,\tau_{i+1}}\delta X_{\tau_{i+1},t}\right ).
			\]
			Using triangle inequality and bounding the conditional expectations in an obvious way, we have
			\begin{align*}
				|\E_sR^{Y}_{s,t}|\le \sum_{i=0}^{N-1}\left(\|\E_{\tau_i}R^{Y}_{\tau_i,\tau_{i+1}}\|_\infty+\|\|\delta Z_{\tau_i,\tau_{i+1}}|\cff_{\tau_i}\|_m\|_\infty|X|_{\alpha;[0,T]}(t-s)^\alpha \right).
			\end{align*}
			From here, the estimate for $\|\delta Z\|_{\beta;m,\infty}$ and the estimates on small intervals \eqref{est.yonsmall} and \eqref{est.Ryonsmall} can be combined to obtain
			\begin{align*}
				\|\E_s R^Y_{s,t}\|_{\infty}&\lesssim \frac{K^2}{\ell_0}(t-s)^{\alpha+\beta}.
			\end{align*}
			This yields the estimate for $\|\E_\bigcdot R^Y\|_{\alpha+\beta;\infty;[0,T]}$ in \eqref{est.apri.m} because $\ell_0\sim K^{-2/\beta''}$.
			The estimate for $\|\delta Y\|_{\alpha;m,\infty;[0,T]}$ follows from  \eqref{est.yonsmall} by similar arguments, completing the proof. 
		\end{proof}

	\subsection{Existence and uniqueness} 
	\label{sub.strong_existence}
		In this section, we construct a solution to \eqref{eqn.srde} by a fixed-point argument.
		\begin{theorem}\label{thm.fixpoint}
			Let $m$ be in $[2,\infty)$ and $\X\in \CC^\alpha$ with \( \frac13< \alpha\le\frac12 \). Let	$b,\sigma$ be random bounded Lipschitz functions,
		 assume that $(f,f')$ belongs to  $\D^{2 \alpha}_XL_{m,\infty}\C^\gamma_b$ while $(Df,Df')$ belongs to  $\D^{\alpha,\alpha''}_XL_{m,\infty}\C^{\gamma-1}_b$.
			Assume moreover that $\gamma>\frac1 \alpha$ and $2 \alpha+\alpha''>1$. Then for every $\xi\in L_0(\cff_0;W)$, there exists a unique $L_{m,\infty}$-integrable  solution to \eqref{eqn.srde} starting from $\xi$ over any finite time interval.
		\end{theorem}

		\begin{remark} No integrability condition is required on $Y_0 = \xi$.
		Also, since $L_{m,\infty}$-integrability implies $L_{2,\infty}$-integrability, it is clear that uniqueness of $L_{m,\infty}$-solutions also holds within the wider class of $L_{2,\infty}$-solutions.
	\end{remark}
	\begin{corollary} \label[corollary]{cor.expInt}
		Let $Y$ be the solution of \cref{thm.fixpoint}. Then $Y$ satisfies the exponential estimate \eqref{est.JohnNirenberg} (with $\abx=W$).
	\end{corollary}
	\begin{proof}
		From \cref{rmk.JwhenBounded}, $\delta Y$ belongs to $C^\alpha L_{m,\infty}$. Being a solution, $Y$ is a.s. continuous and hence the result is a direct consequence of \cref{prop.JNineq}.
	\end{proof}

		Our method deviates from the familiar one for rough differential equations (e.g \cite{FH20}*{Ch.8}) in several ways. The highly non-trivial part is to identify a suitable metric on the space of stochastic controlled rough paths for which a fixed-point theorem can be applied. All estimates, e.g.\ those obtained in \cref{thm.Hroughint} for rough stochastic integrals, have already been prepared in this way.
		As already alluded in \cref{rem:loss}, unless $n=\infty$, a loss of integrability (from $(\gamma-1)n$ down to $n$) appears in the estimates of \cref{lem.composecvec}.
		For this reason, the invariance property of the fixed point map needs to be established on a bounded set of $\D_X^{2\alpha}L_{m,\infty}$.
		As is quickly realized, however, the corresponding distance is too strong to yield any contraction property, which leads us to a weaker metric.

	\subsubsection*{Proof of \cref{thm.fixpoint}}
	By replacing $\gamma$ by $\gamma \wedge 3$ if necessary, we can assume that $2 \le 1/\alpha < \gamma \le 3$. 
	\reply{Using assumptions on $\alpha,\alpha''$, we can choose and fix  $\beta\in(\frac 1 \gamma,\alpha)$, $\beta'\in(0,\beta)$ such that
 } \footnote{The need for this cascade of  exponents is seen later in the proof, cf. $\delta>0$.}
    \begin{align*}
        2 \beta+\beta'>1 \tand 2 \beta+\alpha''>1.
    \end{align*}
	We first construct a local solution, on $[0,T]$ for $T$ small. 
	It suffices to construct a process $(Y,f(Y))$ in $\D_X^{\beta,\beta'}L_{m,\infty}$ such that $Y$ is \( \P \)-a.s.\ continuous and the integral equation \eqref{eqn.int.srde} is satisfied.
       Indeed, if $(Y,f(Y))$ is such a process, then due to \cref{lem.composecvec}, applied with 
       \reply{$(f,f') \in \D^{2 \alpha}_XL_{m,\infty}\C^2_b$, see that}
       $(f(Y),Df(Y)f(Y)+f'(Y))$ belongs to 
       $\D_X^{\beta,\reply{\beta'}}L_{m,\infty}$, 
		The conditions on $\beta, \beta'$ ensure that 
		$\alpha+\beta>1/2$ and $\alpha+\beta+\reply{\beta'}>1$.
		Hence, by \cref{prop.davie}, $Y$ is an $L_{m,\infty}$-integrable solution to \eqref{eqn.srde}.
		
		Having $\beta,\beta'$ chosen as previously, 
we pick a constant
\[
M>\|b\|_\infty + \|\sigma\|_\infty + \|(f,f')\|_{2;\infty} + \bk{(f,f')}_{X;\beta,\beta';m,\infty}
\]
and define $\mathbf{B}_T$ as the collection of processes $(Y,Y')$ in $\D_X^{\beta,\beta'}L_{m,\infty}([0,T],\W;W)$ such that $Y_0=\xi, Y'_0=f_0(\xi)$,
		\begin{equation}\label{ball_BT}
			\|(Y,Y')\|_{X;\beta,\beta';m,\infty}\le M\,.
		\end{equation}
 			It is easy to see that for $T$ sufficiently small, the set $\mathbf{B}_T$ contains the process $t\mapsto (\xi+f_0(\xi)\delta X_{0,t},f_0(\xi))$, and hence, is non-empty.
			For each $(Y,Y')$ in $\mathbf{B}_T$, define
			\begin{equation}
			\label{fixed_point_map}
				\Phi(Y,Y')=\left(\xi+\int_0^\cdot b_r(Y_r)dr+\int_0^\cdot \sigma_r(Y_r)dB_r+\int_0^\cdot f(Y)d\X,f(Y) \right)\,.
			\end{equation}
					
			\reply{
			The first component above is a.s.\ continuous, hence progressively measurable. The second component is progressively measurable by the definition of the space $\D_X^{\beta,\beta'}L^{m,\infty}$ and the measurability assumptions on $f$, while \cref{lem.composecvec} shows that $(f(Y),Df(Y)Y'+f'(Y))$ is again a stochastic controlled rough path with the required estimates.}
			
			
			We will now show that if $T$ is sufficiently small, $\Phi$ has a unique fixed point in $\mathbf{B}_T$, which is a solution to \eqref{eqn.srde}.
		
			\paragraph{Invariance.}
			We show that there is a choice of $T^*=T^*(M,\rho_{\alpha}(\X))$ such that $\Phi$ maps $\mathbf{B}_{T}$ into itself, for any \( T\le T^* \).
			Let $(Y,Y')$ be an element in $\mathbf{B}_T$ and for simplicity put $(Z,Z')=(f(Y),Df(Y)f(Y)+f'(Y))$ (this belongs to $\D_X^{\beta,\beta'}L_{m,\infty}$ by \cref{lem.composecvec}).
			Applying the BDG inequality and standard bounds for Riemann integrals, we have for the drift and diffusion terms
				\[
			\Big\|(\int b(Y)dr,0)\Big\|_{X;\beta,\beta';m,\infty}
			=\Big\|\int b(Y)dr\Big\|_{\beta;m,\infty} + \Big\|\EE\int b(Y)dr \Big\|_{\beta+\beta';\infty}
			\lesssim \|b\|_\infty T^{1-\beta-\beta'},
				\]
			\[
			\Big\|(\int \sigma(Y)dB,0)\Big\|_{X;\beta,\beta';m,\infty} = \Big\|\int \sigma(Y)dr\Big\|_{\beta;m,\infty} 
			\lesssim \|\sigma\|_\infty T^{\frac12-\beta}\,.
			\]

With $\delta :=(\alpha-\beta)\wedge(\beta-\beta')>0$, apply \cref{cor.integral} and \cref{lem.composecvec} to obtain that \footnote{We note that the argument here is different than the standard one for determinstic RDEs. There (see e.g. \cite{FH20})  a factor $T$ to some positive power derives from $\alpha > \beta$, regularity of $\X$ vs  H\"older-scale  of the crp space in which the Picard iteration takes place; this argument includes $ | \delta Y |_{\beta} \lesssim T^{\delta}$, whenever is $Y$ is controlled by $X$, thanks to $2\beta$-regularity of $R^Y$. In case of scrp, the remainer is only conditionally of order $2 \beta$. A different argument is thus needed, which we base on the multiscale structure of $(\alpha, \beta, \beta'')$ for $(\int Z d X, Z, Z')$; the $T^\delta$ then derives from the positive differences of these exponents.}
\begin{align*}
	\left\llbracket \int Z d \X, Z \right\rrbracket_{X ; \beta,
   \beta', m, \infty} 
   &\lesssim \left\llbracket \int Z d \X, Z
   \right\rrbracket_{X ; \alpha, \beta, m, \infty} T^{\delta}
   \\&\lesssim (\|Z\|_{\infty;\infty}+ \|(Z,Z')\|_{X ; \beta, \beta' ; m, \infty}) T^{\delta}
   \\&\lesssim    (\|(f,f')\| _{2;\infty}+\bk{(f,f')}_{X;\beta,\beta';m,\infty} )(1+\|(Y,Y')\|_{X;\beta,\beta';m,\infty}^2)T^{\delta} .
\end{align*}
Summing up the above contributions, we arrive at the bound
\[
\|\Phi(Y,Y')\|_{X;\beta,\beta',m,\infty}\le 
\|f\|_{\infty} + C (1+M^3)T^{\delta'}
\]
for a constant \( C=C(T,\rho_{\alpha}(\X) ) \) which is non-decreasing in \( T \), and where \(\delta' :=\min\{\delta,1-\beta-\beta',\frac12-\beta\} >0\).
The above right hand side is indeed bounded above by \( M \) provided that 
\( T\le T^*:=(\frac{M-\|f\|_\infty }{C(1+M^3)})^{\frac1{\delta'}} \). This proves the desired property.

			\paragraph{Contraction.}
We suppose that \( M,T \) are chosen as in the previous step.
Taking \( T \) smaller if necessary, we now show that $\Phi$ is a contraction on $\mathbf{B}_T$, but for the associated \( L_{m,m} \)-metric $\bk{-;-}_{X;\beta,\beta';m}$ defined in \eqref{def.scrp.bracket} (as opposed to \( L_{m,\infty}\) as in the above proof of invariance).
Because the starting position is fixed at \((\xi,f_0(\xi))\), we have
\begin{align}\label{tmp.equi.metric}
\bk{(Y,Y');(\bar Y,\bar Y')}_{X;\beta,\beta';m}\asymp \norm{Y-\bar Y}_{\infty;m} + \|Y'-\bar Y'\|_{\infty;m} +\bk{(Y,Y');(\bar Y,\bar Y')}_{X;\beta,\beta';m} 
\end{align}
for all \( (Y,Y') ,(\bar Y,\bar Y') \in \mathbf B_T \), and hence, $\bk{-;-}_{X;\beta,\beta';m}$ indeed forms a distance on \( \mathbf B_T \).
Moreover, in view of \cref{prop.C2Lmn}, we see that the resulting space is closed and complete. 

Now, in keeping with the previous notations, we let \( (\bar Z,\bar Z')=(f(\bar Y),Df(\bar Y)\bar Y'+f'(\bar Y))\).
By  \eqref{ineq.super_additive}, we have 
\begin{multline}
\label{tmp.57}
\bk{\Phi(Y,Y');\Phi(\bar Y,\bar Y')}_{X;\beta,\beta';m}\le
 \bklr{\int_0^\cdot b_r(Y_r)dr,0;\int_0^\cdot b_r(\bar Y_r),0}_{X;\beta,\beta';m}
\\
+ \bklr{\int_0^\cdot \sigma_r(Y_r)dB,0;\int_0^\cdot \sigma_r(\bar Y_r)dB,0}_{X;\beta,\beta';m}
+\bklr{\int_0^\cdot  Zd\X,Z;\int_0^\cdot \bar Zd\X,\bar Z}_{X;\beta,\beta';m}
\end{multline}
and we can estimate each term separately.

For the last term, we estimate (with \(   \delta=(\alpha-\beta)\wedge(\beta-\beta') \) is as before)
\[
\bklr{\int_0^\cdot  Zd\X,Z;\int_0^\cdot \bar Zd\X,\bar Z}_{X;\beta,\beta';m}
\lesssim\bklr{\int_0^\cdot  Zd\X,Z;\int_0^\cdot \bar Zd\X,\bar Z}_{X;\alpha,\beta;m}T^{\delta}.
\]
The term $\bk{-;-}_{X;\alpha,\beta;m}$ is estimated via \cref{cor.integral} (with \( n:=m \) therein) and \cref{prop.compose_stab} (with $\alpha'=\alpha$, $\kappa=\beta$, $\kappa'=\beta'$), yielding 
\begin{equation*}
\begin{aligned}
	\bklr{\int_0^\cdot  Zd\X,Z;\int_0^\cdot \bar Zd\X,\bar Z}_{X;\beta,\beta';m}
&\lesssim \Big(\|Z-\bar Z\|_{\infty;m} +\|Z,Z';\bar Z,\bar Z'\|_{X;\beta,\beta';m}\Big)T^{\delta}
\\&
\lesssim \|Z,Z';\bar Z,\bar Z'\|_{X;\beta,\beta';m}T^{\delta}
\\&
\lesssim
\|Y,Y';\bar Y,\bar Y'\|_{X;\beta,\beta';m}T^{\delta}\,.
\end{aligned}
\end{equation*}
To go from first to second line, we have used the fact that \( Z_0=\bar Z_0 \) and thus \( \|Z-\bar Z\|_{\infty;m} \lesssim \|\delta Z-\delta \bar Z\|_{\beta;m} \lesssim \|Z,Z';\bar Z,\bar Z'\|_{X;\beta,\beta';m}\), by definition.
Next, the drift and diffusion terms are estimated as in the proof of invariance, noting this time that the right hand sides are proportional to the corresponding Lipschitz norms (as introduced in \cref{def.randomcoef}). 
Finally, inserting these contributions in \eqref{tmp.57}, using \eqref{tmp.equi.metric}, we obtain that
			\begin{equation}
			\label{previous_1}
				\bk{\Phi(Y,Y');\Phi(\bar Y,\bar Y')}_{X;\beta,\beta';m}
				\le
				C T^{\delta'}
				\bk{Y,Y';\bar Y,\bar Y'}_{X;\beta,\beta';m}.
			\end{equation}
			where 
			the constant \( C \) depends only on
			\( \beta,\gamma, \reply{M, T}, \|(Df,Df')\|_{\gamma;\infty} \), \( \bk{(Df,Df')}_{X;\alpha,\alpha'';m,\infty} \), $\rho_\alpha(\X)$, \( \|b\|_\lip \) and \( \|\sigma\|_\lip \), \reply{with  dependence on $T$ non-decreasing}. This proves that $\Phi$ is indeed a contraction if $T$ is sufficiently small.

		 \paragraph{Concluding the proof.}
				Picard's fixed point theorem asserts that we can find a unique process $(Y,Y')$ in $\mathbf{B}_T$ such that $\Phi(Y,Y')=(Y,Y')$. In particular, $(Y,f(Y))$ is a stochastic controlled rough path in $\D^{\beta,\beta'}_XL_{m,\infty}$ which satisfies equation \eqref{eqn.int.srde}.
		Because the smallness of $T$ only depends on $\rho_\alpha({\X})$ and the norms of the coefficients $b,\sigma,f,f'$ but not on $\xi$, the previous procedure can be iterated to construct a unique solution in $\D_X^{\beta,\beta'}L_{m,\infty}$ over $[0,T_0]$, for any \( T_0>0 \).
		To show uniqueness, we observe from \cref{rmk.JwhenBounded} that if $\bar Y$ is a $L_{m,\infty}$-solution, then $(\bar Y,f(\bar Y))$ belongs to $\D_X^{2 \alpha}L_{m,\infty}$. It follows that $\bar Y$ belongs to $\mathbf{B}_T$ for $T$ sufficiently small. Since $\Phi$ is a contraction on $\mathbf{B}_T$, this shows that the $L_{m,\infty}$-solution is unique on small time intervals, which implies uniqueness on any finite time intervals.
		This proves the theorem.
\hfill\qed


\subsection{Continuous dependence on data}\label{sub.continuity_soln_map}
 We now establish the continuity of the solution to \eqref{eqn.srde} with respect to its full inputs data.
 At first reading, the reader may assume $\alpha = \beta = \beta' $ and $\gamma =3$, with possible focus on time-independent $f$ (which renders harmless all $\beta$ exponents).  In general, these exponents are needed to allow finer spatial regularity assumption on the vector fields, in interplay with their temporal regularity (and that of $\X$).
\begin{theorem}
		\label{thm.stability_precise}
		Let $\xi,\bar \xi$ be in $L_0(\cff_0)$; $\X,\bar \X$ be in $\CC^\alpha$, $\alpha\in(\frac13,\frac12)$; \( \sigma,\bar\sigma,b,\bar b \) be random bounded continuous functions; fix \( m\ge2 \), and parameters \( \gamma\in (2,3] ,\beta\in (0,\alpha] \) such that \(\alpha + (\gamma-1)\beta>1\).
		Consider
		$(f,f')\in \D^{2 \beta}_XL_{m,\infty}\C^\gamma_b$ such that $(Df,Df')$ belongs to $\D^{\beta,\beta'}_XL_{m,\infty}\C^{\gamma-1}_b$
		where $\beta'>0$ is taken so that
		\[
		1-\alpha-\beta <\beta'\le 1,
		\]
		and fix another stochastic controlled vector field $(\bar f,\bar f')\in \D^{\beta,\beta'}_{\bar X}L_{m,\infty}\C^{\gamma-1}_b$.
	 	Let $Y$ be an $L_{m,\infty}$-integrable solution to \eqref{eqn.srde} starting from $\xi$, and similarly denote by
	 	$\bar Y$ an $L_{m,\infty}$-integrable solution to \eqref{eqn.srde} starting from $\bar \xi$ with associated coefficients $(\bar \sigma,\bar f,\bar f',\bar b,\bar \X)$.
	 	Let $M$ be a constant such that 
	 	\begin{multline*}
		 	\rho_\alpha({\X})+\|b\|_{\lip}+\|\sigma\|_{\lip}+\|b\|_{\infty}+\|\sigma\|_{\infty}
		 	\\+\|(f,f')\|_{\gamma;\infty}+\bk{(f,f')}_{X;2 \beta;m,\infty}+\bk{(Df,Df')}_{X;\beta,\beta';m,\infty}\le M
	 	\end{multline*}
		\reply{and
		$$
			\rho_\alpha({\bar\X})+\|\bar b\|_\infty+\|\bar \sigma\|_{\infty}+\|(\bar f,\bar f')\|_{\gamma-1;\infty}+\bk{(\bar f,\bar f')}_{\bar X;\beta,\beta';m,\infty}\le M.
		$$
		}
		We denote $\beta''=\min((\gamma-2)\beta,\beta')$ and
		\begin{multline*}
		 \theta=
		\rho_{\alpha,\beta}(\X,\bar\X) 
		+\sup_{t\in[0,T]}\|\sup_{x\in W}|b_t(x)-\bar b_t(x)|\|_m
		 \\
		+\sup_{t\in[0,T]}\|\sup_{x\in W}|\sigma_t(x)-\bar \sigma_t(x)|\|_m
		+\|(f-\bar f,f'-\bar f')\|_{\gamma-1;m}
		 +\bk{f,f';\bar f,\bar f'}_{X,\bar X;\beta,\beta'';m},
		\end{multline*}
		where  the notations are defined in \eqref{def.rho_metric}, \eqref{def.norms_scvf} and \eqref{def.bk_metric}.

		Then,
		 we have the estimate\footnote{We note that if $m \alpha>1$, then by Kolmogorov continuity theorem
		\[
			\|\sup_{t\in[0,T]}|\delta\tilde Y_{0,t}|\|_m\lesssim \|\delta\tilde Y\|_{\alpha;m}.
		\]}
		\begin{equation}\label{est.stability_precise}
			\|\sup_{t\in[0,T]}  |\delta Y_{0,t}-\delta \bar Y_{0,t}|\|_m
			+\|Y,Y';\bar Y,\bar Y'\|_{X,\bar X;\alpha,\beta;m;[0,T]}
			\lesssim\||\xi-\bar \xi|\wedge1\|_m+\theta,
		\end{equation}
		where the implied constant depends on $\alpha,\beta,\beta',\gamma,T$ and $M$.
	\end{theorem}

\begin{proof}
	Without loss of generality, we can and will assume that $1-\alpha-\beta <\beta'\le (\gamma-2)\beta$, in particular $\beta''=\beta'$.  We introduce the stochastic controlled rough path
\[
Z=Y'=f(Y), \quad
Z'=Df(Y)f(Y)+f'(Y)
\]
and similar for $(\bar Z,\bar Z')$.
Thanks to \cref{rmk.JwhenBounded},
we have that $(Y,f(Y))$ and $(\bar Y,\bar f(\bar Y))$ both belong to $\D^{\alpha,\beta}_{X}L_{m,\infty}\subset \D^{2 \beta}_{X}L_{m,\infty}$ and $\D^{\alpha,\beta}_{\bar X}L_{m,\infty}\subset \D^{\beta,\beta'}_{\bar X}L_{m,\infty}$ respectively.
		Consequently, \cref{lem.composecvec} implies that $(Z,Z')$ and $(\bar Z,\bar Z')$ belong to $ \D^{2 \beta}_XL_{m,\infty}$ and $ \D^{\beta,\beta'}_{\bar X}L_{m,\infty}$ respectively.
		\reply{
			Before turning to the estimates, we record that \cref{prop.apri} provides the a priori bounds needed to apply \cref{prop.compose_stab}. In particular, we have 
\begin{align*}
\|(Y,Y')\|_{X;\alpha,\beta;m,\infty}+\|(\bar Y,\bar Y')\|_{\bar X;\alpha,\beta;m,\infty} \lesssim 1.
\end{align*}
		}
	Let $I$ be a sub-interval of $[0,T]$ and put $\Gamma_I=\sup_{t\in I}\|| Y_t-\bar Y_t|\wedge1\|_m$ .

	\textit{Step 1.} We put $R=\delta Y-Z\delta X$, $\bar R=\delta \bar Y-\bar Z\delta \bar X$ and $\tilde Y=Y-\bar Y$, $\tilde Z=Z-\bar Z$, $\tilde R=R-\bar R$.
	We show that if $|I|$ is small enough (depending on $M,\alpha,\beta,\beta',\gamma,T$), then
	\begin{align}
		\|Y,Y';\bar Y,\bar Y'\|_{X,\bar X;\alpha,\beta;m;I}
		\lesssim \|\tilde R\|_{1/2;m;I}|I|^{1/2-\alpha} + \Gamma_I+\theta.
		\label{est.YYbar_on_I}	
	\end{align}

	Indeed, in view of the identity $\delta \tilde Y=\tilde R+Z\delta X-\bar Z\delta \bar X $, we have
	\begin{align}
		\|\delta\tilde Y\|_{\alpha;m;I}
		&\lesssim \|\tilde R\|_{1/2;m;I}|I|^{1/2-\alpha}+\|Z-\bar Z\|_{\infty;m;I}+\theta
		\nonumber\\&\lesssim \|\tilde R\|_{1/2;m;I}|I|^{1/2-\alpha} 
		+\Gamma_I+\theta .\label{est.dY_on_I}
	\end{align}
	In the above, to estimate $\|Z-\bar Z\|_{\infty;m;I}$, we have used the following inequality which is valid for any random bounded Lipschitz function $h$:
	\begin{align}
		\|h(Y_t)-h(\bar Y_t)\|_m\lesssim( \|h\|_\infty+\|h\|_\lip ) \| |Y_t-\bar Y_t|\wedge1\|_m.
		\label{est.hYhYbar}
	\end{align}
	Combine with \eqref{est.cdz} (with 
	$(\kappa,\kappa')=(\beta,\beta')$), we get
	\begin{align}
		\| \tilde Y'\|_{\beta;m;I}=\| \tilde Z\|_{\beta;m;I}
		\lesssim \|\tilde R\|_{1/2;m;I}|I|^{1/2-\beta}  +\Gamma_I+\theta.
		\label{est.dZ_on_I}
	\end{align}

	 Put 
	\begin{align*}
		J:=\int f(Y)d\X-Z\delta X-Z'\XX
		=R^Y - Z' \XX-\int b(Y)dr -\int \sigma(Y)dB
	\end{align*} and similarly for $\bar J$.
	\reply{Applying \eqref{est.hYhYbar} with $h=\{Dff,f'\}$, we have
	\begin{align}\label{tmp.newZXX}
		\|Z'\XX- \bar Z'\bar \XX \|_{\alpha+\beta;m;I}\lesssim\Gamma_I+\theta.
	\end{align}
	From the identity  
	\begin{align}
		\tilde R=J-\bar J + Z'\XX- \bar Z'\bar \XX +\int (b(Y)-\bar b(\bar Y) ) dr+\int ( \sigma(Y)-\bar \sigma(\bar Y) ) dB ,
		\label{identity.dR}
	\end{align}
	we have by standard estimates (for moment norms), \eqref{est.hYhYbar}, applied with $h=b$ and \eqref{tmp.newZXX}, 
	also noting that $\EE_s \int_s^t(...) d B_t$ vanishes here,
	}
	\begin{align*}
		\|\EE\tilde R\|_{\alpha+\beta;m;I}\lesssim \|\EE(J-\bar J)\|_{\alpha+\beta+\beta';m;I}|I|^{\beta'} 
		+\Gamma_I +\theta.
	\end{align*}
	To estimate $\EE(J-\bar J)$, we apply \eqref{est.RSIE} (with $m=n$) and \eqref{est.zz} to see that 
	\begin{align*}
		\|\EE(J-\bar J)\|_{\alpha+\beta+\beta';m;I} 
		&\lesssim  \|Z,Z';\bar Z,\bar Z'\|_{X,\bar X;\beta,\beta';m;I}+
		 \rho_{\alpha,\beta}(\X,\bar\X)  
		\\&\lesssim  \|Y,Y';\bar Y,\bar Y'\|_{X,\bar X;\beta,\beta';m;I}+ \Gamma_I +\theta.
	\end{align*}
	Hence, we obtain
	\begin{align}\label{est.dR_on_I}
		\|\EE\tilde R\|_{\alpha+\beta;m;I}\lesssim \|Y,Y';\bar Y,\bar Y'\|_{X,\bar X;\beta,\beta';m;I}|I|^{\beta'} +\Gamma_I +\theta.
	\end{align}
	Summing up \eqref{est.dY_on_I}, \eqref{est.dZ_on_I} and \eqref{est.dR_on_I}, we have
	\begin{align*}
		\|Y,Y';\bar Y,\bar Y'\|_{X,\bar X;\alpha,\beta;m;I}
		\lesssim \|\tilde R\|_{1/2;m;I}|I|^{1/2-\alpha} +\|Y,Y';\bar Y,\bar Y'\|_{X,\bar X;\beta,\beta';m;I}|I|^{\beta'} +\Gamma_I +\theta.
	\end{align*}
	Noting that $\|Y,Y';\bar Y,\bar Y'\|_{X,\bar X;\beta,\beta';m;I}\le \|Y,Y';\bar Y,\bar Y'\|_{X,\bar X;\alpha,\beta;m;I}$, we obtain \eqref{est.YYbar_on_I} when $|I|$ is small enough.

	\textit{Step 2.} We show that if $|I|$ is small enough, then 
	\begin{align}
		\|\tilde R\|_{1/2;m;I}
		\lesssim  \Gamma_I+\theta.
		\label{est.R_on_I}
	\end{align}

	Applying \eqref{est.RSIE} (again with $m=n$) and \eqref{est.zz}, we have
	\begin{align*}
		\|J-\bar J\|_{\alpha+\beta;m;I}
		&\lesssim  \|Z,Z';\bar Z,\bar Z'\|_{X,\bar X;\beta,\beta';m;I} + \rho_{\alpha,\beta}(\X,\bar\X) 
		\\&\lesssim  \|Y,Y';\bar Y,\bar Y'\|_{X,\bar X;\beta,\beta';m;I} +\Gamma_I+\theta.
	\end{align*}
	Then we use \eqref{est.YYbar_on_I} to get that
	\begin{align}
		\|J-\bar J\|_{\alpha+\beta;m;I}
		\lesssim \|\tilde R\|_{1/2;m;I}|I|^{1/2-\alpha}+\Gamma_I+\theta.
		\label{est.JJbar_on_I}
	\end{align}
	On the other hand, 
	from \eqref{identity.dR}, we apply \eqref{est.hYhYbar} to see that (noting $\alpha+\beta>1/2$ from our assumptions)
	\begin{align*}
		\|\tilde R\|_{1/2;m;I}
		&\lesssim \|J-\bar J\|_{\alpha+\beta;m;I} + \|Z'-\bar Z'\|_{\infty;m}+ \Gamma_I + \theta
		\\&\lesssim \|J-\bar J\|_{\alpha+\beta;m;I} +\Gamma_I + \theta.
	\end{align*}
	Combining with \eqref{est.JJbar_on_I}, we obtain that 
	\begin{align*}
		\|\tilde R\|_{1/2;m;I}
		\lesssim \|\tilde R\|_{1/2;m;I}|I|^{1/2-\alpha}  + \Gamma_I+\theta.
	\end{align*}
	This yields \eqref{est.R_on_I} provided that $|I|$ is small enough.

	\textit{Step 3.} By combining \eqref{est.YYbar_on_I} and \eqref{est.R_on_I}, we have
	\begin{align*}
		\|Y,Y';\bar Y,\bar Y'\|_{X,\bar X;\alpha,\beta;m;I}
		\lesssim  \Gamma_I+\theta.
	\end{align*}
	We note that 
	\begin{align*}
		 \Gamma_I\lesssim \||\tilde Y_o|\wedge1\|_{m} + \|\delta \tilde Y\|_{\alpha;m;I}|I|^{\alpha}
		 \lesssim  \||\tilde Y_o|\wedge1\|_{m} + \|Y,Y';\bar Y,\bar Y'\|_{X,\bar X;\alpha,\beta;m;I}|I|^{\alpha}.
	\end{align*}
	Hence, if $|I|$ small enough, then 
	\begin{align}
		\|Y,Y';\bar Y,\bar Y'\|_{X,\bar X;\alpha,\beta;m;I}
		\lesssim  \||\tilde Y_o|\wedge1\|_{m} + \theta.
		\label{est.final_on_I}
	\end{align}

	\textit{Step 4.} We put $G_t=\||\tilde Y_t|\wedge1\|_m$. We deduce from the previous step that there is a constant $\ell>0$ such that whenever \( |t-s|\le \ell \), we have
\begin{align}
	\|Y,Y';\bar Y,\bar Y'\|_{X,\bar X;\alpha,\beta;m;[s,t]}\lesssim G_s+\theta. 
\label{key_est_stab}
\end{align}
On the other hand, it is evident that
\begin{align*}
	G_t-G_s\le   \|\delta\tilde  Y\|_{\alpha;m;[s,t]} (t-s)^\alpha \le \|Y,Y';\bar Y,\bar Y'\|_{X,\bar X;\alpha,\beta;m;[s,t]}|t-s|^\alpha. 
\end{align*}
It follows from the above estimates that   
\begin{align*}
	G_t-G_s\lesssim G_s|t-s|^\alpha + \theta 
\end{align*}
whenever $0\le t-s\le \ell$.
A standard argument implies that $G_T\lesssim G_0+ \theta$.

We plug this into \eqref{key_est_stab} to obtain $\|Y,Y';\bar Y,\bar Y'\|_{X,\bar X;\alpha,\beta;m;[s,t]}\lesssim G_0+\theta$ whenever \(0\le t-s\le \ell \).
To obtain the corresponding estimate  on \([0,T] \), it suffices to repeat the second step in the proof of \cref{prop.apri} (details are omitted). 

From \eqref{est.zz}, we also have 
\begin{align*}
	\|Z,Z';\bar Z,\bar Z'\|_{X,\bar X;\beta,\beta';m;[0,T]}
	\lesssim \|Y,Y';\bar Y,\bar Y'\|_{X,\bar X;\alpha,\beta;m;[0,T]} + \theta\lesssim G_0+\theta.
\end{align*}
Finally, to estimate the supremum of the increments of \( \delta Y_{0,\cdot} - \delta \bar Y_{0,\cdot} \), we start from  the equation for $Y-\bar Y$, then apply \eqref{est.integral} and the previously obtained estimates for $\|Y,Y';\bar Y,\bar Y'\|_{X,\bar X;\alpha,\beta;m;[0,T]}$ and $\|Z,Z';\bar Z,\bar Z'\|_{X,\bar X;\beta,\beta';m;[0,T]}$  to the rough stochastic integrals, apply BDG inequality to the stochastic integrals and standard estimates to the drift terms. Details are omitted.
This completes the proof.
\end{proof}

\subsection{Uniqueness at criticality} 
	\label{sub:uniqueness_when_}
	The following result can be seen as extension of \cite[Theorem 3.6]{MR2387018} to the setting of RSDEs. Since deterministic RDEs with time-independent vector fields, as considered in \cite{MR2387018}, are also RSDEs, the counter example(s) therein show that no improvement is possible.
		\begin{theorem}[Uniqueness]\label{thm.uniq.notloc}
			Suppose that $b,\sigma$ are random bounded Lipschitz functions, \( \frac13< \alpha < \frac12 \), $\gamma = 1/\alpha$,			
			 $(f,f')$ belongs to $ \D^{2 \alpha}_XL_{2,\infty}\C^{\gamma}_b$  and $(Df,Df')$ belongs to $ \D^{\alpha,(\gamma-2)\alpha}_XL_{2,\infty}\C^{ \gamma-1}_b$, where $\gamma=1/\alpha$.
			Let $\xi$ be in $L_0(\cff_0;W)$. Let $Y$ be an $L_{2,\infty}$-integrable solution on $[0,T]$ starting from $\xi$. Then $Y$ is unique in following the sense. If  $\bar Y$ is another $L_{2,\infty}$-integrable solution on $[0, T]$ starting from $\xi$ defined on the same filtered probability space $(\Omega,\cgg,\{\cff_t\},\mathbb P)$, then $Y$ and $\bar Y$ are indistinguishable.
		\end{theorem}

	\subsubsection{Davie--Gr\"onwall-type lemma} 
	\label{sub.davie_gronwall_type_lemmas}
	We record an auxiliary result allowing to compare integral remainders as in \cref{thm.Hroughint} when the value of the exponents therein is critical. It is inspired by \cite[Thm 3.6]{MR2387018} and \cite[Thm 2.1]{le2018stochastic}.
	In essence, Davie's inductive argument is replaced by a decomposition  which allows exploiting BDG inequality.
	\begin{lemma}\label[lemma]{lem.Davie_iteration} Let $T,\alpha,\eta,\varepsilon$ be positive numbers and $C,G,\Gamma_1,\Gamma_2$ be nonnegative numbers such that $\eta\in(\frac12,1]$ and $\alpha+\eta>1$.
    Assume that $J$ is an $L_m$-integrable process indexed by $\Delta$ such that
		\begin{gather}
			\|J_{s,t}\|_m\le C|t-s|^{\eta}
			\,,\quad\|\E_sJ_{s,t}\|_m\le C|t-s|^{1+\varepsilon},\label{con.sll1}
			\\\|\delta J_{s,u,t}\|_m\le G\Big(\sup_{[r,v]\subset[s,t]}\|J_{r,v}\|_m\Big)|t-s|^{\alpha}+ \Gamma_2|t-s|^\eta\label{con.dJ}
			\\\shortintertext{and}\label{con.EdJ}
		 \|\E_s \delta J_{s,u,t}\|_m\le G\Big(\sup_{[r,v]\subset[s,t]}\|J_{r,v}\|_m\Big)|t-s|^\alpha+\Gamma_1|t-s|
		\end{gather}
		for every $(s,u,t)$ in $\Deltatwo$.
		Then, there exist positive constants $c=c(\varepsilon,\eta,\alpha,m)$ and $\ell=\ell(\varepsilon,\eta,\alpha,m,G)$ such that
		for every $(s,t)\in \Delta$ with \( |t-s|\le \ell \)
		\begin{align}\label{est.SDavie}
			\|J_{s,t}\|_m\le c
		\Gamma_1\left(1+|\log\frac{\Gamma_1}C|+|\log{(t-s)}|\right)(t-s)+\Gamma_2(t-s)^\eta.
		\end{align}
	\end{lemma}
	\begin{proof}
All the implicit constants in our estimates below depend only on $\varepsilon, \eta,\alpha, m$.

For each integer $k\ge0$, let $\cpp_k$ denote the dyadic partition of $[s,t]$ of mesh size $2^{-k}|t-s|$.
		By triangle inequality, we have
	    \begin{align*}
	      \|J_{s,t}\|_m
	      \le
	      \big\|\sum\nolimits_{[u,v]\in\cpp_j}J_{u,v} \big\|_m+\big\|J_{s,t}-\sum\nolimits_{[u,v]\in\cpp_j}J_{u,v}\big\|_m.
	    \end{align*}
	    We estimate the first term using BDG inequality and condition \eqref{con.sll1}. This yields
	     \begin{align*}
	      \|\sum\nolimits_{[u,v]\in\cpp_j}J_{u,v}\|_m
	      &\lesssim\sum\nolimits_{[u,v]\in\cpp_j}\|\E_uJ_{u,v}\|_m+\left(\sum\nolimits_{[u,v]\in\cpp_j}\|J_{u,v}\|_m^2\right)^{1/2}
	     \\& \lesssim C  2^{- j \varepsilon}(t-s)^{1+\varepsilon}+C2^{-j(\eta-\frac12)}(t-s)^\eta.
	    \end{align*}
	    For the second term, we derive from \cite[id. (3.17)]{LeSSL2} (see also \cite[id. (2.47)]{le2018stochastic})  and BDG inequality  that for \( j\ge1 \)
	    \begin{align*}
	    	&\|J_{s,t}-\sum\nolimits_{[u,v]\in\cpp_j}J_{u,v}\|_m
	    	=\|\sum\nolimits_{k=0}^{j-1}\sum\nolimits_{[u,v]\in\cpp_k}\delta J_{u,(u+v)/2,v}\|_m.
	    	\\&\lesssim\sum\nolimits_{k=0}^{j-1}\sum\nolimits_{[u,v]\in\cpp_k}\|\E_u\delta J_{u,(u+v)/2,v}\|_m +\sum\nolimits_{k=0}^{j-1}\left(\sum\nolimits_{[u,v]\in\cpp_k}\|\delta J_{u,(u+v)/2,v}\|_m^2 \right)^{\frac12}.
	    \end{align*}
	    Applying \eqref{con.EdJ} and \eqref{con.dJ}, we have
	    \begin{align*}
	       \sum\nolimits_{[u,v]\in\cpp_k}\|\E_u\delta J_{u,(u+v)/2,v}\|_m
	     	 &\lesssim G 2^{-k(\alpha+\eta-1)}\|J\|_{\eta;m}(t-s)^{\alpha+\eta}+\Gamma_1(t-s)
	    \end{align*}
	    and
	    \begin{gather*}
	    	\left(\sum\nolimits_{[u,v]\in\cpp_k}\|\delta J_{u,(u+v)/2,v}\|_m^2\right)^{\frac12}
	    	\lesssim G 2^{-k(\alpha+\eta-\frac12)}\|J\|_{\eta;m}(t-s)^{\alpha+\eta}
	    	+ 2^{-k(\eta-\frac12)}\Gamma_2(t-s)^\eta.
	    \end{gather*}
	    Summing in $k$, noting that $\alpha+\eta>1$ and $\eta>\frac12$, we have
	    \begin{equation*}
	      \big\|J_{s,t}-\sum\nolimits_{[u,v]\in\cpp_j}J_{u,v}\big\|_m\lesssim G  \|J\|_{\eta;m}(t-s)^{\alpha+\eta}  + j\Gamma_1(t-s)+\Gamma_2(t-s)^\eta.
	    \end{equation*}

	    From here,  we obtain that for every integer $j\geq1$ and every  $(s,t)\in \Delta$
	    \begin{equation}
		\label{tmp.Jst}
		    \|J\|_{\eta;m}\lesssim C2^{-j\varepsilon}(t-s)^{1+\varepsilon-\eta} +C2^{-j(\eta-\frac12)} +G\|J\|_{\eta;m}(t-s)^\alpha+j\Gamma_1(t-s)^{1-\eta}+\Gamma_2.
	    \end{equation}
	    For \( |t-s|\le \ell\) with sufficiently small $\ell$,  this gives
	    \begin{equation}\label{tmp.Lst}
		    \|J\|_{\eta;m}\lesssim C
			2^{-j \left(\varepsilon\wedge(\eta-\frac12)\right)}
			+j\Gamma_1(t-s)^{1- \eta}+\Gamma_2.
	    \end{equation}

To conclude, we consider two cases.
	If \(j_0:=\left(\varepsilon\wedge(\eta-\frac12)\right)^{-1}\log_2\frac{C}{\Gamma_1(t-s)^{1-\eta}}<2\)
	we have $C\lesssim \Gamma_1(t-s)^{1- \eta}$ and choose $j=1$. Then from \eqref{tmp.Lst}, we have
	\begin{align*}
		 \|J\|_{\eta;m}\lesssim \Gamma_1(t-s)^{1-\eta}+\Gamma_2.
	 \end{align*}	 		 
	which yields \eqref{est.SDavie}.
	If $j_0\ge2$, we  choose $j=\lfloor j_0\rfloor$ so that \(C 2^{-j\left(\varepsilon\wedge(\eta-\frac12)\right)}\le2 \Gamma_1(t-s)^{1- \eta} \) and thus from \eqref{tmp.Lst},
    \begin{align*}
	    \|J\|_{\eta;m}
	    &\lesssim\Gamma_1(1+j_0)(t-s)^{1- \eta}+\Gamma_2
	    \\&\lesssim \Gamma_1\left (1+|\log\frac{\Gamma_1}C|+|\log{(t-s)}|\right )(t-s)^{1- \eta}+\Gamma_2,
    \end{align*}
	    which implies \eqref{est.SDavie}.
This finishes the proof.
	\end{proof}

		\subsubsection{Proof of \cref{thm.uniq.notloc}}
			We hinge on \cref{lem.Davie_iteration}.

			From \cref{rmk.JwhenBounded} and boundedness of the coefficients,
			we see that $(Y,f(Y))$ belongs to $ \D^{2 \alpha}_{X}L_{2,\infty}$. Similarly, $(\bar Y,f(\bar Y))$ belongs to $ \D^{2 \alpha}_{X}L_{2,\infty}$.

			We denote $\tilde Y=Y- \bar Y$ and
			 \[ \tilde Z:=f(Y)- f(\bar Y),\quad
					 \tilde Z'=Df(Y)f(Y)-Df(\bar Y)f(\bar Y)+f'(Y)-f'(\bar Y)
			 \]
			 and $R^{\tilde Z}=\delta \tilde Z-\tilde Z' \delta X$.
			For each $s\le t$, we further denote
			\[
				A_{s,t}=\tilde Z_s \delta X_{s,t}+\tilde Z'_s\XX_{s,t}
			\]
			and
			\begin{align}
				J_{s,t}
				&=\delta\tilde Y_{s,t}-\int_s^t[b_r(Y_r)-b_r(\bar Y_r)]dr-\int_s^t[\sigma_r(Y_r)-\sigma_r(\bar Y_r)]dB_r-A_{s,t}\label{id1}
				\\&=R^{\tilde Y}_{s,t}-\int_s^t[b_r(Y_r)-b_r(\bar Y_r)]dr-\int_s^t[\sigma_r(Y_r)-\sigma_r(\bar Y_r)]dB_r
				\label{id2}
				 -\tilde Z'_s\XX_{s,t}.
			\end{align}
			We now verify that $J$ satisfies the hypotheses of \cref{lem.Davie_iteration} with $m=2$ and every fixed but arbitrary $T>0$.
			First, it follows from \cref{lem.composecvec}  that $(Z,Z^{\prime}),(\bar Z,\bar Z^{\prime})$ belong to $ \D^{2 \alpha}_{X}L_{2,\infty}$.
			Hence, the inequalities in \eqref{est.defJex} hold with $\bar \alpha=\bar \alpha'=\alpha$, showing that \eqref{con.sll1} holds with $\eta=2 \alpha$.
			%
			%
			%
			
			\reply{ 
Define $\Gamma_{s,t}=\sup_{r\in[s,t]}\|\tilde Y_r\|_2$. From \eqref{id1}, the Lipschitz continuity of $b,\sigma,f,Df,f'$, the BDG inequality, and the bounds
\[
	\|\tilde Z_s\|_2+\|\tilde Z'_s\|_2\lesssim \|\tilde Y_s\|_2\le \Gamma_{s,t},
\]
we first obtain that
\[
	\|\delta\tilde Y_{s,t}\|_2
	\lesssim \|J_{s,t}\|_2+\Gamma_{s,t}(t-s)^\alpha .
\]
Using this estimate in the elementary composition estimates of Steps~1--3 in the proof of \cref{prop.compose_stab}, with $\beta=\alpha$ and $\beta'=(\gamma-2)\alpha$, gives
\begin{align}\label{est.tiY}
	\|\delta\tilde Y_{s,t}\|_2
	+\|\delta\tilde Z_{s,t}\|_2
	+\|R^{\tilde Z}_{s,t}\|_2
	&\lesssim
	\|J_{s,t}\|_2+\Gamma_{s,t}(t-s)^\alpha,
\end{align}
\begin{align}
	\|\delta\tilde Z'_{s,t}\|_2
	&\lesssim
	\|J_{s,t}\|_2+\Gamma_{s,t}(t-s)^{(\gamma-2)\alpha},
\end{align}
and
\begin{align}
	\|\E_sR^{\tilde Z}_{s,t}\|_2
	&\lesssim
	\|J_{s,t}\|_2+\Gamma_{s,t}(t-s)^{(\gamma-1)\alpha}.
\end{align}
Indeed, the first estimate is the analogue of Step~1 of \cref{prop.compose_stab} and the trivial bound $|R^{\tilde Z}_{s,t}|\le |\delta \tilde Z_{s,t}|+|\tilde Z'_s||\delta X_{s,t}|$; the second follows from the estimate for the increment of
$Df(Y)f(Y)+f'(Y)-Df(\bar Y)f(\bar Y)-f'(\bar Y)$ in Step~2; and the third follows from the remainder decomposition used in Step~3.

Since
\[
	\delta J_{s,u,t}
	=
	-\delta A_{s,u,t}
	=
	R^{\tilde Z}_{s,u}\delta X_{u,t}
	+\delta\tilde Z'_{s,u}\XX_{u,t},
\]
the preceding bounds imply, for $s\le u\le t$,
\[
	\|\delta J_{s,u,t}\|_2
	\lesssim
	\Big(\sup_{[r,v]\subset[s,t]}\|J_{r,v}\|_2\Big)(t-s)^\alpha
	+\Gamma_{s,t}(t-s)^{2\alpha},
\]
and, using also $\gamma\alpha=1$,
\[
	\|\E_s\delta J_{s,u,t}\|_2
	\lesssim
	\Big(\sup_{[r,v]\subset[s,t]}\|J_{r,v}\|_2\Big)(t-s)^\alpha
	+\Gamma_{s,t}(t-s).
\]
This shows that $J$ satisfies \eqref{con.dJ} and \eqref{con.EdJ}. 
}

			We view $\tilde Y$ as an element in $CL_2$ and suppose that $\tilde Y\neq0$ on $[0,T]$. Since $\tilde Y_0=0$ and $\tilde Y$ belongs to $C^\alpha L_2$, for $k_0$ sufficiently large, we can find a strictly decreasing sequence $\{t_k\}_{k\ge k_0}$ in $[0,T]$ such that for each $k$, $\|\tilde Y_t\|_2<2^{-k}$ for $0<t<t_k$ and $\|\tilde Y_{t_k}\|_2= 2^{-k}$.
			Since $\tilde Y$ is $L_2$-integrable, we have that $\Gamma_0:=\sup_{t\in[0,T]}\|\tilde Y_t\|_2$ is finite.
			The previous argument shows that for each $k$, $J$ satisfies \eqref{con.sll1}-\eqref{con.EdJ} on $\Delta([t_{k+1},t_k])$  with $m=2$, $\eta=2 \alpha$ and $\Gamma=\Gamma_{t_{k+1},t_k}=\sup_{t\in[t_{k+1},t_k]}\|\tilde Y_t\|_2$. Hence, by \cref{lem.Davie_iteration}, we can find an $\ell>0$, such that
			\begin{multline*}
				\|J_{t_{k+1},t_k}\|_2\lesssim \Gamma_{t_{k+1},t_k}\Big (1+|\log{\Gamma_{t_{k+1},t_k}}|+|\log{(t_k-t_{k+1})}|\Big )(t_k-t_{k+1})
				\\
				+\Gamma_{t_{k+1},t_k}(t_k-t_{k+1})^{2 \alpha}
			\end{multline*}
			for every $k$ sufficiently large so that $t_k-t_{k+1}\le \ell$.
			We now observe that $\Gamma_{t_{k+1},t_k}\le 2^{-k}$, $\|\delta \tilde Y_{t_{k+1},t_k}\|_2\ge\|\tilde Y_{t_k}\|_2- \|\tilde Y_{t_{k+1}}\|_2= 2^{-k-1}$ and take into account \eqref{est.tiY} to obtain that
			\begin{align*}
				2^{-k-1}\le \|\delta\tilde Y_{t_{k+1},t_k}\|_2
				\lesssim 2^{-k}\Big(1+k+|\log{(t_k-t_{k+1})}|\Big)(t_k-t_{k+1})+2^{-k}(t_k-t_{k+1})^{\alpha}.
			\end{align*}
			This implies $t_k-t_{k+1}\ge C(1+k)^{-1}$ for some constant $C>0$. Hence, we have $\sum_{k\ge k_0}(t_k-t_{k+1})=\infty$, which is a contradiction.
			It follows that $Y_t=\bar Y_t$ a.s.\ for each $t\in [0,T]$. Since both processes are a.s.\ continuous, they are indistinguishable. \hfill\qed

\subsection{Rough It\^o formula} \label{sec:ito}
Let us start with a digression on the main integrability result \cref{thm.Hroughint}.
Herein we let \( \beta\in (0,\alpha] .\) 
\subsubsection{Extended stochastic controlled rough paths}
While the space \( \D_X^{\beta,\beta'}L_{m,n} \) was needed to address solvability results for SRDEs, when dealing with integration and composition purposes it is enough to use a slightly larger class of stochastic processes, obtained simply by replacing $\delta Z'$ in \cref{def.SCRP}-\ref{good2} by its averaged-type analogue, that is $\E_{\bigcdot}\delta Z'$. 

Suppose that $Z\colon \Omega\times I\to W$ and $Z'\colon\Omega \times I\to \mathcal L(V,W)$
are $\{\mathcal F_t\}$-progressively measurable
and such that
\begin{equation}\label{nota.Gamma}
\Gamma^{\beta,\beta';m,n}(Z,Z';I)
:=
\|\delta Z\|_{\beta;m,n;I} +\sup_{r\in I}\| Z'_r\|_n
+\|\E_{\bigcdot} R^Z\|_{\beta+\beta';n;I} +\|\E_{\bigcdot}\delta Z'\|_{\beta';m,n;I}
\,<\infty
\end{equation}
where \(R^Z_{s,t}=\delta Z_{s,t}-Z'_s \delta X_{s,t}\).
Recalling the notations \eqref{nota.Gamma_1}-\eqref{nota.Gamma_2}, it is clear from that definition that
\[
\Gamma_1^{\beta, \beta';m,n}(\X,\delta Z',R^Z;I)\vee
\Gamma_2^{\beta,\beta';m,n}(\X,\delta Z',R^Z;I)
\lesssim \Gamma^{\beta,\beta';m,n}(Z,Z';I),
\]
where the implicit constant depends on \( \rho_{\alpha,\beta}({\X}) \).
This asserts in particular that \( \int Zd\X \) is well-defined, in the sense of \cref{thm.Hroughint}. 
Our preliminary discussion motivates the next definition.
	\begin{definition}[Extended stochastic controlled rough paths]
	\label{def.ESCRP}
		We say that $(Z,Z')$ is an \textit{extended stochastic controlled rough path} of $(m,n)$-integrability and $(\beta,\beta')$-H\"older regularity with values in $W$ with respect to $\{\cff_t\}$
		if \ref{good1}, \ref{goodgood}, \ref{itm:remainder} of \cref{def.SCRP} hold together with
		\begin{enumerate}[label=(\alph*')]
\setcounter{enumi}{3}
		 	\item\label{escrp_prime} $\sup_{t\in I}\|Z'_t\|_n$ is finite and $\E_\bigcdot\delta Z'$ belongs to $C^{\beta'}_2L_{n}([0,T],\W;\mathcal L(V,W))$;
		\end{enumerate}
	The class of such processes is denoted by $\DD_X^{\beta,\beta'}L_{m,n}([0,T],\W;W)$, or simply
	 $\DD_X^{\beta,\beta'}L_{m,n}$.
	\end{definition}

\subsubsection{Main result and discussion}
We now prove a {\em rough (stochastic) It\^o formula}, to be compared with the classical It\^o formula and the rough It\^o formula \cite[Ch.7]{FH20}.
We call {\em rough It\^o process} any continuous adapted process with dynamics,
\begin{equation} \label{equ:rIto0}
 dY_t (\omega) = b_{t, \omega} dt + \sigma_{t, \omega} d B_t + (Y',Y'')_{t, \omega} d \X_t,
\end{equation}
provided this makes sense (in integral form), with the final term is understood in the sense of rough stochastic integration (\cref{thm.Hroughint}).
Our aim is to show, for $t \in [0,T]$ and with probability one,
\begin{align}
\varphi \left( Y_{t}\right)
& -
\varphi \left( Y_{0}\right)
    - \int_{0}^{t} D \varphi \left(Y_{s}\right) 
    \sigma_{s, \omega} d B_{s}
    - \int_0^t (\mathcal L_{s, \omega} \varphi) \left(Y_s\right) d s   \nonumber \\
&\qquad = \label{equ:rIto}
   \int_{0}^{t} D \varphi \left(Y_{s}\right) 
    Y'_s d \X _{s}  %
  + \frac{1}{2}\int_{0}^{t}D^{2} \varphi\left( Y_{s}\right) \bigl(Y_{s}^{\prime }, Y_{s}^{\prime }\bigr)\,d [ \X ]_{s}    \\
&\qquad =   \label{equ:rStrat}
\int_{0}^{t} D \varphi \left(Y_{s}\right) 
    Y'_s \circ d \X_s %
    - \frac{1}{2} \int_{0}^{t} D \varphi \left(Y_{s}\right) 
    Y''_s d [\X]_{s},
\end{align}
for sufficiently regular $\varphi$, with
\[
	(\mathcal L_{s, \omega} \varphi) (y) := b_{s, \omega} \cdot D \varphi (y) + \tfrac{1}{2} (\sigma \sigma^\dagger)_{s, \omega}: D^2 \varphi (y )
=: b_{s, \omega} \cdot D \varphi (y) + \tfrac{1}{2} a_{s, \omega}: D^2 \varphi (y ),
\]
and rough path bracket $[\X] \equiv (\delta X)^{\otimes 2}- 2\sym(\XX)$, as defined in \cite[Ex. 2.11]{FH20}. We also wrote $\circ d \X \equiv d \X^g$ to denote (stochastic) rough integration against the ``geometrification'' of $\X = (X, \XX)$, explicitly given by $\X^g := (X, \mathrm{Anti}(\XX) +  (\delta X)^{\otimes 2} / 2)$, where  \( \sym(\XX), \mathrm{Anti}(\XX) \) denotes the (resp. anti-)symmetric part of $\XX$; pointwise in $V \otimes V$, cf. Definition \ref{def.RP}. In case of geometric $\X$, we have $[\X] \equiv 0$ and there is no difference between \eqref{equ:rIto} and \eqref{equ:rStrat}.

\begin{theorem}[Rough It\^o] \label{thm.rIto} Let $b,\sigma$ be bounded, progressively measurable,
$\X=(X,\XX) \in \mathscr{C}^\alpha$ for some $\alpha\in (\frac13,\frac12]$, and consider a test-function $\varphi \in \C^\gamma_b$ for some \( \gamma\in(\frac1\alpha,3] \).
Suppose that \reply{\(\sup_{t\in[0,T]}\|Y'_t\|_n<\infty\)} and let the pair $(Y',Y'')$ be an extended stochastic controlled rough path in \(\DD_{X}^{\beta,\beta'} L_{4,n}\), for some parameters \( n>4,\) \(0<\beta' \le \beta\le \alpha\), $\beta''=\min\{\alpha(\gamma-2),\alpha(\tfrac n4-1),\beta'\}$, subject to the conditions%
\footnote{At first reading, take \( \gamma=3 \), \( n=\infty \) and \( \alpha=\beta=\beta' \); then \eqref{parameters_ito} is implied by the condition \( \alpha>\frac13 \).}
\begin{equation}
\label{parameters_ito}
\alpha+\beta>\frac12 \quad \text{and}\quad
\alpha+\beta+\beta''>1\,.
\end{equation}
For each $y$ define%
\footnote{Also write $(\T_{t,\omega} \varphi, \T'_{t,\omega} \varphi)$ to emphasize the progressive nature of this process.}
$$
(\T \varphi, \T' \varphi) (y)  :=
\Big(D\varphi\left( y\right) Y^{\prime }, \ D^2 \varphi \left( y\right) (Y^{\prime },Y^{\prime })
+ D \varphi \left( y \right) Y^{\prime \prime }\Big).
$$


Then, $(\T\varphi, \T'\varphi)(Y)$ belongs to  $\DD_{X}^{\beta, \beta''} L_{2,2}$  %
and the rough stochastic It\^o formulas \eqref{equ:rIto}, \eqref{equ:rStrat} hold.

In these formulas, the integral in
$d\X$ is the rough stochastic integral $\int (\T \varphi, \T' \varphi) (Y) d \X$ and bracket integrals are Young-type integrals constructed by
stochastic sewing, with mesh limit taken in \(L_2\).
\end{theorem}

Examples of rough It\^o processes, to which this It\^o formula is applicable, includes
general RSDEs solutions (as provided by \cref{thm.fixpoint}) with
\[
\begin{aligned}
&b_{t, \omega} = b_t (\omega, Y_t (\omega)),\quad
 \sigma_{t, \omega}  = \sigma_t (\omega, Y_t (\omega)),
\\
& (Y',Y'')_{t, \omega} = \Big(f_{t}\big(\omega,Y_t (\omega)\big), ((Df_{t})f_{t}+ f_{t}')\big(\omega,Y_t (\omega)\big)\Big).
\end{aligned}
\]
This setting also accommodates McKean–Vlasov equations with rough common noise, in which case $b_{t,\omega} = \tilde b_t ( \omega, Y_t (\omega)) = b_t (\omega, Y_t (\omega), \mu_t)$ where $\mu_t$ is the law of $Y_t$, denoted by $\law(Y_t ; \X)$, and similar for the other coefficient fields. (Well-posedness of such rough McKean--Vlasov equations is treated in \cite{friz2025McKean}; our point here is only that solutions are rough It\^o processes, hence amenable to It\^o's formula.)
%

We see many potential applications of  \eqref{equ:rIto}, and various extensions thereof, 
notably in the area of rough (stochastic) PDEs and rough (doubly stochastic) BSDE, also in mean-field situations, all left to subsequent investigations. 
That said, we showcase a concrete use of \cref{thm.rIto} in revisiting some concepts (martingale problem, Fokker-Planck equation)
that will be familiar to many readers with stochastic analysis background. For any sufficiently nice test function $\varphi$, 
Theorem \ref{thm.rIto} allows to define a martingale,
\begin{equation} \label{equ:roughMP}
      M^\varphi_t  := \varphi \left( Y_{t}\right) - \varphi \left( Y_{0}\right)
    - \int_0^t (\mathcal L_{s, \omega} \varphi) \left(Y_s\right) d s   \\
  - \int_{0}^{t} (\T_{s, \omega} \varphi, \T'_{s, \omega} \varphi) (Y_s) d \X _{s},
\end{equation}
where we have assumed that $\X$ is geometric (for simplicity only, otherwise carry along a $d [\X]$- integral, see Remark \ref{rem.nongeo} below for details).
We say that $Y=Y(\omega)$ solves the {\em rough martingale problem},
$\mathrm{RMP}(\mathcal L; \T, \T';\X)$.
Mind that all coefficients fields are progressive and we are far from a Markovian situation.

 Even so, we can see that the flow of probability laws of $Y_t$ is
 measure-valued solution to an {\em effective} rough Fokker--Planck equation. To this end, define effective Markovian characteristics, i.e. (measurable) functions given by 
 \[
 	\bar b_t (y) := \E ( b_{t,\omega} | Y_t = y),
	\enskip
	\bar a_t (y) := \E ( a_{t,\omega} | Y_t = y),
 \]
 with effective $(\LL_t \varphi) (y) = \E ( \mathcal L_{t,\omega} \varphi) (y) | Y_t = y) $, equivalently defined as $\mathcal L_{t,\omega}\varphi$ above, but using the effective data $(\bar b_t, \bar a_t)$, which we may assume jointly measurable (cf. Proposition 5.1 in \cite{Brunick2013}.)
 
 We further define $(\TT_t \varphi) (y) = \E( ( \T_{t,\omega} \varphi) (y) | Y_t = y) $ and similarly $\TT'$.

 \begin{theorem} \label{thm:mimick0} Let $Y$ be a rough It\^o process of the form \eqref{equ:rIto0}, for some geometric rough path $\X$, subject to the condition of  \cref{thm.rIto}. Then the flow of (deterministic) probability measures $\mu_t = \law (Y_t ; \X)$ satisfies the measure-valued rough partial differential (forward) equation
 \[
	 d \mu_t = \LL^\star_t \mu_t dt + \TT_t^\star \mu_t d \X_t, \quad \mu_0 = \law (Y_0 ; \X),
 \] understood in analytically weak and integral sense. More precisely, for all $\varphi \in \C^\gamma_b$ for some $\gamma\in(\frac{1}{\alpha},3]$, we have
 \begin{equation} \label{equ:mimick0}
            \langle \mu_t , \varphi \rangle =   \langle \mu_0 , \varphi \rangle
             + \int_0^t \langle  \mu_s , (\LL_s \varphi ) \rangle ds + \int_0^t  ( \langle \mu_s , \TT_s \varphi \rangle,     \langle \mu_s , \TT'_s \varphi \rangle ) d \X_s.
\end{equation}
\end{theorem}
\begin{remark}\label{rem.nongeo} The rough forward equation of  \cref{thm:mimick0} is not valid as written for non-geometric rough paths. Indeed,
let us introduce random first and second order differential operators \( \T'_{1;t,\omega}\varphi(y) =D\varphi(y)Y_t'' (\omega) \), \( \T'_{2;t,\omega}\varphi(y)=D^2\varphi(y)(Y'_t,Y'_t) (\omega) \)
so that \( \T'=\T'_1+\T'_2 \).
A look at \eqref{equ:rStrat} reveals that
the correct equation involves a Young drift correction
and reads
 \[
 \begin{aligned}
d \mu_t &= \LL^\star_t \mu_t dt + \TT_t^\star \mu_t d \X + \tfrac12(\TT'_{2;t})^\star \mu_t d [\X]
\\&
=\LL^\star_t \mu_t dt + \TT_t^\star \mu_t \circ d \X -\tfrac12(\TT'_{1;t})^\star\mu_t d[\X]
\end{aligned}
\]
where for \( i=1,2\), we wrote \( \TT'_{i;t}\varphi(y):= \mathbb E[\T'_{i;t,\omega}\varphi(y)|Y_t=y]\).
 \end{remark}

\begin{example} Assume rough McKean--Vlasov dynamics with progressively measurable coefficients, with $\mu_t = \law (Y_t ; \X)$ where
 \[
 	dY_t (\omega) = b_t ( \omega, Y_t (\omega), \mu_t ) dt + \sigma_t ( \omega, Y_t (\omega), \mu_t )d B_t + f_t(\omega,  Y_t (\omega), \mu_t )d \X_t.
 \]
 The stochastic rough integral is understood as $\int (Y',Y'') d \X$ with $Y'_{t,\omega} := f_t(\omega,  Y_t (\omega), \mu_t )$,
 and $Y''_{t,\omega}$ given as sum of $((Df_t)f_t)(Y_t (\omega))$ and a term that captures the controlled structure (in $t$)  induced by the dependence on $\law (Y_t ; \X)$, the full specification of which is left to
 \cite{friz2025McKean}. There, the reader can also find existence and uniqueness of such equations, together with propagation of chaos results \cite{BFHL25} with fixed (rough) common noise. Upon randomization of
 $\X$, similar to Appendix \ref{app:RSDE}, this yields a (common noise) robustification of an important class of equations, e.g. \cite{lacker2022superposition}, where the authors also emphasizes the importance of
 random coefficients. (This seems out of reach of previous work on rough McKean--Vlasov which however dealt with a different problem: the case of random rough paths $\X = \X (\omega)$ which is not at all our goal; a more detailed literature review is left to \cite{friz2025McKean}.) 

  In the ``Markovian'' McKean--Vlasov situation when coefficient dependence $( \omega, Y_t (\omega), \mu )$ is replaced by $(Y_t (\omega), \mu )$, the conditioning procedure for the coefficients is trivial, i.e.
  $$\bar b_t (y) = b_t (y, \mu_t),
\enskip
 \bar \sigma_t (y) = \sigma_t (y, \mu_t), $$ and one easily arrives at the rough forward equation
 \begin{equation} \label{equ:RFE}
	 d \mu_t = (\mathcal L_t[\mu_t])^\star \mu_t dt + ( \T_t [\mu_t])^\star \mu_t d \X,
 \end{equation}
 with second order differential operator $\mathcal L[\nu]$ given by
 \[
 	(\mathcal L_t[\nu]) \varphi (y) = b_t (y,\nu) \cdot D \varphi (y) + \tfrac{1}{2} (\sigma_t \sigma_t^\dagger)(y,\nu) : D^2 \varphi (y ).  \]
 \end{example}

 \begin{remark} In the setting of the random rough approach to RSDEs, it is seen in \cite{MR3746646, coghi2019rough} that uniqueness results for such rough forward equations can be obtained by forward-backward duality, more specifically if one has a (spatially) regular solution to the rough Kolmogorov backward equation in duality with \eqref{equ:mimick0}.
  \end{remark}
  \noindent {\bf Update}:  At revision stage, a complete analysis of \eqref{equ:RFE} has been achieved in \cite{BFS25}.  

 \begin{proof}[Proof of \cref{thm:mimick0}] Taking expectation in \eqref{equ:roughMP}, one sees that
$$ 
 \E \varphi \left( Y_{t}\right) = \E \varphi \left( Y_{0}\right)
    + \E \int_0^t (\mathcal L_{s, \omega} \varphi) \left(Y_s\right) d s  
  + \E \int_{0}^{t} (\T_{s, \omega} \varphi, \T'_{s, \omega} \varphi) (Y_s) d \X _{s}.
 $$
 Applying Fubini theorem and tower property for conditional expectations, we have 
$$ \E \int_0^t (\mathcal L_{s, \omega} \varphi) \left(Y_s\right) d s = \int_0^t  \E [ (\mathcal L_{s, \omega} \varphi) \left(Y_s\right) ] d s = \int_0^t 
 \E[ (\bar{\mathcal L}_{s} \varphi) \left(Y_s\right) ] d s, $$
where we note that measurability of 
$ (\bar{\mathcal L}_{s} \varphi) \left(Y_s\right) = 
\bar b_{s} \cdot D \varphi (Y_s) + \tfrac{1}{2} a_{s}: D^2 \varphi (Y_s)
$
follows e.g. by \cite[Prop 5.1]{Brunick2013}. To see that 
 \[
	 \E \int_{0}^{t} (\T_{s, \omega} \varphi, \T'_{s, \omega} \varphi) (Y_s) d \X _{s}
    =
    \int_{0}^{t}  \E [ \T_{s, \omega} \varphi (Y_s) , \T'_{s, \omega} \varphi (Y_s)  ]  d \X _{s}
    =
     \int_{0}^{t} \E [ (\bar \T_s \varphi, \bar \T'_s \varphi)(Y_s)] d \X _{s},
 \]
we need a stochastic variant of rough Fubini \cite[Ex. 4.10]{FH20}, since $(Z,Z')=(\T_{s, \omega} \varphi, \T'_{s, \omega} \varphi) (Y_s)$ is only {\em stochastically} controlled. Fortunately, the argument is simple. By the very definition
 of the stochastic rough integral $\int (Z,Z') d \X $, it is the stochastic sewing limit of $A :=Z\delta X + Z'
 \XX$, which satisfies the assumptions of stochastic sewing. Since $\E_{\bigcdot}$ can only lower any given moment norm, it is clear that $\E A= (\E Z) \delta X + (\E Z')
 \XX$ then satisfies the assumption of the classical sewing lemma. But then we have that the (random) Riemann sums based on $A$ converge in some moment norm to the stochastic rough integral $\int (Z,Z') d\X$, hence their means  also converge. But these means are exactly the  (deterministic) Riemann sums based on $\E A$, convergent by the classical sewing lemma.
\end{proof}

\subsubsection*{Proof of the rough It\^o formula}
\begin{proof}
[Proof of \cref{thm.rIto}] In view of definition of $\beta''$ there is no loss of generality in assuming \( n\in (4,4(\gamma-1)]\).

Let us first check that the right hand side of \eqref{equ:rIto} is meaningful, which will be the case if the stochastic rough integral is well-defined. We note that 
\reply{the assumptions on the coefficients, together with
\(\sup_{t\in[0,T]}\|Y'_t\|_n<\infty\), ensure that
\((Y,Y')\in\D_X^{\alpha,\beta}L_{4,n}\).}
From the stability of compositions (\cref{lem.composecvec}) and the fact that $D\varphi$ has regularity $\bar \gamma-1:=\frac n4$,
we have that
\[
(Z,Z')=(D\varphi(Y),D^2\varphi(Y)Y')\in \D^{\alpha,\beta\wedge[(\bar\gamma-2)\alpha]}_XL_{4,4}.
\]
Moreover, we have that 
\( (\T\varphi,\T'\varphi)(Y)=(ZY',Z'Y' + ZY'') \) defines an extended stochastic controlled rough path such that
\begin{equation}
\label{Gamma_well_def}
(\T\varphi,\T'\varphi)(Y)\in \DD^{\beta,\beta''}_{X}L_{2,2}\,.
\end{equation}
\reply{Indeed, from the algebraic identity \( R^{ZY'}_{s,t}= R^{Z}_{s,t}Y'_{s} + \delta Z_{s,t}\delta Y'_{s,t} + Z_s R^{Y'}_{s,t} \), using H\"older inequality, we have 
	\begin{multline*}
		\|\E_{\bigcdot} R^{ZY'}\|_{\beta+\beta'';2} \lesssim \|\E_{\bigcdot} R^Z\|_{\beta+\beta'';4}\sup_{t}\|Y'_t\|_n + \|\delta Z\|_{\beta;4,4}\|\delta Y'\|_{\beta';4,n} 
		\\+ \sup_t \|Z_t\|_4 \|\E_{\bigcdot}R^{Y'}\|_{\beta+\beta';n}	,
	\end{multline*}
	which shows that $\E_{\bigcdot} R^{ZY'}$ belongs to $C^{\beta+\beta''}_2L_2$. From 
	   $ \delta(ZY'')_{s,t} = \delta Z_{s,t}Y''_{t} + Z_s\delta Y''_{s,t}$, we have
\begin{align*}
		\|\E_{\bigcdot}\delta(ZY'')\|_{\beta'';2} &\lesssim \|\delta Z\|_{\alpha;4}\sup_{t}\|Y''_t\|_n + \sup_t \|Z_t\|_4 
		\|\E_{\bigcdot}\delta Y''\|_{\beta';n}
	\end{align*}
	with similar estimates for $\|\E_{\bigcdot}\delta(Z'Y')\|_{\beta'';2}$, which shows that $\E_{\bigcdot}\delta(Z'Y'+ZY'')$ belongs to $C^{\beta''}_2L_2$. Analogously, one can show that $\delta(ZY')$ belongs to $C^{\beta}_2L_2$.
	Next, define 
	\[A_{s,t}=\wei{D\varphi(Y_s),\delta Y_{s,t}}+\wei{D^2\varphi(Y_s),Y_s^{\prime\otimes 2}\XX_{s,t}} +\tfrac12\wei{D^2\varphi(Y_s),\delta Y_{s,t}^{\otimes 2} - 2Y'^{\otimes 2}_s\sym \XX_{s,t} }.\]
	Taylor theorem shows that
	\begin{equation}
\label{taylor}
\begin{aligned}
\varphi(Y_t)-\varphi(Y_s)
&=\wei{D\varphi(Y_s),\delta Y_{s,t}} +\tfrac12\wei{D^2\varphi(Y_s), \delta Y_{s,t}^{\otimes2}}+O(|\delta Y_{s,t}|^{\gamma})
\\&
=: A_{s,t} + O(|\delta Y_{s,t}|^{\gamma})\,.
\end{aligned}
\end{equation}
For any partition \( \cpp \) of \( [0,t ]\), we find in particular that
\(  \varphi(Y_t)-\varphi(Y_0) = \sum_{[u,v]\in \cpp}[\varphi(Y_v)-\varphi(Y_u)]= \sum_{[u,v]\in \cpp} A_{u,v} + O(|\cpp|^{\gamma\alpha-1}) \) in $L_1$ using $\|\delta Y_{s,t}\|_4\lesssim (t-s)^\alpha$. Hence, we have
\[
 \varphi(Y_t)-\varphi(Y_0)= L_1\text{-}\lim_{|\cpp|\to0}\sum_{[u,v]\in\cpp} A_{u,v} \,.
\]
}

Let \[ \mathcal I_{s,t}:=\int_{s}^{t} D \varphi (Y_r)\sigma_rd B_r + \int_s^t (\mathcal L_r \varphi) (Y_r) d r + \int_s^t \T\varphi(Y_r)d \X _r + \frac{1}{2}\int_s^tD^2 \varphi( Y_r) \bigl(Y_r^{\prime }, Y_r^{\prime }\bigr)\,d [ \X ]_r\]
which is a well-defined adapted quantity in \( L_2 \)  for each \( (s,t)\in \Delta \) (by \eqref{Gamma_well_def}).
If we can show the existence of \( \lambda>\frac12 \) and \( \mu>1 \) such that
	 \begin{equation}\label{est.A1_ito}
			 	\|\mathcal I_{s,t}-A_{s,t}\|_{2} \lesssim (t-s)^{\lambda}\,
			 \end{equation}
			 and
			\begin{equation}\label{est.A2_ito}
				\|\E_s(\mathcal I_{s,t}-A_{s,t})\|_2\lesssim (t-s)^{\mu},
			\end{equation}
\reply{then by the uniqueness part of \cref{prop.SSL}, we have that $\mathcal I_{0,t}=\varphi(Y_t)-\varphi(Y_0)$ a.s. and the desired conclusion follows.}
\smallskip

\textit{Step 1: proof in the case when \( \X \) is geometric.}
To obtain the bounds \eqref{est.A1_ito}-\eqref{est.A2_ito},
we write (noting that \( [\X]=0\)),
\begin{equation}
\label{taylor_ito}
\begin{aligned}
\mathcal I_{s,t}-A_{s,t}
&=
\int_{s}^{t} (D\varphi(Y_r)-D\varphi (Y_s))\sigma_rd B_r
+\int_s^t(D\varphi(Y_r)-D\varphi (Y_s))b_rdr
\\&\quad
+ \Big(\int_s^t \T\varphi(Y_r)d \X _r-\int_s^t D\varphi(Y_s)Y'_rd \X _r -\wei{D^2\varphi(Y_s),Y_s^{\prime\otimes 2}\XX_{s,t}}\Big)
\\&\quad \quad
+\Big(%
\reply{
\frac12\int_s^t a_r:D^2\varphi(Y_r)\,dr
} %
-\frac12\wei{D^2\varphi(Y_s),\delta Y_{s,t}^{\otimes 2} - 2Y'^{\otimes 2}_s\sym \XX_{s,t} }\Big)
\\&
=J^1_{s,t} +\dots + J^4_{s,t}
\end{aligned}
\end{equation}
and estimate each term separately.
The first term is easily estimated through It\^o isometry.
Indeed, we have
\[
\|J^1_{s,t}\|_2 \lesssim (t-s)^{\frac12+\alpha}
\,,\quad
\|\E_sJ^1_{s,t}\|_2 =0\,,
\]
for an implicit constant only depending on \( |\varphi|_2 \).
Similarly, we find
\begin{equation}
\label{est.drift_ito}
\|J^2_{s,t}\|_2 \lesssim (t-s)^{1+\alpha}\,,
\quad
\|\E_sJ^2_{s,t}\|_2 \lesssim (t-s)^{1+\alpha}.
\end{equation}
For the third term,  putting
\begin{align*}
	E_{s,t}=\int_s^t Y'_rd\X_r-Y'_s\delta X_{s,t}-Y''_s\XX_{s,t},
\end{align*}
 we have
\begin{align*}
	J^3_{s,t}
	&=\int_s^t( \T\varphi,\T'\varphi)(Y_r)d \X _r - \T\varphi(Y_s)\delta X_{s,t} - \T'\varphi(Y_s)\XX_{s,t}
	- \wei{D\varphi(Y_s),E_{s,t}}.
\end{align*}
\reply{We apply \eqref{est.Hzm} and \eqref{est.Hcondz}, taking into account that \((Y',Y'')\in\DD_X^{\beta,\beta'}L_{4,n}\) and \((\T\varphi,\T'\varphi)(Y)\in\DD_X^{\beta,\beta''}L_{2,2}\), and that \(D\varphi(Y_s)\) is \(\cff_s\)-measurable, to obtain }
\[
\|J^3_{s,t}\|_2 \lesssim (t-s)^{\alpha+\beta}
\tand
\|\E_sJ^3_{s,t}\|_2 \lesssim (t-s)^{\alpha+\beta+\beta''}\,,
\]
\reply{where the implied constants depend on the corresponding extended controlled-path quantities and, as regards the test function, only on the finite quantity \( |\varphi|_\gamma \).\footnote{Indeed, by \cref{lem.composecvec} and the product estimates leading to
\eqref{Gamma_well_def}, the dependence of the extended
controlled-path quantities of
\((\T\varphi,\T'\varphi)(Y)\) on the test function is through
\[
|D\varphi|_\infty
+
|D^2\varphi|_\infty
+
[D^2\varphi]_{\bar\gamma-2},
\qquad
\bar\gamma=1+\frac n4\le\gamma.
\]
These quantities are bounded by \( |\varphi|_\gamma \). 
}

}

\reply{ 
	Finally, using $[\X]=0$ and \eqref{equ:rIto0},  we can write
\begin{multline*}
	\frac12\wei{D^2\varphi(Y_s),\delta Y_{s,t}^{\otimes 2} - 2Y'^{\otimes 2}_s\sym \XX_{s,t}}
	=\frac12\wei{D^2\varphi(Y_s),\delta Y_{s,t}^{\otimes 2} - Y'^{\otimes 2}_s\delta X_{s,t}^{\otimes2}}
	\\=
	\frac12 \Big\langle D^2\varphi(Y_s),(\int_s^t \sigma_{r}dB_r)^{\otimes2} + (\int_{s}^tY'_rd\X_r)^{\otimes2} - Y'^{\otimes 2}_s\delta X_{s,t}^{\otimes2} + \bar J^4_{s,t}\Big\rangle
\end{multline*}
where 
\begin{align*}
	\bar J^4_{s,t} &:=2 \sym (\int_s^t \sigma_{r}dB_r\otimes \int_{s}^tY'_rd\X_r )+O(t-s)^{1+\varepsilon} 
\end{align*}
and $O(t-s)^{1+\varepsilon}$, e.g. in $L_2$-sense, contains terms related to $\int bdr$. It is clear that $\|\bar J^4_{s,t}\|_2\lesssim (t-s)^{\alpha +\frac12}$. Moreover, by writing
\begin{align*}
	\E_s\left[\int_s^t \sigma_{r}dB_r\otimes \int_{s}^tY'_rd\X_r\right] 
	&=\E_s\left[\int_s^t \sigma_{r}dB_r\otimes (\int_{s}^tY'_rd\X_r -Y'_s\delta X_{s,t})\right]
\end{align*}
and applying \eqref{est.Hzm}, we see that $\|\E_s\bar J^4_{s,t}\|_2\lesssim (t-s)^{\alpha+\beta+\frac12}+O(t-s)^{1+\varepsilon}$.
Consequently, it follows from  It\^o isometry and standard arguments that
\[
\|J^4_{s,t}\|_2 \lesssim (t-s)^{\alpha+\frac12}\,,
\quad
\|\E_sJ^4_{s,t}\|_2 \lesssim (t-s)^{1+\varepsilon}\,,
\]
for some $\varepsilon>0$.
Hence our conclusion.\smallskip
}

\textit{Step 2: general case.}
In the notation of \eqref{equ:rStrat},
we remark that \eqref{equ:rIto0} is equivalent to
\begin{equation}
\label{equivalent_dynamics}
dY_t(\omega)= b_{t,\omega}dt -\tfrac12 Y''_{t,\omega} d[\X]_t + \sigma_{t,\omega}dB_t + (Y',Y'')_{t,\omega}\circ d\X_t
\end{equation}
where the second integral is a Young one.\footnote{More precisely, Young-type integral constructed by stochastic sewing.}
Indeed, we have the Davie-type expansion
\[
\delta Y_{s,t} - \int _s^t b_rdr -\int _s^t\sigma_rdB_r
= Y'_s\delta X_{s,t} + Y''_s(\XX^{g}-\tfrac12\delta[\X])_{s,t} + \tilde J_{s,t}.
\]
Moreover, the term
\[
\tilde J_{s,t}:=-\tfrac12\int_s^t Y''_{r}d[\X]_r + Y''_s\tfrac12\delta[\X]_{s,t}
\]
satisfies \reply{(by stochastic sewing, using $\beta'$-regularity of $\EE \delta Y''$ in $L_n$)
\[
\|\tilde J_{s,t}\|_2\lesssim(t-s)^{2\alpha},
\qquad
\|\E_s\tilde J_{s,t}\|_2
\lesssim(t-s)^{2\alpha+\beta'}.
\]
}
and so \eqref{equivalent_dynamics} is also a consequence of the uniqueness part of \cref{prop.SSL}.

Now, the claimed formula follows by the same argument as in Step 1, where the drift term is replaced by a mixed Lebesgue/Young integral. (In this case the estimate \eqref{est.drift_ito} has to be replaced by the inequalities
\(\|J^2_{s,t}\|_2 \lesssim (t-s)^{1+\alpha} + (t-s)^{3\alpha}\,,
\|\E_sJ^2_{s,t}\|_2 \lesssim (t-s)^{1+\alpha} +(t-s)^{3\alpha}
\).)
Applying the geometric rough It\^o formula, we have
\begin{multline*}
\varphi ( Y_{t}) -\varphi ( Y_{0})
    - \int_{0}^{t} D \varphi (Y_{s})\sigma_{s, \omega} d B_{s}
    - \int_0^t (\mathcal L_{s, \omega} \varphi) (Y_s) d s
\\
=-\frac12\int _0^t D\varphi(Y_s)Y''_sd[\X]_s
  + \int_{0}^{t} D \varphi (Y_{s})
    Y'_s d \X^g _{s}
\\
=  \int_{0}^{t} D \varphi (Y_{s}) Y'_s d (\X^g - (0,\tfrac12\delta[\X]))_{s} + \frac12\int _0 ^t D^2\varphi(Y_s)(Y_s',Y_s')d[\X]_s\,,
\end{multline*}
as claimed.
\end{proof}

\begin{remark}
\label{rem:M_f}
(a) The rough stochastic It\^o formulas \eqref{equ:rIto}, \eqref{equ:rStrat} still hold if in \cref{thm.rIto}, the integrability exponent  $(4,n)$ is replaced by $(m,\infty)$ for some $m \ge 2$.  Indeed, the  assumption
\(\sup_{t\in[0,T]}\|Y'_t\|_\infty<\infty\), together with
\((Y',Y'')\in\bar{\mathbf{D}}_X^{\beta,\beta'}L_{m,\infty}\), $m \ge 2$, implies that $(Y,Y')$ is in ${\mathbf{D}}_X^{\alpha,\beta,}L_{m,\infty}$. Moreover, the
composition estimate gives
\[
(Z,Z')
=
(D\varphi(Y),D^2\varphi(Y)Y')
\in
\mathbf{D}_X^{\alpha,\,
\beta\wedge[\alpha(\gamma-2)]}L_{m,\infty}.
\]
The product identities used in the proof, together with conditional
H\"older's inequality, then show that
\[
(\mathcal{T} \varphi, \mathcal{T} '\varphi)(Y)
\in
\bar{\mathbf{D}}_X^{\beta,\beta''}L_{m,\infty}.
\]
From here, the proof of \cref{thm.rIto} can be repeated verbatim to obtain the desired conclusion.

(b)
\cref{thm.rIto} provides an explicit decomposition of \( \varphi(Y) \) in terms of a martingale (stochastic integral) and a rough stochastic integral.
For lower regularity exponents e.g.\ when \( \gamma\in(1,\frac1\alpha] \), it is still true that $\varphi(Y)$ is the sum of a martingale and a random controlled rough path.
(This is reminiscent of the Fukushima-decompoition found e.g. in \cite[Thm 5.2.2]{FOT11}.)
Indeed, this is a consequence of the decomposition \cref{thm.mj}, which holds even more generally when \( \varphi \rightsquigarrow f_t(\omega,\cdot) \) has the structure of a stochastic controlled vector field.

More precisely, let $(f,f')$ be in $\D^{\beta,\beta'}_XL_{m,\infty}\C^{\gamma}_b$,  for some $\gamma\in(1,2]$, $m\in[2,\infty)$ and  $0<\beta'\le\beta\le\alpha$. Let $n\in[\gamma m,\infty]$ and $(Y,Y')$ be a stochastic controlled rough path in $\D_X^{\beta,\beta'}L_{m,n}$.
			We assume that
			\[
				\beta+\beta'' >\frac12,\quad\text{where}\quad \beta''=\min\{(\gamma-1)\beta,\beta'\}.
			 \]
			Then, there exist processes $M^f,Y^f$ such that
			\begin{enumerate}[(i)]
				\item\label{fmj1} $f_t(Y_t)=M^f_t+Y^f_t$ a.s.\ for every $t\in[0,T]$;
				\item $M^f$ is an $\{\cff_t\}$-martingale, $M^f_0=0$;
				\item\label{fmj3} $Y^f$ is $\{\cff_t\}$-adapted and satisfies
				\begin{multline*}
					\|\|Y^f_t-Y^f_s-(Df_s(Y_s)Y'_s+f'_s(Y_s)) \delta X_{s,t}|\cff_s\|_m\|_{\frac n{\gamma}}
					\\\lesssim([(f,f')]_{\gamma;\infty}+ \|(f,f')\|_{X;\beta,\beta';m,\infty})(1\vee|\delta X|_\alpha)(1\vee\|Y,Y'\|_{X;\beta,\beta';m,n}^{\gamma})|t-s|^{ \beta+\beta''}
				\end{multline*}
				for every $(s,t)\in \Delta$.
			\end{enumerate}
			Furthermore, given $(Y,Y')$ in $\D_X^{\beta,\beta'}L_{m,n}$, the pair of processes $(M^f,Y^f)$ is characterized uniquely by \ref{fmj1}-\ref{fmj3}. \smallskip

			Indeed,
			putting $(Z,Z')=(f(Y),Df(Y)Y'+f'(Y))$, we see from \cref{lem.composecvec} that $(Z,Z')$ is a stochastic controlled rough path in $\D_X^{\beta,\beta''}L_{m,\frac n{\gamma}}$ and 
			\begin{align*}
				\|(Z,Z')\|_{X;\beta,\beta'';m,\frac n{\gamma}}\lesssim ([(f,f')]_{\gamma;\infty}+ \|(f,f')\|_{X;\beta,\beta';m,\infty})(1\vee\|Y,Y'\|_{X;\beta,\beta';m,n}^{\gamma}).
			\end{align*}
			An application of \cref{thm.mj} gives the result.
		\end{remark}

\subsection{Weak solutions} 
\label{sub:weak_existence}
 	Herein, we study weak solutions of	\eqref{eqn.srde}.
 These are defined in such a way that is transparent from the corresponding classical notion for SDEs.
 Namely, given an initial probability distribution $\mu$ on $W$ and $m\ge2$,
	 a \textit{weak solution} to \eqref{eqn.srde}  starting from $\mu$ consists of a filtered probability space $(\tilde \Omega,\mathcal {\tilde G},\mathbb {\tilde P},\{\mathcal {\tilde F}_t\})$ together with a pair $(\tilde Y,\tilde B)$ such that $\tilde B$ is an $\{\tilde \cff_t\}$-Brownian motion in $\Vone$, $\law(\tilde Y_0)=\mu$, and $\tilde Y$ is an $L_{2,\infty}$-solution to
	\begin{equation}\label{def:weak_sol}
		d\tilde Y_t=b_t(\tilde Y_t)dt+\sigma_t(\tilde Y_t)d\tilde B_t+(f_t,f'_t)(\tilde Y_t)d\X_t, \quad t\in[0,T].
	\end{equation}
	Such weak solution is $L_{m,\infty}$-integrable if $\tilde Y$ is an $L_{m,\infty}$-solution on the stochastic basis $(\tilde \Omega,\mathcal {\tilde G},\mathbb {\tilde P},\{\mathcal {\tilde F}_t\})$.\smallskip

In constrast to other sections, we assume here that the coefficients in \eqref{def:weak_sol} are deterministic.
Namely, \( \omega\mapsto g_t(\omega,\cdot)\) is constant for every \( t\in I \)
and each \( g\in \{b,\sigma,f,f'\} \).
Likewise, we will call \( (f,f') \) a deterministic controlled vector field and write
\begin{equation}
\label{def.deter_f}
(f,f')\in \mathscr D_X^{\beta,\beta'}\mathcal C^{\gamma}_b
\end{equation}
if \( (f,f')\) is deterministic and belongs to \( \D_X^{\beta,\beta'}L_{m,n}\mathcal C^{\gamma}_b \) for some \( \beta,\beta'>0 \), \( \gamma>1 \), and (irrelevant) parameters \( n \) and \( m \).
We will abbreviate for convenience
\[
\bk{(f,f')}_{X;\beta,\beta'}:= \bk{(f,f')}_{X;m,n;\beta,\beta'},
\quad \quad
\|(f,f')\| _{\gamma}=\|(f,f')\|_{\gamma;n}
\]
(and so on).
\smallskip

	The first result is concerned about the existence of weak solutions in this setting.
	\begin{theorem}
	\label{thm:main_weak}
		Suppose that $b,\sigma$ are  bounded continuous  and \( (f,f')\) is a deterministic controlled vector field in \reply{ \( \mathscr D_X^{\beta,\beta'}\mathcal C_b^{\gamma-1} \) with $\frac13< \beta\le \alpha$, $2 \beta+\beta'>1$ and  $\gamma\in(\frac1 \beta,3]$.} 
		Let $\mu$ be a probability measure on $W$.
		Then for every $m\ge2$, there exists a weak solution $(\tilde \Omega,\mathcal {\tilde G},\mathbb {\tilde P},\{\mathcal {\tilde F}_t\};\tilde Y,\tilde B)$ to \eqref{eqn.srde} starting from $\mu$ which is $L_{m,\infty}$-integrable for every $m\ge2$.
	\end{theorem}
We need the following intermediate result.
	\begin{lemma}\label[lemma]{lem.lim}
		Let $\beta,\beta',m,n$ be as in \cref{thm.Hroughint}.
		Let $(Z,Z'), \{(Z^k,Z'^k)\}_{k\ge0}$ be extended stochastic controlled rough paths \reply{(with respect to different filtrations)} such that (recall the notation in \eqref{nota.Gamma})
		\begin{equation*}
			\Gamma^{\beta,\beta';m,n}(Z,Z';[0,T])\vee \sup_{k\ge0}
		\Gamma^{\beta,\beta';m,n}(Z^k,Z'^k;[0,T])
		<\infty
		\end{equation*}
		and for each $s\in[0,T]$, $\lim_k(Z^k_s,Z'^k_s)=(Z_s,Z'_s)$ in $L_m$. Then
		\begin{equation*}
			\lim_k\sup_{t\in[0,T]}\bigg|\int_0^tZ^kd\X-\int_0^tZd\X\bigg|=0
			\text{ in }L_m.
		\end{equation*}
	\end{lemma}
	\begin{proof}
		Define for each $(s,t)\in \Delta$, $A^k_{s,t}=Z^k_s \delta X_{s,t}+Z'^k_s\XX_{s,t}$ and similarly for $A_{s,t}$. By assumptions, we have for each $s\in[0,T]$, $\lim_k \sup_{t\in[s,T]}| A^k_{s,t}-A_{s,t}|=0$ in $L_m$.
		Applying \cref{thm.Hroughint,cor.SSLk} yields the result.
	\end{proof}

	\begin{proof}[Proof of \cref{thm:main_weak}]
		Using mollifiers, we can find a sequence of functions $\{b^n,\sigma^n,f^n,(f^n)'\}$ such that	$b^n,\sigma^n$ are  bounded Lipschitz functions (with respect to spatial variables), while $(f^n,(f^n)')$ belongs to \( \mathscr D_X^{\beta,\beta'}\mathcal C_b^3 \),
		\[
		 	\lim_n\sup_{t\in[0,T]}\left(|f_t^n-f_t|_{\tilde\gamma-1}+|(f^n)'_t-f'_t|_{\tilde\gamma-2}+|b^n_t-b_t|_\infty+|\sigma^n_t-\sigma_t|_\infty \right)=0
		\]
		\reply{for $\tilde\gamma\in(1/\beta,\gamma)$ and}
		\begin{multline*}
			\sup_n\left (\|(f^n,(f^n)')\| _{\gamma-1;[0,T]}+\bk{(f^n,(f^n)')}_{X;\beta,\beta';[0,T]}+|b^n|_\infty+|\sigma^n|_\infty\right )
			\\\le C(\|(f,f')\| _{\gamma-1;[0,T]},|b|_\infty,|\sigma|_\infty)\,.
		\end{multline*}

		Let $(\Omega,\cgg,\P,\{\cff_t\})$ be a probability space which support an $\{\cff_t\}$-Brownian motion $B$ and a random variable $\xi$ with law $\mu$.
		For each $n$, let $Y^n$ be the unique solution on $[0,T]$ to the rough stochastic differential equation
		\[
			dY^n=b^n(r,Y^n)dr+\sigma^n(r,Y^n)dB+(f^n,(f^n)')(r,Y^n)d\X, \quad Y^n_0=\xi.
		\]
		From \cref{thm.fixpoint}, $Y^n$ exists and is an $L_{m,\infty}$-solution for every $m\ge2$.
		From  \cref{prop.apri},
		 we see that for every $m\ge2$,
		\[
			\sup_n\|\delta Y^n\|_{\alpha;m}\le \sup_n\|\delta Y^n\|_{\alpha;m,\infty}<\infty\,.
		\]

		\reply {
			This in turn implies that the law of $\{(Y^n,B)\}_n$ is tight on $C([0,T];W)\times C([0,T];\Vone)$. Hence, after passing to a subsequence (not relabelled), the sequence $\law(Y^n,B)$ converges weakly to some probability measure $\nu$ on  $C([0,T];W)\times C([0,T];\Vone)$.  By Skorokhod's representation theorem, there exist a probability space $ (\widetilde\Omega,\widetilde{\mathcal G},\widetilde{\mathbb P})$ and random variables  $(\widetilde Y^n,\widetilde B^n), (\widetilde Y,\widetilde B)$,  such that $ \law(\widetilde Y^n,\widetilde B^n) = \law(Y^n,B)$,  $\law(\widetilde Y,\widetilde B)=\nu$,  and $ (\widetilde Y^n,\widetilde B^n) \longrightarrow (\widetilde Y,\widetilde B)$  almost surely in $C([0,T];W)\times C([0,T];\Vone)$. Since for every $n$, $ \law(\widetilde B^n)=\law(B)$, all the processes $\widetilde B^n$, and hence $\widetilde B$ have Wiener measure.  Let 
			\[ \widetilde{\mathcal F}_t := \sigma\!\bigl( \widetilde Y_s,\widetilde B_s: 0\le s\le t \bigr)^{\widetilde{\mathbb P}} \]
			denote the completed natural filtration of the pair $(\widetilde Y,\widetilde B)$. To verify that $\widetilde B$ is Brownian with respect to $\{\widetilde{\mathcal F}_t\}$, let $0\le s<t\le T$ and let $F,G$ be  bounded continuous functionals such that $F$ depends only on the path up to time $s$. For every $n$, 
			\[ \widetilde{\mathbb E} \Big[ F(\widetilde Y^n,\widetilde B^n) G(\widetilde B^n_t-\widetilde B^n_s) \Big] =\widetilde{\mathbb E} \Big[ F(\widetilde Y^n,\widetilde B^n) \Big] \widetilde{\mathbb{E}}\Big [G(\widetilde B^n_t-\widetilde B^n_s)  \Big] . \] 
			Passing to the limit by bounded convergence yields 
			\[ \widetilde{\mathbb E} \Big[ F(\widetilde Y,\widetilde B)  G(\widetilde B_t-\widetilde B_s) \Big] =\widetilde{\mathbb E} \Big[ F(\widetilde Y,\widetilde B) \Big] \widetilde{\mathbb{E}}\Big [G(\widetilde B_t-\widetilde B_s)  \Big]. \] 
			Hence $\widetilde B_t-\widetilde B_s$ is independent of $\widetilde{\mathcal F}_s$ and is Gaussian with covariance $(t-s)I$. Therefore $\widetilde B$ is a Brownian motion with respect to $\{\widetilde{\mathcal F}_t\}$. Finally, because $\law(\widetilde Y^n,\widetilde B^n) = \law(Y^n,B)$, each pair $(\widetilde Y^n,\widetilde B^n)$ satisfies the approximating RSDE. Passing to the limit in the drift term by bounded convergence, in the It\^o integral by the $L^2$-isometry, and in the stochastic rough integral by \cref{lem.lim}, we obtain that $(\widetilde Y,\widetilde B)$ satisfies the original RSDE. Consequently, the tuple $(\widetilde \Omega,\mathcal {\widetilde G},\mathbb {\widetilde P},\{\mathcal {\widetilde F}_t\},\widetilde Y,\widetilde B)$ forms a weak solution of \eqref{eqn.srde}.
		}
	\end{proof}
	We now turn our attention to uniqueness.
	\begin{theorem}[Uniqueness in law]\label{thm.weakuniq}
		Let $\sigma,b$ be bounded Lipschitz functions and suppose that both $(f,f'),(Df,Df')$ are deterministic controlled vector fields in $\mathscr D^{2\alpha}_X\mathcal C_b^{\gamma}$ and $\mathscr D_X^{\alpha,(\gamma-2)\alpha}\mathcal C_b^{\gamma-1}$ respectively, where $\gamma\ge1/\alpha$.
		Let $(Y,B,\{\cff_t\})$ and $(\bar Y,\bar B,\{\bar\cff_t\})$ be two integrable solutions to \eqref{eqn.srde}
		 defined respectively on stochastic bases $(\Omega,\cgg,\P)$ and $(\bar \Omega,\bar\cgg,\bar\P)$ such that $\law(Y_0)=\law(\bar Y_0)$. Then $Y$ and $\bar Y$ have the same law on $C([0,T];W)$.
	\end{theorem}
	\begin{proof}
		When $\gamma>1/\alpha$, for $T>0$ small enough, we have
		\[
		(Y,f(Y))=\lim_{n\to \infty}\underbrace{ \Phi^{T, B}\circ\Phi^{T, B}\dots \circ \Phi^{T, B}}_{n\text{ times }}
		\]
		where $\Phi^{T, B}$ is the fixed point map given by \eqref{fixed_point_map}
		(we emphasize here its dependency on the underlying Brownian motion). In particular, there is a measurable map $\Psi\colon C([0,T];\Vone)\to C([0,T];W)$ such that $ Y|_{ \Omega\times[0,T]}=\Psi( B)$.
		In the critical case when $\gamma=1/\alpha$, we reason in the following way.
		We choose a sequence  $\{(f^n,(f^n)')\}$ as in the proof of \cref{thm:main_weak}. Let $Y^n$ be the solution to \eqref{eqn.srde} with coefficients $(b, \sigma,f^n, (f^n)')$.
		By a tightness argument similar to the one in the proof of \cref{thm:main_weak}, we can find a complete filtered probability space $(\tilde \Omega,\{\tilde\cgg_t\},\tilde \P)$, an $\{\tilde\cff_t\}$-Brownian motion $\tilde B$ and processes $\tilde Y^n$ on it so that
		\begin{itemize}
		  	\item $(\tilde Y^n,\tilde B)\stackrel{law}{=}(Y^n,B)$,
		  	\item there is a subsequence $\{k_n\}$ so that $\lim_n\tilde Y^{k_n}=\tilde Y^{(k)}$  in $C([0,T];W)$ a.s.
		  	\item $\tilde Y^{(k)}$ as above is a solution to \eqref{eqn.srde} with coefficients $(b,\sigma,f,f')$.
		\end{itemize}
		Let $\{k_n\}$ and $\{l_n\}$ be two subsequences such that $\lim_n\tilde Y^{k_n}=\tilde Y^{(k)}$ and $\lim_n\tilde Y^{l_n}=\tilde Y^{(l)}$  in $C([0,T];W)$ a.s. Since $\tilde Y^{(k)}$ and $\tilde Y^{(l)}$ are solutions to \eqref{eqn.srde} on the same stochastic basis, by \cref{thm.uniq.notloc}, it is necessary that $\tilde Y^{(k)}=\tilde Y^{(l)}$. As in \cite{MR1392450}, this shows that the sequence $\{Y^n\}$ converges to a limit $Y$ in $C([0,T];W)$ which is a solution to \eqref{eqn.srde}.
		On the other hand, writing $Y^{n}=\Psi^n(B)$ by the previous argument for subcritical cases, we see that $Y=\Psi(B)=\lim_n \Psi^n(B)$ for a measurable function $\Psi$.

		Repeating the argument over any interval of the form $[nT,(n+1)T]$ for $n\in\mathbb N$, such a relation implies that the distribution of $Y$ under $\mathbb{ P}$ depends on $\sigma,f,b$ but not on the stochastic basis $(\Omega,\{\cff_t\},\P)$. This also gives $\P\circ Y^{-1}=\bar\P\circ\bar Y^{-1}$.
	\end{proof}

\appendix

\section{Randomized RSDEs and conditioned SDEs} \label{app:RSDE}

Part of our motivation was the ``partially quenched'' study of doubly SDEs, driven jointly by independent standard Brownian motion $B$ and $W$, but conditionally on $W$. (As always, $B$ and $W$ may be multidimensional.)  To keep in
technicalities to a minimum we consider $(B,W)$ given on a product stochastic basis, $\W = \W^B \otimes \W^W$, $\omega=(\omega^B,\omega^W)$,and let $\mathbf{W} = \mathbf{W} (\omega^W)$ be the It\^o-Brownian rough path over $W$ so that
$\mathbf{W} (\omega^W) \in \mathscr{C}^\alpha ([0,T])$, $1/3 < \alpha < 1/2$, for all $\omega^W$. According to Proposition \ref{prop.davie}, any solution $Y = Y^\X (\omega^B)$ to the RSDE
\begin{equation}\label{eqn.srdeA}
		dY_t(\omega)=b_t(\omega,Y_t(\omega);\X)dt+\sigma_t(\omega,Y_t(\omega);\X)dB_t(\omega)+(f_t,f'_t)(\omega,Y_t(\omega);\X)d\X_t, 
\end{equation}
with initial datum $Y_0 = \xi$,
satisfies, on the stochastic basic $\W^B$, an accompanying integral equation (cf. \eqref{eqn.int.srde}), where our notation highlights the fact that all coefficients may depend on $\X \in \mathscr{C}^\alpha$.
 Assume that 
\begin{itemize}
  \item  for every $\X$, there is a unique solution to \eqref{eqn.srdeA} provided by Picard iteration in some moment space (as provided by \cref{thm.fixpoint})   
  \item  all coefficients are progressive measurable, also w.r.t, $y \in\mathbb{R}^{d_Y} $ and the rough path $\mathbf{X} \in \mathscr{C}^\alpha$, in the precise sense of measurablilty w.r.t. the product $\sigma$-field of the progessive field, and the Borel sets of $\mathbb{R}^{d_Y}$ and $\mathscr{C}^\alpha$, respectively.
\item all coefficients are causal in $\mathbf{X} \in \mathscr{C}^\alpha$
\end{itemize}

\begin{theorem}\label{thm:doublyIto}
There exists a jointly progressively measurable version of $Y_t^\X (\omega^B)$ as function of $(t,\omega^B,\X)$. Its randomization $Y^{\mathbf{X}}\big|_{\mathbf{X}= \mathbf{W} (\omega^W)}$ admits a continuous modification denoted by $\bar{Y} (\omega)$, which solves on $\W$  the ``doubly'' It\^o stochastic differential equation

\begin{equation}\label{eqn.srdeA1}
		d \bar{Y}_t(\omega)=\bar{b}_t(\omega,\bar{Y}_t(\omega))dt+\bar{\sigma}_t(\omega,\bar Y_t(\omega))dB_t+(\bar f_t, \bar f'_t)(\omega,\bar Y_t(\omega))d W_t, 
\end{equation}
where, for $\Xi \in \{b,\sigma,f,f'\}$, we write $\bar{\Xi}_s(\omega) \coloneqq \Xi_s(\omega^B,\mathbf{W}(\omega^W))$. Moreover, 
\[
  \mathrm{Law} \big(\bar{Y}_t \,\big|\, \mathcal{F}^W_T\big)(\omega)
  = \mathrm{Law} \!\left(Y_t^{\mathbf{X}}\right)\Big|_{\mathbf{X}=\mathbf{W}(\omega)}
\]
which provides explicit access to the regular conditional distribution of $\bar{Y}_t$, given $ \mathcal{F}^W_T$.
\end{theorem}

This result appears in  \cite{FLZ25} where randomization of RSDEs is studied in the generality of rough It\^o process, building on \cite{FLZ24x}. Leaving details to these papers, the major remark here is that one has to solve an uncountable family of RSDEs, parametrized by $\X \in \mathscr{C}^\alpha$. The existence of suitable jointly measurable version is then obtained by measurable selection techniques, which, in turn,  guarantees measurability of subsequent randomization.
(Note that in the present generality of $\X$-dependent coefficients, one cannot expect continuity of $\X \mapsto Y^{\mathbf{X}}$.)

\section{John--Nirenberg inequality} 
\label{sec:john_nirenberg_inequality}
  We present a self-contained proof of \cref{prop.JNineq}. The main argument relies on the following result.
  \begin{proposition}\label[proposition]{prop.genJN}
    Let $V$ be a continuous adapted process. Suppose that for every $s\le t$, we have
    \begin{align}\label{tmp.Vst}
      \|\E|\delta V_{s,t}||\cff_s\|_\infty\le \Gamma(t-s)^\kappa.
    \end{align}
    Then there are universal finite constants $C,c>0$ which are independent from $\Gamma,\kappa,T$ such that
    \begin{align}\label{est.expV}
      \E e^{\lambda\sup_{t\in[0,T]}|\delta V_{0,t}|}\le C e^{c (\lambda\Gamma)^{1/\kappa}T}
      \quad\text{for every} \quad \lambda>0.
    \end{align}
  \end{proposition}
  \begin{proof}[Proof of \cref{prop.JNineq}]
    Define $V_t=|Y_t-Y_0|_\abx$. Then $V$ is a.s. continuous and satisfies
    \begin{align*}
      \left\|\E(|\delta V_{s,t}||\cff_s)\right\|_\infty
      \le
      \left\|\E(|\delta Y_{s,t}|_\abx|\cff_s)\right\|_\infty
      \le \|\delta Y\|_{\kappa;1,\infty}(t-s)^\kappa,
    \end{align*}
    for every $(s,t)\in \Delta$. From here, \cref{prop.JNineq} is a direct consequence of \cref{prop.genJN}.
  \end{proof}
  To show \cref{prop.genJN}, we need the following elementary result.
   \begin{lemma}\label[lemma]{lem.xy}
    If $X$ and $Y$ are nonnegative random variables satisfying
    \begin{align*}
      \P(Y> \alpha+\beta)\le \theta\P(Y>\alpha)+\P(X>\theta\beta)
    \end{align*}
    for every $\alpha>0$, $\beta>0$ and $\theta\in(0,1)$; then for for every $m\in(0,\infty)$,
    \begin{align*}
      \|Y\|_m\le c_m m\|X\|_m
    \end{align*}
    where the constant $c_m$ is given by $(c_m)^m=m(1+1/m)^{(m+1)^2}$. (Note that $\sup_{m\ge1}c_m<\infty$.)
  \end{lemma}
  \begin{proof}
    We choose $\beta= h \alpha$ for some $h>0$ and integrate the inequality with respect to $m \alpha^{m-1}d \alpha$ over $(0,k/(1+h))$ to get that
    \begin{multline*}
      (1+h)^{-m} \int_0^km \alpha^{m-1}\P(Y>\alpha)d \alpha\le \theta\int_0^{k}m \alpha^{m-1}\P(Y>\alpha)d \alpha
      \\+\int_0^\infty m \alpha^{m-1}\P(X>\theta h  \alpha)d \alpha.
    \end{multline*}
    Sending $k\to\infty$ and using the layer cake representation $\E X^m=\int_0^\infty m \alpha^{m-1}\P(X>\alpha)d \alpha$, we obtain that
    \begin{align*}
      \left[(1+h)^{-m}-\theta\right]\E Y^m\le (\theta h)^{-m}\E X^m.
    \end{align*}
    We now choose $h=\frac1m$ and $\theta=\left(\frac m{m+1}\right)^{m+1}$ to obtain the result.
  \end{proof}
  \begin{proof}[Proof of \cref{prop.genJN}]
    Let $\lambda>0$ be fixed.
    For each $(s,t)\in \Delta$, define
    \begin{align*}
      V^*_t=\sup_{r\in[0,t]}|\delta V_{0,r}|
      \quad\text{and}\quad
      M_{s,t}=\|\E_se^{\lambda(V^*_t-V^*_s)}\|_\infty.
    \end{align*}
    Following \cite{le2022BMO}, it is sufficient to establish that
    \begin{align}
      M_{s,t}\le M \quad\text{whenever}\quad 2\lambda \Gamma(t-s)^\kappa\le e^{-3}\label{tmp.M1}
      \shortintertext{and}
      M_{s,t}\le M_{s,u}M_{u,t} \quad\text{whenever}\quad s\le u\le t\label{tmp.submult}
    \end{align}
    for some universal finite constant $M$.
    Indeed, assume for the moment that \eqref{tmp.M1}-\eqref{tmp.submult} hold.
	Partitioning $[0,T]$ by points $0=t_0<t_1<\ldots<t_n=T$ so that $\lambda \Gamma(t_k-t_{k-1})^\kappa\le e^{-3}$ for each $k$, one sees that
    \begin{align*}
      M_{0,T}\le \prod_{k=1}^n M_{t_{k-1},t_k}\le M^n.
    \end{align*}
    Omitting details, one can then choose $\{t_k\}$ efficiently so that $n$ is approximately $1+T(e^3 \lambda \Gamma)^{\frac1 \kappa}$. With such choice, the above estimate for $M_{0,T}$ implies \eqref{est.expV}.\smallskip

    Being a simple consequence of conditioning, the proof of \eqref{tmp.submult} is left to the reader. 
	Inequality \eqref{tmp.M1} is a variant of the classical John--Nirenberg inequality for continuous processes (see \cite[Excercise A.3.2]{MR2190038}). Its proof is divided into several steps below.
    \smallskip

    \textit{Step 1.} We show that for every $(s,t)\in \Delta$ and every stopping time $\mu$ satisfying $s\le \mu\le t$, one has
    \begin{align}\label{est.stoS}
      \left\|\E(|\delta V_{\mu,t}| |\cff_\mu)\right\|_\infty\le \Gamma(t-s)^\kappa.
    \end{align}

    Indeed, fix $(s,t)\in \Delta$ and put $C=\Gamma(t-s)^\kappa$.
    Let $\mu$  be a stopping time, $s\le \mu\le t$, and suppose that $\mu$ takes finitely many values $\{s_1<\ldots<s_k\}$. We have
    \begin{align*}
      \E_\mu|V_t-V_{\mu}|
      &=\sum_{j}\E_\mu[|V_t-V_{\mu}|\mathbf{1}_{(\mu=s_j)}]
      =\sum_{j}\mathbf{1}_{(\mu=s_j)}\E_{s_j}[|V_t-V_{s_j}|]
      \le C,
    \end{align*}
    where we used \eqref{tmp.Vst} to obtain the last inequality.
    For a general stopping time $\mu$,  $s\le \mu\le t$, define for each $n$, the stopping time $\mu^n$,
    \begin{align*}
      &\mu^n=0\quad\text{if}\quad \mu=0,
      \\&\mu^n=j2^{-n}t \quad\text{if}\quad (j-1)2^{-n}t< \mu\le j2^{-n}t,\, j\le 2^nt.
    \end{align*}
    It is obvious that $\{\mu^n\}$ is decreasing to $\mu$ and $\mu^n\le t$.
    Then by triangle inequality
    \begin{align*}
      \E_\mu[|V_{t}-V_{\mu}|\wedge N]
      &\le \E_\mu\E_{\mu^n}[|V_{t}-V_{\mu^n}|]+\E_\mu[|V_{\mu^n}-V_{\mu}|\wedge N]
      \\&\le C+\E_\mu[|V_{\mu^n}-V_{\mu}|\wedge N].
    \end{align*}
    Note that $\lim_n V_{\mu^n}=V_\mu$ a.s. so that by Fatou lemma and Lebesgue dominated convergence theorem, we have $\E_\mu[|V_{t}-V_{\mu}|\wedge N]\le C$.
    Sending $N\to\infty$ yields \eqref{est.stoS}.

    \smallskip

    \textit{Step 2.}
    We show that \
    \begin{gather}\label{ineq.JNweighted}
      \|\|\sup_{s\le r\le t}|V_r-V_{s}| |\cff_s \|_m\|_\infty\le 2c_mm \Gamma(t-s)^\kappa.
    \end{gather}

    Fix $s,t$.
    Without loss of generality, we can assume that $2\Gamma(t-s)^\kappa=1$ so that by the previous step, for every stopping time $\mu$ with $s\le \mu\le t$, we have
    \begin{align}\label{tmp.constep1}
      \|\E_\mu|\delta V_{\mu,t}|\|_\infty\le 1/2 .
    \end{align}

      We put $V^*=\sup_{r\in[s,t]}|V_r- V_{s}|$.
      Let $\alpha, \beta$ be two positive numbers and define
      \begin{align*}
        \mu=t\wedge\inf\{r\in[s,t]:|V_r- V_{s}|> \alpha\},
        \quad
        \nu=t\wedge\inf\{r\in[s,t]:|V_r- V_{s}|> \alpha+\beta\},
      \end{align*}
      with the standard convention that $\inf(\emptyset)=\infty$ (so \( \mu=t \) and \( \nu=t \) when these sets are empty). Clearly $\mu$ and $\nu$ are stopping times and $s\le \mu\le \nu\le t$.

      On the event $\{V^*> \alpha+\beta\}$, we have $|V_\nu- V_{s}|\ge \alpha+\beta$ and $|V_\mu-V_s|\ge \alpha$.
      In view of the triangle inequality $|V_\nu-V_s|\le|V_\nu-V_{\mu}|+|V_{\mu}-V_s|$, this implies that
      \begin{align*}
        \{V^*> \alpha+\beta\}\subset\{|V_\nu-V_{\mu}|\ge \beta,\ V^*> \alpha\}.
      \end{align*}
      It follows that for every $G\in\cff_s$ and every $\theta\in(0,1)$,
      \begin{align*}
        \P(V^*>\alpha+\beta,\ G)
        &\le \P(|V_\nu-V_{\mu}|\ge \beta,\ V^*>\alpha,\ G)
        \\&\le \P(|V_\nu-V_{\mu}|\ge \theta^{-1},\ V^*>\alpha,\ G)
        +\P( 1> \theta \beta,\ V^*>\alpha,\ G).
      \end{align*}
      By conditioning, noting that $\{V^*>\alpha\}$ is $\cff_\mu$-measurable, and applying Markov inequality we have
      \begin{align*}
        \P(|V_\nu-V_{\mu}|\ge \theta^{-1},\ V^*>\alpha,\ G)
        \le \theta \|\E_\mu|\delta V_{\mu,\nu}|\|_\infty\P(V^*>\alpha,\ G).
      \end{align*}
      The conditional expectation is estimated using \eqref{tmp.constep1}, this yields
      \begin{align*}
        \P(|V_\nu-V_{\mu}|\ge \theta^{-1},\ V^*>\alpha,\ G)
        \le \theta\P(V^*>\alpha,\ G).
      \end{align*}
      Hence, we obtain from the above that
      \begin{align*}
        \P(V^*> \alpha+\beta,\ G)
        &\le \theta\P(V^*> \alpha,\ G)+\P( 1> \theta \beta,\ G).
      \end{align*}
      Applying \cref{lem.xy}, we get $\|V^*\mathbf{1}_G\|_m\le c_mm\|\mathbf{1}_G\|_m$.
      Given that $G$ is arbitrary in $\cff_s$, a classical argument entails \eqref{ineq.JNweighted}.

      \smallskip

      \textit{Step 3.}
      Fix $\lambda>0$. For $(s,t)$ such that $2\lambda \Gamma(t-s)^\kappa\le e^{-3}$, we have by Taylor's expansion and \eqref{ineq.JNweighted} that
      \begin{align*}
        \left\|\E_s \exp\left({\lambda\sup_{r\in[s,t]}|V_r-V_s|}\right)\right\|_\infty
        &\le\sum_{m=0}^\infty \frac{\lambda^m}{m!}\left\|\E_s\left(\sup_{r\in[s,t]}|V_r-V_s|\right)^m\right\|_\infty
        \\&\le1+\sum_{m=1}^\infty \frac{\lambda^m}{m!}(c_mm)^m(2 \Gamma(t-s)^\kappa)^m\le M
      \end{align*}
      where $M=1+\sum_{m=1}^\infty a_m$ and $ a_m= \frac{(c_mm)^m }{m!}e^{-3m}$.
      Because $\lim_{m\to\infty}\frac{a_{m+1}}{a_m}=e^{-1}$, $M$ is finite by the ratio test.
      Since $V^*_t-V^*_s\le\sup_{r\in[s,t]}|V_r-V_s|$, the previous estimate also implies that $\|\E_s e^{\lambda({V^*_t-V^*_s})}\|_\infty\le M$ whenever $2\lambda \Gamma(t-s)^\kappa\le e^{-3}$, which is equivalent to \eqref{tmp.M1}.
    \end{proof}

\section{Symbolic Index} \label{app:symbolic_index}

\renewcommand{\arraystretch}{1.2}\setlength{\tabcolsep}{6pt}%
\centering
\begin{longtable}{l p{0.62\linewidth} l}
%
%
%
%
\noindent%
$V, W$ & finite-dimensional Banach spaces & --- \\ 
$\abx, {\mathcal Y}$ & generic (not necessarily separable) Banach space & --- \\
$\W$ & stochastic basis $(\Omega,\mathcal G, \P;\{\cff_t\})$ & Sec.~\ref{sec:ss} \\
$\E, \E_s $ & (conditional) expectation (given $\cff_s$) & Eq. \ref{not:E_s} \\ 
$\E_\bigcdot A$ & 
$(s,t;\omega) \mapsto \E_s (A_{s,t} )(\omega)$ 
& Eq. \ref{shorthand_bigcdot} \\ 
$\X = (X,\XX)$ & generic $\alpha$-Hölder rough path 
& Def. \ref{def.RP} \\
$\rho_\alpha, \rho_{\alpha,\alpha'}, \rho_{\alpha,\alpha'}$ & distance between (resp. size of) H\"older rough path(s) 
& Eq. \ref{def.rho_metric} \\ 
$\delta Y_{s,t}$ & increment of a path $Y$: $Y_t - Y_s$ & Eq. \ref{nota:increments} \\ 
$\delta A_{s,u,t}$ & 3-point increment: $A_{s,t} - A_{s,u} - A_{u,t}$ & Eq. \ref{nota:increments_2} \\ 
$| \cdot |$ & generic norm on Banach space, also length of interval & --- \\ 
$\| \xi \|_m$ & standard (quasi-)norm of r.v. $\xi \in L_m$ & Sec.~\ref{sec:ss} \\
$\| \xi | \cff \|_m$ & conditional moment norm: $\E (| \xi |^m | \cff)^{\frac{1}{m}}$ & Eq. \ref{Lmn_norm} \\
$\| \xi \|_{m,n}$ & mixed moment norm: $\|\| \xi |\mathcal{F} \|_m \|_n$
& Sec.~\ref{sec:mms} \\ 
$|Y|_\infty$ & supremum norm: $\sup_t |Y_t|$, for path & --- \\
$[Y]_\alpha, |Y|_\alpha$ & Hölder (semi)norm: $|\delta Y|_\alpha$, $|Y|_\infty + [Y]_\alpha$ respectively  & Eq. \ref{holder_1} \\ 
$|A|_\kappa$ & Hölder norm: $\sup_{s<t}\frac{|\|A_{s,t}|}{|t-s|^\kappa }$ for $2$-parameter maps & Eq. \ref{holder_2} \\
$\|A\|_{\infty;m,n}$ & $\sup_{s<t}\|\|A_{s,t}|\cff_s\|_m\|_n$ & Def.~\ref{def.C2mn} \\ 
$\|A\|_{\kappa;m,n}$ & $\sup_{s<t}\frac{\|\|A_{s,t}|\cff_s\|_m\|_n}{|t-s|^\kappa}$ & Def.~\ref{def.2amn} \\
$\|A\|_{\kappa;m}$ &
$\|A\|_{\kappa;m,m}$  & Eq. \ref{Aam} \\
$\bk{A(.)}_{\kappa,m,n}$ & $\| |A|_\infty \|_{\kappa;m,m}$, with $|A|_{\infty;s,t} = \sup_x |A_{s,t}(x)|$ 
& Eq.~\ref{def.bk} \\ 
$\|Y\|_{\infty;m}$ & $\sup_{t } \| Y_ t \|_m$, for process \\
$\|Y\|_{\alpha;m,n}$ & $\|Y\|_{\infty;m} + \|\delta Y\|_{\alpha;m,n}$ & Eq. \ref{Ykmn} \\ 
$\|Y\|_{\alpha;m}$ & 
$\|Y\|_{\alpha;m,m}$. 
& 
Eq. \ref{Yam}  \\
$\|Y\|_{0;m}$ & 
$\|Y\|_{\infty;m} + \|\delta Y\|_{0;m} \asymp \| Y \|_{\infty;m} $ & 
Eq. \ref{Yam}  \\
$[f]_\alpha, |f|_\alpha$ & H\"older (semi)norm: $\sup_{x\neq y}\frac{|f(x)-f(y)|}{|x-y|^\alpha}$, $|f|_\infty + [f]_\alpha$ resp. & Sec.~\ref{sec:fs} \\
$|f|_\gamma$ & for $\gamma = N+\alpha, \alpha 
\in (0,1]$: Lipschitz norm 
& Sec.~\ref{sec:fs} \\ 
$L_0(\abx)$ & (strongly) measurable $\abx$-valued random variables & Sec.~\ref{sec:ss} \\ 
$L_m(\abx)$ & moment space of $\abx$-valued random variables & Sec.~\ref{sec:ss} \\ 
$L_{m,n}(\abx)$ & mixed moment space of $\abx$-valued random variables & Eq.~\ref{Lmn_space} \\ 
$C(I;\abx)$ & continuous paths on interval $I$ & --- \\ 
$C_2^{\alpha}(I;\abx)$ & $\alpha$-Hölder type $2$-parameter functions & Sec. \ref{sec:rps} \\ 
$C^{\alpha}(I;\abx)$ & $\alpha$-Hölder paths $Y:I\to\abx$, seminorm $|\delta Y|_\alpha$ & Sec. \ref{sec:rps} \\ 
$\mathscr{C}^\alpha ( \mathscr{C}^{0,\alpha}_g) (I,V)$ & space of (geometric) $\alpha$-Hölder rough paths over $V$ & Def. \ref{def.RP} \\ 
$C_2L_{m,n}(I,\W;\abx)$ & space of $\abx$-valued two-parameter processes 
& Def. \ref{def.C2mn} \\ 
$C^{\kappa}L_{m,n}(I,\W;\abx)$ & processes $Y$ with $Y_t\in C(I;L_m(\abx))$, $\delta Y\in C_2^\kappa L_{m,n}$ & Def. \ref{def.2amn} \\
$C_2^{\kappa}L_{m,n}(I,\W;\abx)$ & $2$-parameter stochastic process space & Def. \ref{def.2amn} \\ 
$\C_b(V,W)$ & continuous bounded maps $f:V\to W$ 
& Sec. \ref{sec:fs} \\ 

$\C_b^\kappa(V,W)$ & Lipschitz space of functions $f:V\to W$ & Sec.~\ref{sec:fs} \vspace{0.3cm}  \\

$\D_X^{\beta,\beta'}L_{m,n}, \D_X^{2\beta}L_{m,n}$& stochastic ($X$-)controlled rough paths (s.c.r.p.) &
Def.~\ref{def.SCRP} \\
$(Z,Z')$ & generic (stochastic controlled) rough path & --- \\
$\bk{Z,Z'}_{X,\beta,\beta';m,n}$ & $	\|\delta Z\|_{\beta;m,n}  + \| \delta Z'\|_{\beta';m,n} + \|\E_{\bigcdot} R^Z\|_{\beta+\beta';n}$ &  Eq.~\ref{def.bracketDL} \\ 
$\|Z,Z'\|_{X,\beta,\beta';m,n}$ & $\bk{(Z,Z')}_{X,\beta,\beta';m,n}  + \| Z' \|_{\infty ; n} $, seminorm on $\D_X^{\beta,\beta'}L_{m,n}$ & Eq.~\ref{def.normDbar} \\ 
$\bk{-;-} _{X,\bar X;\beta,\beta';m,n}$
		&  \hspace{-0.5cm} $\|\delta Z- \delta \bar Z\|_{\beta;m,n}
	+\|\delta Z'- \delta \bar Z'\|_{\beta';m,n}+\|\E_\bigcdot R^{Z}-\E_\bigcdot \bar R^{\bar Z}\|_{\beta+\beta';n}$, & Eq.~\ref{def.scrp.bracket} \\

$\|-;-\|_{X,\beta,\beta';m,n}$ & \hspace{-0.5cm} $ \bk{Z,Z';\bar Z,\bar Z'}_{X,\bar X;\beta,\beta';m,n} +  \| Z' - \bar Z' \|_{\infty ; n}$, distance on $\sim$
& Eq.~\ref{def.scrp.metric} \vspace{0.3cm}  \\

$\D_X^{\beta,\beta'}L_{m,n}\C^\gamma_b$ & space of stochastic controlled vector fields & Def. \ref{def.scvec} \\
$(f,f')$ & generic stochastic controlled vector field (s.c.v.f.) &  --- \\
$[(f,f')]_{\gamma;n}$ & $\sup_{s}(\|[f_s]_\gamma\|_n+\||f'_s|_{\gamma-1}\|_n)$ & Eq. \ref{def.norms_scvf} \\ 
$\|(f,f')\|_{\gamma;n}$ &
$\sup_{s}(\| |f_s|_\gamma\|_n+\||f'_s|_{\gamma-1}\|_n)$ & Eq. \ref{def.norms_scvf} \\
$\bk{(f,f')}_{X;\beta,\beta';m,n}$ & $\bk{\delta f}_{\beta;m,n}+\bk{\delta Df}_{\beta';m,n}+\bk{\delta f'}_{\beta';m,n}+\bk{\E_\bigcdot R^f}_{\beta+\beta';n}$, 
seminorm on $\D_X^{\beta,\beta'}L_{m,n}\C^\gamma_b$  & Eq. \ref{def.norms_scvf} \\
\end{longtable}

\bibliography{processes_martingale}
\end{document}